\documentclass{article}

\usepackage{amsmath}
\usepackage{amsfonts}
\usepackage{amsthm}
\usepackage{mathtools}

\usepackage{booktabs}
\usepackage{arydshln}

\usepackage{orcidlink}

\usepackage{hyperref}
\hypersetup{
    pdftitle={Be greedy and learn: efficient and certified algorithms for parametrized optimal control problems},
    pdfsubject={Parametrized optimal control problems and their efficient solution using greedy algorithms and machine learning},
    pdfauthor={Hendrik Kleikamp, Martin Lazar, Cesare Molinari},
    pdfkeywords={Parametrized optimal control, Greedy algorithm, Machine learning, Deep neural networks, Kernel methods, Error estimation}
}

\usepackage{xurl}
\hypersetup{breaklinks=true}

\usepackage{cleveref}

\usepackage{authblk}

\usepackage{dcolumn}
\newcolumntype{d}[1]{D{.}{.}{#1}}
\newcommand\mc[1]{\multicolumn{1}{c}{#1}}

\usepackage{tikz}
\usepackage{pgfplots}
\pgfplotsset{compat=1.18}
\usepgfplotslibrary{statistics}
\usetikzlibrary{positioning}

\pgfplotsset{
    legend image with text/.style={
        legend image code/.code={%
            \node[anchor=center] at (0.3cm,0cm) {#1};
        }
    },
}

\newenvironment{customlegend}[1][]{%
    \begingroup
    \csname pgfplots@init@cleared@structures\endcsname
    \pgfplotsset{#1}%
}{%
    \csname pgfplots@createlegend\endcsname
    \endgroup
}%
\def\addlegendimage{\csname pgfplots@addlegendimage\endcsname}

\newcommand{\addlegendimageintext}[1]{%
    \tikz {
        \begin{customlegend}[
            legend entries={\empty},
            legend style={
                draw=none,
                inner sep=0pt,
                column sep=0pt,
                nodes={inner sep=0pt}}]
        \addlegendimage{#1}
        \end{customlegend}
    }%
}

\usepackage[noEnd=true]{algpseudocodex}
\usepackage{algorithm}

\algnewcommand{\LineComment}[1]{\State \(\triangleright\) #1}

\renewcommand{\Return}[1]{\State{\textbf{return} #1}}

\usepackage{geometry}
\usepackage[maxbibnames=99]{biblatex}
\makeatletter
\renewcommand*\env@matrix[1][*\c@MaxMatrixCols c]{%
	\hskip -\arraycolsep
	\let\@ifnextchar\new@ifnextchar
	\array{#1}}
\makeatother

\definecolor{color0}{rgb}{0.65,0,0.15}
\definecolor{color1}{rgb}{0.84,0.19,0.15}
\definecolor{color2}{rgb}{0.96,0.43,0.26}
\definecolor{color3}{rgb}{0.99,0.68,0.38}
\definecolor{color4}{rgb}{1,0.88,0.56}
\definecolor{color5}{rgb}{0.67,0.85,0.91}
\definecolor{color6}{rgb}{0.27,0.46,0.71}

\renewcommand{\d}[1]{\,\text{d}#1}

\newcommand{\argmax}{\operatorname{argmax}}
\newcommand{\argmin}{\operatorname{argmin}}

\newcommand{\X}{X}
\newcommand{\U}{U}

\newcommand{\setR}{\mathbb{R}}
\newcommand{\setN}{\mathbb{N}}

\newcommand{\norm}[1]{\left\lVert#1\right\rVert}
\newcommand{\Span}[1]{\operatorname{span}\left(#1\right)}

\newcommand{\Gramian}{\Lambda_\mu^R}

\newcommand{\params}{\mathcal{P}}
\newcommand{\paramstrain}{\mathcal{P}_\mathrm{train}}

\newcommand{\redSpaceTransformed}{Y_\mu}

\newcommand{\RBROM}{G-ROM}
\newcommand{\DNNROM}{DNN-ROM}
\newcommand{\VKOGAROM}{VKOGA-ROM}
\newcommand{\GPRROM}{GPR-ROM}

\newtheorem{theorem}{Theorem}
\newtheorem{lemma}{Lemma}

\theoremstyle{remark}
\newtheorem{remark}{Remark}

\newtheorem{assumption}{Assumption}
\crefname{assumption}{Assumption}{Assumptions}

\begin{document}
    \title{Be greedy and learn: efficient and certified algorithms for parametrized optimal control problems}
    \author[1]{Hendrik Kleikamp\,\orcidlink{0000-0003-1264-5941}\,\thanks{Corresponding author, {\tt hendrik.kleikamp@uni-muenster.de}}}
    \author[2]{Martin Lazar\,\orcidlink{0000-0002-4034-5770}\,}
    \author[3]{Cesare Molinari\,\orcidlink{0000-0003-0864-5682}\,}
    \affil[1]{Institute for Analysis and Numerics, Mathematics Münster, University of Münster, Einsteinstrasse 62, 48149 Münster, Germany, {\tt hendrik.kleikamp@uni-muenster.de}}
    \affil[2]{Department of Electrical Engineering and Computing, University of Dubrovnik, \'Cira~Cari\'ca~4, 20 000 Dub\-rov\-nik, Croatia, {\tt mlazar@unidu.hr} }
    \affil[3]{Dipartimento di Matematica, University of Genova, Via Dodecaneso 35, 16146, Genoa, Italy, {\tt molinari@dima.unige.it}}
    \date{28.07.2023}

	\maketitle

    \begin{abstract}
        \noindent 
        We consider parametrized linear-quadratic optimal control problems and provide their online-efficient solutions by combining greedy reduced basis methods and machine learning algorithms. To this end, we first extend the greedy control algorithm, which builds a reduced basis for the manifold of optimal final time adjoint states, to the setting where the objective functional consists of a penalty term measuring the deviation from a desired state and a term describing the control energy. Afterwards, we apply machine learning surrogates to accelerate the online evaluation of the reduced model. The error estimates proven for the greedy procedure are further transferred to the machine learning models and thus allow for efficient a posteriori error certification. We discuss the computational costs of all considered methods in detail and show by means of two numerical examples the tremendous potential of the proposed methodology.
    \end{abstract}

    \section{Introduction}
In this work, we are concerned with a family of parameter-dependent optimal control problems where the state equations are given as a linear, time-invariant, infinite-dimensional system. The objective functional is a quadratic functional that measures the deviation from a desired final state and an energy of the control. For a given parameter the problem consists of finding a control that minimizes the objective functional under the constraint of the state equation.
\par
In this setting, one is typically interested in solving the optimal control problem several times for different parameters -- either in a multi-query or real-time context. In either case, solving the exact optimal control problem for lots of parameters is prohibitively costly, in particular if the dimension of the state space is large. We therefore aim to develop a reduced order model that is built during a (potentially costly) offline phase by computing the exact optimal control only for few, carefully selected parameters. Afterwards, the reduced model can be explored efficiently during the online phase for previously unconsidered parameters.
\par
The reduced basis algorithms were successfully developed and applied to parameter-dependent control problems during the last decade (e.g.~\cite{ballarin2022spacetime,dede2012reduced,lazar2016greedy}), with the offline phase mainly exploring proper orthogonal decomposition (POD) or a greedy sampling procedure or their combination (POD-greedy) in case of time-dependent reduced basis methods~\cite{hesthaven2016certified}. However, the cost of the corresponding online phase might still appear high, as the computation of the projection (of the solution to a reduced basis) relies on the full-order systems. An alternative approach would employ novel numerical tools that can handle high-dimensional problems and face the curse of dimensionality.
\par
To this end we combine model reduction and machine learning techniques. Although recently many papers appeared in which these two methods are combined, see for instance~\cite{daniel2020model}, the results are rather scarce when it comes to control systems~\cite{lazar2022control}. In~\cite{dalsanto2020data,hesthaven2018nonintrusive}, non-intrusive reduced basis methods that rely on neural networks were successfully applied for computing solutions to parametric PDEs. We want to develop a similar approach with the aim of efficiently treating parameter-dependent control problems. The idea is to train neural networks (or some other machine learning tool) to accurately predict the coefficients of solutions in a reduced basis, with a computational effort that is independent of the dimension of the full-order model.  The training is performed in the offline phase with a negligible cost, since the required data are generated by the greedy algorithm itself anyway, irrespective of a possible application of machine learning tools. As we shall see, such an approach will enable a significant computational speedup of the online phase, not only when compared to  computation of the exact solutions from scratch, but also in comparison with standard reduced basis approaches, which fully justifies its development.
\par
The paper is organized as follows: We first introduce the problem setting for parametrized linear-quadratic optimal control problems and derive the corresponding optimality system in~\Cref{sec:parametrized-optimal-control-problems}. Afterwards, in~\Cref{sec:reduced-order-modeling-greedy}, we describe a greedy algorithm to construct a reduced order model for parametrized optimal control problems. Additionally, we prove that the presented greedy algorithm is a weak greedy algorithm and derive a priori and a posteriori error estimates. In~\Cref{sec:reduced-order-machine-learning} we introduce a machine learning based approach to accelerate the online computation of approximate optimal controls. We further describe three different machine learning methods that are applied in our numerical experiments: deep neural networks, kernel methods and Gaussian process regression. The subsequent~\Cref{subsec:computational-costs} discusses in detail the required computational costs for the offline and the online stages of all described algorithms. In~\Cref{sec:numerical-experiments} we show the potential of our proposed algorithms on two examples coming from the field of optimal control of partial differential equations (PDEs). The paper concludes in~\Cref{sec:conclusion-outlook} with a discussion of the theoretical and practical results and an outlook to future research related to this contribution.

	\section{Parametrized Optimal Control Problems}\label{sec:parametrized-optimal-control-problems}
We are interested in parametrized linear optimal control problems where the objective functional consists of a term that penalizes the deviation from a target state at final time and a term measuring the energy of the applied control. The state equation serves as a constraint and is given in form of a linear time-invariant system with linear control. The parameter enters the operators governing the state equation as well as the initial conditions and the target state. In this section, we first give a detailed definition of the problem we consider. Afterwards, the associated optimality system is introduced and a linear system of equations for the optimal final time adjoint state is derived.

\subsection{Notation and Problem Definition}
Let~$\X$ and~$\U$ be real Hilbert spaces with scalar products~$\langle \cdot, \cdot \rangle_{\X}$ and~$\langle \cdot, \cdot \rangle_{\U}$ as well as associated norms~$\norm{\cdot}_{\X}$ and~$\norm{\cdot}_{\U}$, respectively. We omit the index in the scalar product and in the corresponding norm when the spaces are clear from the context. Let~$\mathcal{L}(\U,\X)$ denote, for instance, the set of linear and bounded operators from~$\U$ to~$\X$. In the following, we will refer to~$\X$ as the state and to~$\U$ as the control space. We consider parametrized linear control systems of the form
\begin{equation}\label{equ:parametrized-control-system}
    \begin{aligned}
          \dot{x}_\mu(t) &= A_\mu x_\mu(t) + B_\mu u_\mu(t),\qquad t\in[0,T], \\
          x_\mu(0) &= x_\mu^0,
    \end{aligned}
\end{equation}
where~$\mu\in\params$ denotes the parameter from a compact subset~$\params$ of some Banach space (that can also be infinite-dimensional), $A_\mu\in\mathcal{L}(\X,\X)$ and~$B_\mu\in\mathcal{L}(\U,\X)$ are parameter-dependent operators, $x_\mu^0\in \X$ is the parameter-dependent initial state, $x_\mu\colon[0,T]\to \X$ denotes the state trajectory, $u_\mu\colon[0,T]\to\U$ is the control, and~$T>0$ is the final time. Below, we consider the problem of steering the control system in~\eqref{equ:parametrized-control-system} close to a given (potentially parameter-dependent) target state~$x_\mu^T\in\X$. The following (natural) assumption is supposed to hold throughout the rest of the paper:
\begin{assumption}[Lipschitz continuity of parameter to system components maps]\label{as:continuity-parameter-to-system-matrices}
    The subset~$\params$ is compact and the mappings~$\params\ni\mu\mapsto A_\mu\in\mathcal{L}(\X,\X)$ and~$\params\ni\mu\mapsto B_\mu\in\mathcal{L}(\U,\X)$ from parameter to system matrices are Lipschitz continuous. In addition, we assume that the mappings~$\params\ni\mu\mapsto x_\mu^0\in\X$ and~$\params\ni\mu\mapsto x_\mu^T\in\X$ from parameter to initial and target state are Lipschitz continuous.
\end{assumption}
Furthermore, we introduce certain function spaces for the controls and states: As mentioned above, we denote by~$\U$ the space of admissible controls at fixed times. The associated space of time-dependent controls is given as~$G\coloneqq L^2([0,T];\U)$. Similarly, we have defined the state space~$\X$. We also define the space of time-dependent states~$H\coloneqq L^2([0,T];\X)\cap C^1([0,T];\X)$.
\par
Given a parameter~$\mu\in\params$, we are interested in finding a control~$u_\mu^*\in G$ that minimizes the functional~$\mathcal{J}_\mu\colon G\to\setR$, defined for a control~$u\in G$ as
\begin{align}\label{equ:cost-functional}
        \mathcal{J}_\mu(u) \coloneqq \frac{1}{2}\left[\langle x_\mu(T)-x_\mu^T,M\left(x_\mu(T)-x_\mu^T\right)\rangle + \int\limits_0^T \langle u(t),Ru(t)\rangle \d{t}\right],
\end{align}
where~$M\in\mathcal{L}(\X,\X)$ and~$R\in\mathcal{L}(\U,\U)$ satisfy the following assumptions.
\begin{assumption}[Properties of the weighting operators]\label{as:weightingops}
The operator $M\in\mathcal{L}(\X,\X)$ is self-adjoint and positive-semidefinite, while~$R\in\mathcal{L}(\U,\U)$ is self-adjoint and strictly positive-definite, meaning that~$R\geq\alpha I$ for some~$\alpha>0$.
\end{assumption}
The strict positive-definiteness of~$R$ implies in particular that~$R$ is invertible and that~$\mathcal{J}_\mu$ is strongly convex with respect to the control~$u$ and thus possesses a unique minimizer. In addition, $x_\mu\in H$ is the solution of~\cref{equ:parametrized-control-system} associated to the control~$u\in G$, and~$x_\mu^T\in \X$ is the target state. For notational simplicity, we omit stating explicitly the dependence of the state trajectory~$x_\mu$ on the control~$u$. The first term in the functional~$\mathcal{J}$ penalizes a deviation of the state at the final time~$T$ from the target state~$x_\mu^T$, where the deviation in different components can be weighted by the operator~$M$. The second term measures the energy of the control with respect to the weighting operator~$R$. We can summarize our parametrized, linear-quadratic optimal control problem as follows:
\begin{align}\label{equ:parametrized-optimal-control-problem}
    \min\limits_{u\in \U} \mathcal{J}_\mu(u),\qquad\text{subject to } \dot{x}_\mu(t)=A_\mu x_\mu(t)+B_\mu u(t) \text{ for } t\in[0,T],\quad x_\mu(0)=x_\mu^0.
\end{align}
Under the given integrability assumptions on the control, the state equation has a unique solution by Ca\-ra\-théo\-do\-ry's existence theorem, see for instance Chapter~I.4. in~\cite{hale1980ordinary}. As already stated above, the objective functional~$\mathcal{J}_\mu$ is strongly convex with respect to the control~$u$. Moreover, it is quadratic and so lower-semicontinuous. Hence, the optimal control problem in~\cref{equ:parametrized-optimal-control-problem} is well-posed and has a unique solution, see for instance Corollary~2.20 in~\cite{peypouquet2015convex}.

\begin{remark}[Parameter-dependent weighting matrices]
    It is also possible to consider parameter-dependent weighting operators~$M_\mu\in\mathcal{L}(\X,\X)$ and~$R_\mu\in\mathcal{L}(\U,\U)$ that change for different parameters~$\mu\in\params$. Under the assumption that the maps~$\params\ni\mu\mapsto M_\mu\in\mathcal{L}(\X,\X)$ and~$\params\ni\mu\mapsto R_\mu\in\mathcal{L}(\U,\U)$ are Lipschitz continuous, the theory and the algorithms developed below would not change, one solely has to replace~$M$ by~$M_\mu$ and~$R$ by~$R_\mu$, respectively. For notational simplicity and since we do not consider this case in our numerical experiments, we stick to the setting of parameter-independent operators~$M$ and~$R$. In the case of parameter-independent self-adjoint and positive-definite weighting operators, the weighting operators can also be interpreted as the introduction of an additional scalar product on the spaces~$\X$ and~$\U$, respectively. These scalar products are equivalent to the standard Euclidean scalar product if the spaces~$\X$ and~$\U$ are finite-dimensional.
\end{remark}
In the subsequent section, we present and discuss the optimality system associated to the optimal control problem~\eqref{equ:parametrized-optimal-control-problem}. As we will see, solving this optimality system can become costly already for a single parameter in case of moderate to large scale systems.

\subsection{Linear-quadratic Optimal Control and the Optimality System}\label{subsec:optimality-system}
By means of methods from the calculus of variations, one can derive the following theorem that characterizes the optimal control by considering the adjoint system.
\begin{theorem}[Optimality system for the optimal control problem]\label{thm:optimality-system}
    Let~$\mu\in\params$ be a parameter, $u_\mu^*\in G$ an optimal control, i.e.~a solution of \eqref{equ:parametrized-optimal-control-problem}, and denote by~$x_\mu^*\in H$ the associated state trajectory, i.e.~the solution of~\eqref{equ:parametrized-control-system} for the control~$u_\mu^*$. Then there exists an adjoint solution~$\varphi_\mu^*\in H$, such that the linear boundary value problem
    \begin{subequations}\label{equ:optimality-system-main}
        \begin{equation}\label{equ:optimality-system-odes}
            \begin{aligned}
                \dot{x}_\mu(t) &= A_\mu x_\mu(t)+B_\mu u_\mu(t), \\
                -\dot{\varphi}_\mu(t) &= A_\mu^* \varphi_\mu(t), \\
                u_\mu(t) &= -R^{-1}B_\mu^* \varphi_\mu(t),
            \end{aligned}
        \end{equation}
        for~$t\in[0,T]$ with initial respectively terminal conditions
        \begin{align}\label{equ:optimality-system-boundary-conditions}
            x_\mu(0) = x_\mu^0,\qquad \varphi_\mu(T)=M\left(x_\mu(T)-x_\mu^T\right),
        \end{align}
    \end{subequations}
    is solved by~$x_\mu=x_\mu^*$, $\varphi_\mu=\varphi_\mu^*$ and~$u_\mu=u_\mu^*$. In~\cref{equ:optimality-system-odes}, $A_\mu^*\in\mathcal{L}(\X,\X)$ and~$B_\mu^*\in\mathcal{L}(\X,\U)$ denote the adjoint operators of~$A_\mu$ and~$B_\mu$, respectively.
\end{theorem}
\begin{proof}
    See~\Cref{app:proof-optimality-system}.
\end{proof}

The optimality system in~\cref{equ:optimality-system-main} shows that~$u_\mu^*$, $x_\mu^*$ and~$\varphi_\mu^*$ are already uniquely determined by the optimal final time adjoint~$\varphi_\mu^*(T)$. To be more precise, we can explicitly compute~$u_\mu^*(t)$, $x_\mu^*(t)$ and~$\varphi_\mu^*(t)$ from~$\varphi_\mu^*(T)$ for~$t\in[0,T]$ using the following equations, see also~\Cref{fig:visualization-equations}:
\begin{align}\label{fromdualtoprimal}
    \varphi_\mu^*(t) &= e^{A_\mu^*(T-t)}\varphi_\mu^*(T), \\
    u_\mu^*(t) &= -R^{-1}B_\mu^*\varphi_\mu^*(t), \label{control} \\
    x_\mu^*(t) &= e^{A_\mu t}x_\mu^0 - \int\limits_0^t e^{A_\mu(t-s)}B_\mu R^{-1}B_\mu^* e^{A_\mu^*(T-s)}\varphi_\mu^*(T)\d{s}.\label{state}
\end{align}

\begin{figure}[H]
    \centering
    \begin{tikzpicture}[node distance=1.75cm]
        \node[draw] (xmu0) {$x_\mu^0=x_\mu(0)$};
        \node[draw, right=5cm of xmu0] (xmut) {$x_\mu(t):t\in[0,T]$};
        \node[draw, right=of xmut, fill=color5] (xmuT) {$x_\mu(T)\approx x_\mu^T$};

        \node[draw, below=2.5cm of xmuT, fill=color3] (varphimuT) {$\varphi_\mu(T)$};
        \node[draw, left=5.75cm of varphimuT] (varphimut) {$\varphi_\mu(t):t\in[0,T]$};

        \node[draw, above=0.9cm of varphimut] (umut) {$u_\mu(t):t\in[0,T]$};

        \draw[->, >=latex] (xmu0) -- node[midway, above] {$\dot{x}_\mu(t)=A_\mu x_\mu(t)+B_\mu u_\mu(t)$} (xmut);
        \draw[->, >=latex] (xmut) -- (xmuT);
        \draw[->, >=latex] (varphimuT) -- node[midway, above] {$-\dot{\varphi}_\mu(t)=A_\mu^*\varphi_\mu(t)$} (varphimut);
        \draw[->, >=latex] (varphimut) -- node[midway, left] {$u_\mu(t)=-R^{-1}B_\mu^*\varphi_\mu(t)$} (umut);
        \draw (umut.north) -- (umut.north |- xmut.west);

        \node (lc) at ($(xmu0.west |- varphimut.south) + (-1, -0.5)$) {};
        \node[label={[label distance=-5pt]45:{$t$}}] (rc) at ($(xmuT.east |- varphimut.south) + (1, -0.5)$) {};
        \draw[->, >=latex] (lc) -- (rc);
        \node at ($(xmu0.south |- lc) + (0, -0.5)$) {$0$};
        \node at ($(varphimuT.south |- lc) + (0, -0.5)$) {$T$};
        \draw ($(xmu0.south |- lc.east)$) -- ($(xmu0.south |- lc.east) + (0, -0.2)$);
        \draw ($(varphimuT.south |- lc.east)$) -- ($(varphimuT.south |- lc.east) + (0, -0.2)$);

        \node[anchor=east] at ($(xmu0) + (-1.5, 0)$) {State:};
        \node[anchor=east] at ($(xmu0 |- umut) + (-1.5, 0)$) {Control:};
        \node[anchor=east] at ($(xmu0 |- varphimut) + (-1.5, 0)$) {Adjoint:};
    \end{tikzpicture}
    \caption{Schematic visualization of the computations required to approach the final time state~$x_\mu(T)$ (marked in blue) from a final time adjoint~$\varphi_\mu(T)$ (marked in orange).}
    \label{fig:visualization-equations}
\end{figure}
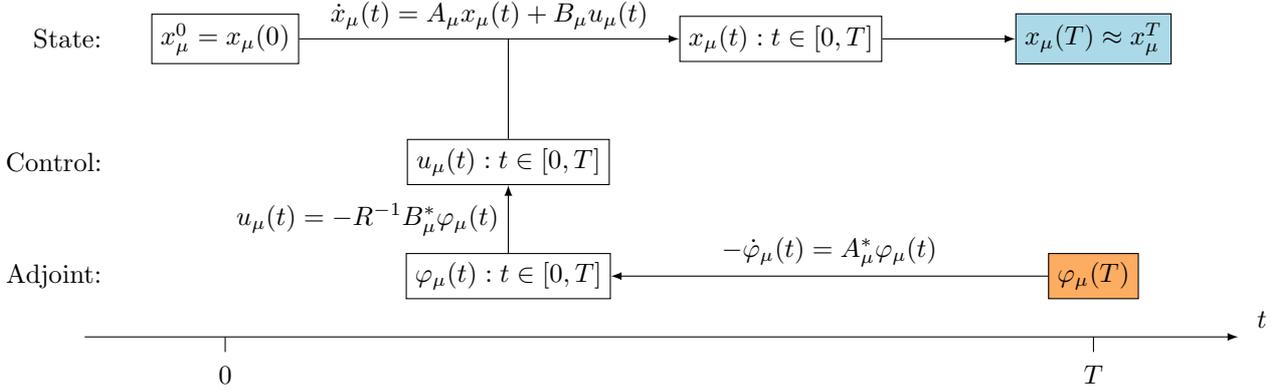

However, according to the boundary condition in~\cref{equ:optimality-system-boundary-conditions}, the optimal final time adjoint~$\varphi_\mu^*(T)$ is coupled to the optimal state~$x_\mu^*(T)$ at the final time~$T$, and hence, cannot be computed directly. To circumvent this issue, we compute~$x_\mu^*(T)$ in terms of~$\varphi_\mu^*(T)$ to obtain a linear equation for~$\varphi_\mu^*(T)$. It holds
\[
    x_\mu^*(T) = e^{A_\mu T}x_\mu^0 - \int\limits_0^T e^{A_\mu(T-s)}B_\mu R^{-1}B_\mu^* e^{A_\mu^*(T-s)}\d{s} \cdot \varphi_\mu^*(T) = e^{A_\mu T}x_\mu^0 - \Gramian \varphi_\mu^*(T),
\]
where the \emph{weighted controllability Gramian}~$\Gramian\in\mathcal{L}(\X,\X)$ is defined as
\begin{align}\label{equ:definition-gramian-matrix}
    \Gramian \coloneqq \int\limits_0^T e^{A_\mu(T-s)}B_\mu R^{-1}B_\mu^* e^{A_\mu^*(T-s)}\d{s}.
\end{align}
The Gramian~$\Gramian$ is bounded due to the boundedness of all involved operators and the finite time horizon we consider. According to~\Cref{thm:optimality-system}, we have that
\[
    \varphi_\mu^*(T) = M\left(x_\mu^*(T)-x_\mu^T\right),
\]
and it further holds
\begin{align}\label{equ:state-gramian-adjoint-relation}
    x_\mu^*(T) = e^{A_\mu T}x_\mu^0 - \Gramian \varphi_\mu^*(T)
\end{align}
as shown above. Combining these two equations, we obtain
\[
    \varphi_\mu^*(T) = M\left(e^{A_\mu T}x_\mu^0 - x_\mu^T - \Gramian \varphi_\mu^*(T)\right).
\]
Rearranging the previous equation gives the following Lemma that states the resulting equation the optimal final time adjoint~$\varphi_\mu^*(T)\in \X$ has to solve:
\begin{lemma}[Linear system for the optimal final time adjoint]\label{lem:linear-system}
    Let~$\varphi_\mu^*(T)$ denote the optimal adjoint state at time~$T$ that determines the solution of the optimality system~\eqref{equ:optimality-system-main}. Then it holds
    \begin{align}\label{equ:linear-system-for-optimal-final-time-adjoint}
        \left(I+M\Gramian\right)\varphi_\mu^*(T) = M\left(e^{A_\mu T}x_\mu^0 - x_\mu^T\right),
    \end{align}
    where~$I\in\mathcal{L}(\X,\X)$ denotes the identity.
\end{lemma}

Since the operator~$R$ is positive-definite, the same holds for~$R^{-1}$. Therefore, the Gramian~$\Gramian$ is self-adjoint and positive-semidefinite. We also emphasize at this point that the controllability Gramian as defined here depends on the parameter~$\mu\in\params$ in the same way as the involved system operators do.
\par
We make the following assumption that is supposed to hold throughout the rest of the paper and will become useful for the error estimation in~\Cref{sec:reduced-order-modeling-greedy}:
\begin{assumption}[Positivity of the product of the weighting and the Gramian operator]\label{as:symmetric-product}
    We assume that the operator~$M\Gramian\in\mathcal{L}(\X,\X)$ is positive-semidefinite for all parameters~$\mu\in\params$.
\end{assumption}
As a simple example of a weighting operator~$M\in\mathcal{L}(\X,\X)$ that fulfills~\Cref{as:symmetric-product}, we will use in our numerical experiments in~\Cref{sec:numerical-experiments} the choice~$M=\kappa I$ for a suitable constant~$\kappa>0$, where~$I\in\mathcal{L}(\X,\X)$ denotes the identity. More generally, the assumption will be satisfied if~$M$ and the Gramian~$\Gramian$ commute. 
\begin{remark}[Computation of products with the Gramian]\label{rem:applying-gramian-to-vectors}
    We emphasize that for a given vector~$p\in \X$, the product~$\Gramian p$ can be computed without constructing~$\Gramian$ explicitly. Instead, one uses that~$-\Gramian p=x_\mu(T)$ (see~\cref{equ:state-gramian-adjoint-relation}), where~$x_\mu$ solves system~\eqref{equ:optimality-system-main} with boundary conditions~$x_\mu(0)=0$ and~$\varphi_\mu(T)=p$, respectively. It is therefore sufficient to solve~\eqref{equ:optimality-system-main} by first solving the equation for~$\varphi_\mu$ backward in time, then computing~$u_\mu$, and finally solving the state equation for~$x_\mu$ forward in time. Assembling~$\Gramian\in\mathcal{L}(\X,\X)$ (which depends on the parameter~$\mu\in\params$) would be prohibitively costly and is infeasible for moderate to large scale systems when having to do so for multiple parameters.
\end{remark}

Because of large computational costs required for solving the optimal control problem for a single parameter (see also the discussions in~\Cref{subsec:computational-costs}), we aim for building a reduced order model that replaces the equation in~\eqref{equ:linear-system-for-optimal-final-time-adjoint} by a linear system of small dimension that can be solved faster while still providing a sufficiently accurate approximation of the optimal solution.

    \section{Reduced Order Modeling by a Greedy Algorithm}\label{sec:reduced-order-modeling-greedy}
In the previous section, we derived a linear system for the optimal final time adjoint state~$\varphi_\mu^*(T)$ and observed that the optimal control is already completely determined by~$\varphi_\mu^*(T)$. The main goal of this section is to extend the greedy control algorithm introduced in~\cite{lazar2016greedy} to the setting of parametrized optimal control problems of the form~\eqref{equ:parametrized-optimal-control-problem}.

\subsection{Error Estimation by considering the Residual}\label{subsec:residual-based-error-estimator}
Given a parameter~$\mu\in\params$ and an arbitrary vector~$p\in\X$, we define the error estimator~$\eta_\mu\colon\X\to\setR$ as
\begin{align}\label{equ:definition-error-estimator}
    \eta_\mu(p) \coloneqq \norm{M\left(e^{A_\mu T}x_\mu^0 - x_\mu^T\right) - (I+M\Gramian)p},
\end{align}
where the term inside the norm corresponds to the residual of equation~\eqref{equ:linear-system-for-optimal-final-time-adjoint}.
\par
We obtain the following theorem which states that~$\eta_\mu$ is an efficient and reliable error estimator for the final time adjoint state, where we use in the proof that the operator~$M\Gramian$ is positive (see~\Cref{as:symmetric-product}).
\begin{theorem}[Efficient and reliable error estimator for the distance from the adjoint]\label{thm:error-estimator-adjoint}
    Let~$\mu\in\params$ be a parameter, $\varphi_\mu^*(T)\in\X$ the corresponding optimal final time adjoint solving equation~\eqref{equ:linear-system-for-optimal-final-time-adjoint}, and let~$p\in\X$ be an approximate final time adjoint. Then it holds
    \begin{align}\label{equ:error-estimator-bounds}
        \norm{\varphi_\mu^*(T)-p} \quad \leq \quad \eta_\mu(p) \quad \leq \quad \norm{I+M\Gramian}_{\mathcal{L}(\X,\X)} \ \norm{\varphi_\mu^*(T)-p}.
    \end{align}
\end{theorem}
\begin{proof}
    For the optimal final time adjoint ~$\varphi_\mu^*(T)\in\X$, according to~\Cref{lem:linear-system},  it holds
    \[
        M(e^{A_\mu T}x_\mu^0-x_\mu^T)=(I+M\Gramian)\varphi_\mu^*(T).
    \]
     Plugging this equation into the definition of the error estimator, we obtain
    \begin{align*}
        \eta_\mu(p)^2 &= \norm{(I+M\Gramian)(\varphi_\mu^*(T)-p)}^2 \\
        &=\norm{\varphi_\mu^*(T)-p)}^2+\norm{M\Gramian(\varphi_\mu^*(T)-p)}^2+2\langle M\Gramian(\varphi_\mu^*(T)-p),(\varphi_\mu^*(T)-p)\rangle \\
        &\geq \norm{\varphi_\mu^*(T)-p)}^2,
    \end{align*}
    where we used that $M\Gramian$ is positive. From here the lower bound in~\eqref{equ:error-estimator-bounds} follows directly. The upper bound comes simply from the definition of the operator norm.
\end{proof}
The error estimator~$\eta_\mu(p)$ can be used to quantify the quality of an approximation of the optimal final time adjoint~$\varphi_\mu^*(T)$ by an arbitrary~$p\in\X$ without ever computing~$\varphi_\mu^*(T)$. Instead, it is only necessary to solve the linear initial value problem~\eqref{equ:optimality-system-main} twice: first with~$x_\mu(0)=0$, $\varphi_\mu(T)=p$ (in order to compute~$-\Gramian p$) and then with~$x_\mu(0)=x_\mu^0$, $\varphi_\mu(T)=0$ (i.e.~without control of the state equation~\eqref{equ:parametrized-control-system}) in order to determine~$e^{A_\mu T}x_\mu^0$ (the term~$e^{A_\mu T}x_\mu^0$ corresponds to the free dynamics). The error estimator~$\eta_\mu$ can therefore be evaluated much more efficiently than solving the full optimal control problem.
\par
In the following section, we describe an algorithm that constructs a low-dimensional subspace of~$\X$ in which an approximation of the optimal final time adjoint is searched. The basis of this subspace is chosen by a greedy algorithm that uses the error estimator~$\eta_\mu$ to determine the next parameter and in particular the associated final time adjoint that is added to the basis.

\subsection{Greedy Procedure for constructing a Reduced Basis}\label{subsec:greedy-for-reduced-basis}
Let~$\varepsilon>0$ be a prescribed tolerance. We aim at constructing a reduced subspace~$\X^N\subset\X$ of low dimension~$\dim\X^N=N$, such that~$\X^N$ approximates the manifold
\begin{align}\label{equ:definition-solution-manifold}
    \mathcal{M} \coloneqq \{\varphi_\mu^*(T):\mu\in\params\} \subset\X
\end{align}
of all possible optimal final time adjoints up to a tolerance of~$\varepsilon$. Since~$\params$ is compact and the mapping~$\params\ni\mu\mapsto\varphi_\mu^*(T)\in\X$ is Lipschitz continuous (which follows from~\Cref{as:continuity-parameter-to-system-matrices} and the characterization of~$\varphi_\mu^*(T)$ as solution of the linear equation in~\eqref{equ:linear-system-for-optimal-final-time-adjoint}), the set~$\mathcal{M}$ is also compact. The distance~$\operatorname{dist}(Y,Z)$ of a subspace~$Y\subset\X$ to a compact set~$Z\subset\X$ is defined as
\[
    \operatorname{dist}(Y,Z) \coloneqq \sup\limits_{z\in Z}\inf\limits_{y\in Y} \norm{y-z}.
\]
We thus search for a reduced space~$\X^N\subset\X$ such that~$\operatorname{dist}(\X^N,\mathcal{M})\leq\varepsilon$. To this end, greedy algorithms construct a sequence of reduced bases~$\Phi^0,\Phi^1,\dots$, starting with~$\Phi^0=\emptyset\subset\X$, and corresponding reduced spaces~$\X^0=\Span{\Phi^0}=\{0\},\X^1=\Span{\Phi^1},\dots$ by successively adding new basis functions. These functions are selected as snapshots, i.e.~optimal final time adjoints for certain parameters. Assume a reduced basis~$\Phi^k=\{\varphi_1,\dots,\varphi_k\}\subset\X$ of dimension~$k$ is given with associated reduced space~$\X^k=\Span{\Phi^k}$. The next basis function~$\varphi_{k+1}\in\X$ is chosen as~$\varphi_{k+1}=\varphi_{\mu_{k+1}}^*(T)$, where the parameter~$\mu_{k+1}\in\params$ is determined such that the distance~$\operatorname{dist}(\X^k,\{\varphi_{\mu_{k+1}}^*(T)\})$ of~$\varphi_{\mu_{k+1}}^*(T)$ to~$\X^k$ (or a suitable estimate of this distance that is cheaper to compute) is the largest among all parameters~$\mu\in\params$, i.e.~it holds
\[
    \mu_{k+1} \in \argmax\limits_{\mu\in\params}\,\operatorname{dist}(\X^k,\{\varphi_\mu^*(T)\}).
\]
After choosing the parameter~$\mu_{k+1}$, the optimal final time adjoint~$\varphi_{\mu_{k+1}}^*(T)$ is computed by solving~\cref{equ:linear-system-for-optimal-final-time-adjoint} and added to the reduced basis, i.e.~the new basis~$\Phi^{k+1}$ is defined as~$\Phi^{k+1}=\Phi^k\cup\{\varphi_{k+1}\}$ for~$\varphi_{k+1}=\varphi_{\mu_{k+1}}^*(T)$. After updating the reduced basis, the error for the parameter~$\mu_{k+1}$ vanishes, since the optimal final time adjoint for~$\mu_{k+1}$ is now contained in the reduced space~$\X^{k+1}$, i.e.~$\varphi_{\mu_{k+1}}^*(T)\in\Phi^{k+1}\subset\X^{k+1}$.
\par
To make the construction of~$\Phi^N$ computationally feasible one usually restricts the search for a new parameter to a finite set~$\paramstrain\subset\params$ of~$n_\mathrm{train}\coloneqq\lvert\paramstrain\rvert<\infty$ training parameters. This set~$\paramstrain$ of training parameters is chosen to be dense in the overall set~$\params$ of parameters (in the sense that the union of balls with certain radius around the training parameters covers the whole parameter set), see also the pseudocode of the procedure provided in~\Cref{alg:offline-greedy} below. Furthermore, instead of computing the true distance~$\operatorname{dist}(\X^k,\{\varphi_\mu^*(T)\})$ of a subspace~$\X^k$ to~$\varphi_\mu^*(T)$ for all training parameters~$\mu\in\paramstrain$, an estimate of this distance is used that does not require the computation of~$\varphi_\mu^*(T)$. In our setting, we will estimate the distance using the error estimator from~\cref{equ:definition-error-estimator} and a suitably chosen approximate final time adjoint.
\par
Once having constructed a reduced basis~$\Phi^k$, it is used to obtain an approximate final time adjoint for an arbitrary parameter~$\mu\in\params$. To this end, we first compute~$x_i^\mu=(I+M\Gramian)\varphi_i\in\X$ for all~$i=1,\dots,k$. The state~$-\Gramian\varphi_i\in\X$ for~$i\in\{1,\dots,k\}$ corresponds to the final time state~$x_\mu(T)$ of the system~\eqref{equ:optimality-system-odes} with zero initial datum and the control determined by the basis function~$\varphi_i\in\Phi^k$. Thus, $x_i^\mu$ can be seen as  the perturbation of the final state~$-\Gramian\varphi_i\in\X$ by the operator~$-\left((\Gramian)^{-1} + M\right)$, i.e.~it holds~$-\left((\Gramian)^{-1} + M\right)\left(-\Gramian\varphi_i\right)=(I+M\Gramian)\varphi_i=x_i^\mu$. Hence, the space~$\redSpaceTransformed^k\coloneqq\Span{x_1^\mu,\dots,x_k^\mu}=(I+M\Gramian)\X^k\subset\X$ contains all possible perturbed final time states that can be generated from the final time adjoints in the reduced basis~$\Phi^k$ for the system determined by the parameter~$\mu$ and starting from zero initial conditions. This motivates, together with the linear equation for the optimal final time adjoint from~\Cref{lem:linear-system}, to consider the projection of~$M\left(e^{A_\mu T}x_\mu^0 - x_\mu^T\right)$ onto~$\redSpaceTransformed^k$ and therefore the distance~$\operatorname{dist}(\redSpaceTransformed^k,\{M\left(e^{A_\mu T}x_\mu^0 - x_\mu^T\right)\})$, see also~\Cref{fig:proof-error-estimator-reduced-space}. Let us denote by~$P_Y(x)\in Y$ the orthogonal projection of a vector~$x\in\X$ onto a subspace~$Y\subset\X$. Then it holds
\[
    \operatorname{dist}\left(\redSpaceTransformed^k,\left\{M\left(e^{A_\mu T}x_\mu^0 - x_\mu^T\right)\right\}\right) = \norm{M\left(e^{A_\mu T}x_\mu^0 - x_\mu^T\right) - P_{\redSpaceTransformed^k}\Big(M\left(e^{A_\mu T}x_\mu^0 - x_\mu^T\right)\Big)}.
\]
We can express the projection in the basis of~$\redSpaceTransformed^k$ given by the vectors~$x_1^\mu,\dots,x_k^\mu$ (which form a linearly independent set since the corresponding parameters~$\mu_1,\dots,\mu_k\in\paramstrain$ were selected by the greedy procedure). Let us write
\begin{align}\label{equ:projection-onto-parameter-dependent-state-space}
    P_{\redSpaceTransformed^k}\Big(M\left(e^{A_\mu T}x_\mu^0 - x_\mu^T\right)\Big) = \sum\limits_{i=1}^{k} \alpha_i^\mu x_i^\mu \in \redSpaceTransformed^k
\end{align}
for some coefficients~$\alpha_1^\mu,\dots,\alpha_k^\mu\in\setR$. Then we choose the approximate final time adjoint~$\tilde{\varphi}_\mu^k\in\X^k$ as
\begin{align}\label{equ:definition-approximate-adjoint}
    \tilde{\varphi}_\mu^k = \sum\limits_{i=1}^{k} \alpha_i^\mu \varphi_i.
\end{align}
With these definitions, we have that
\[
    (I+M\Gramian)\tilde{\varphi}_\mu^k = \sum\limits_{i=1}^{k} \alpha_i^\mu x_i^\mu = P_{\redSpaceTransformed^k}\Big(M\left(e^{A_\mu T}x_\mu^0 - x_\mu^T\right)\Big),
\]
and in particular
\[
    \eta_\mu(\tilde{\varphi}_\mu^k)=\operatorname{dist}\left(\redSpaceTransformed^k,\left\{M\left(e^{A_\mu T}x_\mu^0 - x_\mu^T\right)\right\}\right).
\]
This allows us to estimate the efficiency of the reduced space~$\X^k$ through the error estimator~$\eta_\mu$, i.e.
\[
    \operatorname{dist}\left(\X^k,\left\{\varphi_\mu^*(T)\right\}\right) \approx \eta_\mu(\tilde{\varphi}_\mu^k).
\]
More precisely, the following theorem holds.
\begin{theorem}[Efficient and reliable error estimator for a reduced space]\label{thm:error-estimator-reduced-space}
    Let~$\mu\in\params$ be a parameter and~$\varphi_\mu^*(T)\in\X$ the optimal final time adjoint solving the optimality system in~\cref{equ:optimality-system-main}. Then, for the error estimator~$\eta_\mu(\tilde{\varphi}_\mu^k)$ with~$\eta_\mu$ introduced in~\cref{equ:definition-error-estimator} and~$\tilde{\varphi}_\mu^k$ from~\cref{equ:definition-approximate-adjoint}, it holds
    \begin{equation}\label{equ:RS-error}
        \operatorname{dist}(\X^k,\{\varphi_\mu^*(T)\}) \quad \leq \quad \eta_\mu(\tilde{\varphi}_\mu^k) \quad \leq \quad \norm{I+M\Gramian}_{\mathcal{L}(\X,\X)}\cdot\operatorname{dist}(\X^k,\{\varphi_\mu^*(T)\}).
    \end{equation}
    
\end{theorem}
\begin{proof}
    Due to orthogonality of the projection and the Pythagorean theorem, it holds
    \[
        \norm{\varphi_\mu^*(T)-\tilde{\varphi}_\mu^k}^2 = \norm{\varphi_\mu^*(T)-P_{\X^k}(\varphi_\mu^*(T))}^2 + \norm{\tilde{\varphi}_\mu^k-P_{\X^k}(\varphi_\mu^*(T))}^2
    \]
    and therefore
    \begin{align*}
        \operatorname{dist}(\X^k,\{\varphi_\mu^*(T)\})^2 &= \norm{\varphi_\mu^*(T)-P_{\X^k}(\varphi_\mu^*(T))}^2 \\
        &= \norm{\varphi_\mu^*(T)-\tilde{\varphi}_\mu^k}^2 - \norm{\tilde{\varphi}_\mu^k-P_{\X^k}(\varphi_\mu^*(T))}^2 \\
        &\leq \eta_\mu(\tilde{\varphi}_\mu^k)^2
    \end{align*}
    by~\Cref{thm:error-estimator-adjoint}. This implies the lower bound in~\cref{equ:RS-error}. 
    \par
    In order to obtain the upper one, we estimate
    \begin{align*}
        \eta_\mu(\tilde{\varphi}_\mu^k) &= \norm{M\left(e^{A_\mu T}x_\mu^0 - x_\mu^T\right) - (I+M\Gramian)\tilde{\varphi}_\mu^k} \\
        &= \norm{(I+M\Gramian)\varphi_\mu^*(T)-P_{\redSpaceTransformed^k}\Big((I+M\Gramian)\varphi_\mu^*(T)\Big)} \\
        &\leq \norm{(I+M\Gramian)\varphi_\mu^*(T)-(I+M\Gramian)P_{\X^k}(\varphi_\mu^*(T))} \\
        &\leq \norm{I+M\Gramian}_{\mathcal{L}(\X,\X)}\norm{\varphi_\mu^*(T)-P_{\X^k}(\varphi_\mu^*(T))} \\
        &= \norm{I+M\Gramian}_{\mathcal{L}(\X,\X)}\operatorname{dist}(\X^k,\{\varphi_\mu^*(T)\}).
    \end{align*}
    Here we used that~$(I+M\Gramian)P_{\X^k}(\varphi_\mu^*(T))\in\redSpaceTransformed^k$, which follows from the fact that~$P_{\X^k}(\varphi_\mu^*(T))\in\X^k$ and~$\redSpaceTransformed^k = (I+M\Gramian)\X^k$.
\end{proof}

For a visualization of the quantities involved in the proof of~\Cref{thm:error-estimator-reduced-space} and their relationships, see~\Cref{fig:proof-error-estimator-reduced-space}.
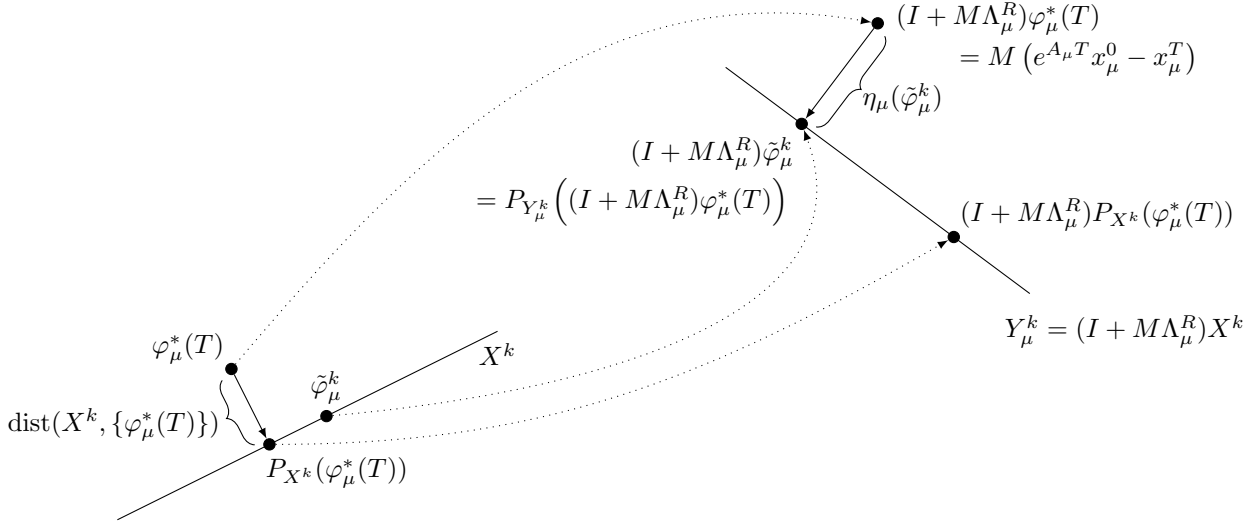
\begin{figure}[H]
    \centering
    \begin{tikzpicture}
		\draw (0,0) -- (5,2.5);
		\node () at (5,2.2) {$\X^k$};

		\node[draw, circle, fill, inner sep=1.5pt, label={[label distance=-1pt]90:{$\tilde{\varphi}_\mu^k$}}] (tilde-varphi-mu-k) at (2.75, 1.375) {};
		\node[draw, circle, fill, inner sep=1.5pt, label={[label distance=-1pt, xshift=-5pt]273:{$P_{\X^k}(\varphi_\mu^*(T))$}}] (P-Xk-varphi-mu-star) at (2, 1) {};
		\node[draw, circle, fill, inner sep=1.5pt, label={[label distance=-4pt]115:{$\varphi_\mu^*(T)$}}] (varphi-mu-star) at (1.5, 2) {};
		\draw[->, >=latex] (varphi-mu-star) -- (P-Xk-varphi-mu-star);

		\draw (8,6) -- (12,3);
		\node () at (13.25,2.5) {$\redSpaceTransformed^k = (I+M\Gramian)\X^k$};

		\node[draw, circle, fill, inner sep=1.5pt, label={[label distance=-4pt, xshift=0pt, yshift=-3pt]245:{\shortstack{\qquad\qquad\qquad$(I+M\Gramian)\tilde{\varphi}_\mu^k$\\$=P_{\redSpaceTransformed^k}\Big((I+M\Gramian)\varphi_\mu^*(T)\Big)$}}}] (I-tilde-varphi-mu-k) at (9, 5.25) {};
		\node[draw, circle, fill, inner sep=1.5pt, label={[label distance=-4pt]45:{$(I+M\Gramian)P_{\X^k}(\varphi_\mu^*(T))$}}] (I-P-Xk-varphi-mu-star) at (11, 3.75) {};
		\node[draw, circle, fill, inner sep=1.5pt, label={[label distance=-18pt, xshift=-22pt, yshift=-12pt]45:{\shortstack{$(I+M\Gramian)\varphi_\mu^*(T)$\\\qquad\qquad\qquad$=M\left(e^{A_\mu T}x_\mu^0 - x_\mu^T\right)$}}}] (I-varphi-mu-star) at (10, 6.58333) {};
		\draw[->, >=latex] (I-varphi-mu-star) -- (I-tilde-varphi-mu-k);

		\draw[->, >=latex, dotted] (tilde-varphi-mu-k) to[out=5, in=-70] (I-tilde-varphi-mu-k);
		\draw[->, >=latex, dotted] (P-Xk-varphi-mu-star) to[out=0, in=-150] (I-P-Xk-varphi-mu-star);
		\draw[->, >=latex, dotted] (varphi-mu-star) to[out=50, in=170] (I-varphi-mu-star);

		\draw [decorate, decoration={brace, raise=5pt, amplitude=5pt}] (I-varphi-mu-star) -- (I-tilde-varphi-mu-k) node[midway, below right, xshift=5pt] {$\eta_\mu(\tilde{\varphi}_\mu^k)$};
		\draw [decorate, decoration={brace, raise=5pt, amplitude=5pt}] (P-Xk-varphi-mu-star) -- (varphi-mu-star) node[midway, below left, xshift=-8pt, yshift=3pt] {$\operatorname{dist}(\X^k,\{\varphi_\mu^*(T)\})$};
	\end{tikzpicture}
    \caption{Visualization of the different final time adjoints and states that occur in the proof of~\Cref{thm:error-estimator-reduced-space}. Solid arrows represent orthogonal projections onto the respective spaces~$\X^k$ and~$\redSpaceTransformed^k$ while dotted arrows visualize applications of~$I+M\Gramian$.}
    \label{fig:proof-error-estimator-reduced-space}
\end{figure}

\begin{remark}[On the efficiency bound in~\Cref{thm:error-estimator-reduced-space}]
    We emphasize at this point that the first bound in~\Cref{thm:error-estimator-reduced-space} holds for an arbitrary choice~$p\in\X^k$ for the approximate final time adjoint, i.e.~we have~$\operatorname{dist}(\X^k,\{\varphi_\mu^*(T)\})\leq\eta_\mu(p)$ for all~$p\in\X^k$. However, the efficiency of the error estimator~$\eta_\mu(\tilde{\varphi}_\mu^k)$, i.e.~the second inequality in~\Cref{thm:error-estimator-reduced-space}, heavily relies on the choice of the approximate final time adjoint~$\tilde{\varphi}_\mu^k$. The efficiency bound is required to maintain the convergence of the weak greedy algorithm, see~\Cref{subsec:error-analysis-greedy-algorithm}. 
\end{remark}

To formulate the greedy algorithm in pseudo-code, we observe that the following properties hold, which will allow for an error analysis of the greedy algorithm in~\Cref{subsec:error-analysis-greedy-algorithm}: The parameter to solution map, given as
\[
    \params\ni\mu\mapsto\varphi_\mu^*(T)=\left(I+M\Gramian\right)^{-1}M\left(e^{A_\mu T}x_\mu^0 - x_\mu^T\right)\in\X,
\]
is Lipschitz continuous (due to the Lipschitz continuity in~\Cref{as:continuity-parameter-to-system-matrices}, and since the function that maps an operator to its inverse is also Lipschitz continuous), as already mentioned above. The constant~$C_{\varphi^*}>0$ is defined as the Lipschitz constant of this mapping, i.e.~it holds
\[
    \norm{\varphi_\mu^*(T)-\varphi_{\tilde{\mu}}^*(T)} \leq C_{\varphi^*}\norm{\mu-\tilde{\mu}}\qquad\text{for all }\mu,\tilde{\mu}\in\params.
\]
Moreover, the norm~$\norm{I+M\Gramian}$ is uniformly bounded over the parameter domain. The constant~$C_\Lambda>0$ is defined as
\[
    C_\Lambda \coloneqq \sup_{\mu\in\params} \norm{I+M\Gramian}_{\mathcal{L}(\X,\X)} < \infty.
\]
Both of these statements are direct consequences of~\Cref{as:continuity-parameter-to-system-matrices} and the compactness of the parameter set~$\params$. We further remark that it holds~$C_\Lambda\geq 1$, which immediately follows from~\cref{equ:error-estimator-bounds}, and that~$C_{\varphi^*}\geq 0$. In addition to the constants above, we define~$0<\gamma\leq 1$ as
\begin{align}\label{equ:definition-greedy-constant-gamma}
    \gamma \coloneqq \frac{1}{C_{\varphi^*}+C_\Lambda}.
\end{align}
The constant~$\gamma$ will play a crucial role in the error analysis and in the performance of the greedy algorithm.
\par
The procedure of the greedy algorithm is summarized in~\Cref{alg:offline-greedy}.
\begin{algorithm}[H]
    \caption{Offline phase of the greedy procedure}\label{alg:offline-greedy}
    \begin{algorithmic}[1]
        \Require Tolerance~$\varepsilon>0$
        \Ensure Reduced basis~$\Phi^N\subset\X$, reduced space~$\X^N=\Span{\Phi^N}\subset\X$, training parameters and associated reduced coefficients collected in a set~$D_\mathrm{train}\subset\paramstrain\times\setR^N$
        \Procedure{Greedy}{$\varepsilon$}
        \State $\tilde{\varepsilon} \gets C_\Lambda\gamma\varepsilon$,\quad$\delta \gets \frac{\tilde{\varepsilon}}{C_\Lambda}$
        \State define a training set~$\paramstrain\subset\params$ such that for all~$\mu\in\params$ there exists~$\tilde{\mu}\in\paramstrain$ with~$\norm{\mu-\tilde{\mu}}\leq\delta$
        \State $k \gets 0$,\quad$\Phi^0 \gets \emptyset$,\quad$\X^0 \gets \{0\}\subset\X$\label{lst:start-greedy-without-tolerance}
        \State $\tilde{\varphi}_\mu^0 \gets 0$ for all~$\mu\in\paramstrain$
        \State select next parameter~$\mu_{1} \gets \argmax_{\mu\in\paramstrain} \eta_\mu(\tilde{\varphi}_\mu^0)$
        \While{$\eta_{\mu_{k+1}}(\tilde{\varphi}_{\mu_{k+1}}^k) > \tilde{\varepsilon}$}\label{lst:offline-termination-criterion}
            \State compute optimal final time adjoint~$\varphi_{k+1} \gets \varphi_{\mu_{k+1}}^*(T)$ by solving~\cref{equ:linear-system-for-optimal-final-time-adjoint}
            \State $\Phi^{k+1} \gets \Phi^k\cup\{\varphi_{k+1}\}$\qquad (apply orthonormalization using Gram-Schmidt algorithm if desired)
            \State $\X^{k+1} \gets \Span{\Phi^{k+1}}$
            \State $k \gets k+1$
            \For{$\mu\in\paramstrain$}
                \State compute~$x_i^\mu \gets (I+M\Gramian)\varphi_i$ for~$i=1,\dots,k$
                \State assemble operator~$\bar{\X}_\mu \gets [x_1^\mu\ \cdots\ x_k^\mu]\in \mathcal{L}(\setR^k,\X)$ for projection onto~$\redSpaceTransformed^k$
                \State compute coefficients~$\alpha^\mu=(\alpha_1^\mu,\dots,\alpha_k^\mu)^\top\in\setR^k$ as solution of the~$k\times k$~linear system (see~\cref{equ:projection-onto-parameter-dependent-state-space})\label{lst:computation-coefficients-offline-greedy}
                    \[
                        \bar{\X}_\mu^*\bar{\X}_\mu\alpha^\mu=\bar{\X}_\mu^* M\left(e^{A_\mu T}x_\mu^0 - x_\mu^T\right)
                        \vspace{.5em}
                    \]
                \State compute final time adjoint~$\tilde{\varphi}_\mu^k \gets \sum_{i=1}^{k} \alpha_i^\mu \varphi_i$
            \EndFor
            \State select next parameter~$\mu_{k+1} \gets \argmax_{\mu\in\paramstrain} \eta_\mu(\tilde{\varphi}_\mu^k)$
        \EndWhile
        \Return{$\Phi^N \gets \Phi^k$, $\X^N \gets\X^k$, $D_\mathrm{train} \gets \{(\mu,\alpha^\mu):\mu\in\paramstrain\}$}
        \EndProcedure
    \end{algorithmic}
\end{algorithm}

In the greedy procedure in~\Cref{alg:offline-greedy} we also collect and return all training parameters and the corresponding reduced coefficients in the set~$D_\mathrm{train}$. This will be convenient and useful for the machine learning training in~\Cref{sec:reduced-order-machine-learning}, see in particular~\Cref{alg:offline-machine-learning}.

\begin{remark}[Orthonormalization of the reduced basis]
    To improve numerical stability and to simplify computations of projections, new basis functions can be orthonormalized against the previous ones before adding them to the reduced basis. For that purpose, for instance the Gram-Schmidt algorithm can be used which results in an orthonormal reduced basis~$\Phi^{N}$.
\end{remark}

\begin{remark}[Construction of a reduced basis without using the greedy procedure]
    Instead of applying the greedy algorithm described above, it is also possible to determine a reduced basis using methods such as the POD method~\cite{graessle2021model}. This algorithm requires a set of given training snapshots~$\varphi_{\mu_1}^*(T),\dots,\varphi_{\mu_{n_\mathrm{train}}}^*(T)\in\X$ for several training parameters~$\mu_1,\dots,\mu_{n_\mathrm{train}}\in\params$ that has to be computed beforehand. However, solving the optimal control problem many times is very costly. Since an efficient error estimator is available in our setting, it might therefore be beneficial to select the reduced basis functions in a greedy manner. Nevertheless, if enough training data is already available, the proper orthogonal decomposition might be a plausible alternative to the greedy procedure.
\end{remark}

\subsection{Error Analysis of the Greedy Algorithm}\label{subsec:error-analysis-greedy-algorithm}
In this section, we will investigate the error of the greedy algorithm and prove that~\Cref{alg:offline-greedy} indeed provides a reduced basis~$\Phi^N$ and corresponding reduced space~$\X^N\subset\X$ such that each optimal final time adjoint~$\varphi_\mu^*(T)\in\X$ can be approximated by an error of at most~$\varepsilon$ in the reduced space~$\X^N$ (see~\Cref{thm:weak-greedy-and-approximation-error}). Before doing so, we give a short general introduction to (weak) greedy algorithms and recall a result by DeVore et al.~\cite{devore2013greedy} (see~\Cref{thm:devore-greedy}) stating that weak greedy algorithms produce approximating sequences of reduced spaces with the same convergence rates as the optimal spaces in the sense of Kolmogorov.
\par
Let~$V$ be a Hilbert space and~$K\subset V$ a compact subset of~$V$. Greedy algorithms aim at constructing a sequence of subspaces~$V_N\subset V$ of dimension~$N\in\setN$ that approximate the set~$K$. Similarly to~\Cref{alg:offline-greedy} described in~\Cref{subsec:greedy-for-reduced-basis}, the spaces~$V_N$ are constructed by selecting new basis functions in a ``greedy'' fashion, i.e.~by selecting elements from~$K$ that are currently poorly approximated by elements from~$V_N$. The general procedure is stated in the following pseudocode:
\begin{algorithm}[H]
    \caption{General (weak) greedy algorithm}\label{alg:general-weak-greedy}
    \begin{algorithmic}[1]
        \Require Greedy constant~$\gamma\in(0,1]$, tolerance~$\varepsilon>0$
        \Ensure Reduced basis~$\Psi^N\subset V$, reduced space~$V^N=\Span{\Psi^N}\subset V$
        \State $N\gets 0$,\quad$\Psi^N=\emptyset$,\quad$V^0=\{0\}\subset V$
        \While{$\operatorname{dist}(V^N,K)>\varepsilon$}
            \State choose next element~$v_{N+1}\in K$ such that it holds
            \begin{equation}\label{equ:greedy-choice-condition}
                \hspace{2em}\operatorname{dist}(V^N,\{v_{N+1}\}) \geq \gamma\cdot\operatorname{dist}(V^N,K)
                \vspace{.75em}
            \end{equation}
            \State $\Psi^{N+1} \gets \Psi^N\cup\{v_{N+1}\}$
            \State $V^{N+1} \gets \Span{\Psi^{N+1}}$
            \State $N \gets N+1$
        \EndWhile
        \Return{$\Psi^N$, $V^N$}
    \end{algorithmic}
\end{algorithm}
For a constant~$\gamma<1$ the algorithm above is called \emph{weak} greedy algorithm, while for~$\gamma=1$ the procedure is also known as \emph{strong} greedy algorithm. By considering the optimal reduced subspace of dimension~$N\in\setN$, we can measure the quality of the reduced space obtained by the weak greedy algorithm. We denote by
\begin{equation}
\label{d_n}
    d_N(K) \coloneqq \inf\limits_{\substack{W\subset V,\\\dim W=N}}\operatorname{dist}(W,K)    
\end{equation}
the so-called \emph{Kolmogorov~$N$-width}, where~$W$ denotes an~$N$-dimensional linear subspace of~$V$. In this way, $d_N(K)$ measures the performance of the optimal approximation of~$K$ by a linear space of fixed dimension~$N$.
\par
In contrast, by~$\sigma_N(K)$ we denote the error of the reduced space~$V^N\subset V$ constructed by the greedy algorithm and defined as
\begin{equation}
\label{sigma_n}
    \sigma_N(K) \coloneqq \operatorname{dist}(V^N, K).
\end{equation}
The importance and efficiency of (weak) greedy algorithms is ensured by the following theorem that gives a connection between the decay of the Kolmogorov~$N$-width~$d_N(K)$ with respect to~$N$ and the decay of the error of the reduced space~$\sigma_N(K)$.
\begin{theorem}[DeVore et al.~\cite{devore2013greedy}, Corollary 3.3 (ii) and (iii)]\label{thm:devore-greedy}
    For the weak greedy algorithm with constant~$\gamma$ in a Hilbert space~$V$ and a compact set~$K\subset V$, we have the following:
    \begin{enumerate}
        \item[(i)] If~$d_N(K)\leq C_0N^{-\alpha}$, $N=1,2,\dots$, then~$\sigma_N(K)\leq C_1N^{-\alpha}$, $N=1,2,\dots$, with~$C_1\coloneqq 2^{5\alpha+1}\gamma^{-2}C_0$.
        \item[(ii)] If~$d_N(K)\leq C_0e^{-c_0N^\alpha}$, $N=1,2,\dots$, then~$\sigma_N(K)\leq\sqrt{2C_0}\gamma^{-1}e^{-c_1N^\alpha}$, $N=1,2,\dots$, where~$c_1\coloneqq 2^{-1-2\alpha}c_0$.
    \end{enumerate}
\end{theorem}
This theorem implies that weak greedy algorithms preserve the decay rates of the Kolmogorov~$N$-widths, in such a way providing the optimal approximation errors, up to a multiplicative constant.
\par
In our setting, we consider the Hilbert space~$V=\X$ and aim at approximating the set~$K=\mathcal{M}$ of optimal final time adjoint states over the parameter domain~$\params$ as defined in~\cref{equ:definition-solution-manifold}. We will prove in the following that~\Cref{alg:offline-greedy} is indeed a weak greedy algorithm as introduced in~\Cref{alg:general-weak-greedy}, in particular that the resulting reduced spaces~$\X^N$ satisfy relation~\eqref{equ:greedy-choice-condition}. For this reason, according to~\Cref{thm:devore-greedy}, we can transfer convergence properties of the Kolmogorov~$N$-width of~$\mathcal{M}$ to the error decay of the spaces~$\X^N$ constructed by Algorithm 1.
\begin{theorem}[Weak greedy algorithm and approximation error]\label{thm:weak-greedy-and-approximation-error}
    The procedure presented in~\Cref{alg:offline-greedy} is a weak greedy algorithm with the constant~$\gamma$ defined in~\cref{equ:definition-greedy-constant-gamma}. Furthermore, for the resulting reduced space~$\X^N\subset\X$ it holds for all parameters~$\mu\in\params$ the approximation error estimate
    \begin{align}\label{equ:approximation-error-estimate}
        \operatorname{dist}(\X^N,\{\varphi_\mu^*(T)\}) \quad \leq \quad \varepsilon.
    \end{align}
\end{theorem}
\begin{proof}
    Let~$k\in\{1,\dots,N-1\}$ be an arbitrary iteration of~\Cref{alg:offline-greedy} such that the termination criterion is not fulfilled. Let further be the corresponding reduced basis~$\Phi^k$ and reduced space~$\X^k=\Span{\Phi^k}$ given. We have to prove that it holds
    \begin{align}\label{equ:distance-greedy-inequality}
        \operatorname{dist}(\X^k,\{\varphi_{\mu_{k+1}}^*(T)\}) \geq \gamma\cdot\max\limits_{\mu\in\params}\,\operatorname{dist}(\X^k,\{\varphi_\mu^*(T)\}) = \gamma\cdot\operatorname{dist}(\X^k,\mathcal{M}),
    \end{align}
    with a constant~$\gamma$ independent of~$k$ (see also~\cref{equ:greedy-choice-condition}). To this end, let~$\mu\in\params$ be an arbitrary parameter. Then there exists a parameter~$\tilde{\mu}\in\paramstrain$, such that~$\norm{\mu-\tilde{\mu}}\leq\delta$. Due to the choice of the parameter~$\mu_{k+1}\in\paramstrain$, it holds
    \[
        \eta_{\mu_{k+1}}(\tilde{\varphi}_{\mu_{k+1}}^k) \geq \eta_{\tilde{\mu}}(\tilde{\varphi}_{\tilde{\mu}}^k)
    \]
    for all~$\tilde{\mu}\in\paramstrain$. We then have the estimate
    \begin{align*}
        \operatorname{dist}(\X^k,\{\varphi_\mu^*(T)\}) &\leq \norm{\varphi_\mu^*(T)-\varphi_{\tilde{\mu}}^*(T)} + \operatorname{dist}(\X^k,\{\varphi_{\tilde{\mu}}^*(T)\}) \\
        &\leq C_{\varphi^*}\delta + \operatorname{dist}(\X^k,\{\varphi_{\tilde{\mu}}^*(T)\}) \\
        &\leq C_{\varphi^*}\frac{\tilde{\varepsilon}}{C_\Lambda} + \eta_{\tilde{\mu}}(\tilde{\varphi}_{\tilde{\mu}}^k) \\
        &\leq C_{\varphi^*}\frac{\tilde{\varepsilon}}{C_\Lambda} + \eta_{\mu_{k+1}}(\tilde{\varphi}_{\mu_{k+1}}^k) \\
        &\leq \left(\frac{C_{\varphi^*}}{C_\Lambda}+1\right)\eta_{\mu_{k+1}}(\tilde{\varphi}_{\mu_{k+1}}^k) \\
        &\leq \left(\frac{C_{\varphi^*}}{C_\Lambda}+1\right)C_\Lambda\cdot\operatorname{dist}(\X^k,\{\varphi_{\mu_{k+1}}^*(T)\}) \\
        &= (C_{\varphi^*}+C_\Lambda)\cdot\operatorname{dist}(\X^k,\{\varphi_{\mu_{k+1}}^*(T)\}) \\
        &= \frac{1}{\gamma}\cdot\operatorname{dist}(\X^k,\{\varphi_{\mu_{k+1}}^*(T)\}).
    \end{align*}
    Here, we used that~$\tilde{\varepsilon} < \eta_{\mu_{k+1}}(\tilde{\varphi}_{\mu_{k+1}}^k)$ since the termination criterion is not fulfilled. This proves the inequality in~\eqref{equ:distance-greedy-inequality} by multiplying both sides by~$\gamma>0$, since the parameter~$\mu\in\params$ was arbitrary. Consequently, \Cref{alg:offline-greedy} is a weak greedy algorithm.
    \par
    Regarding the approximation error estimate, we obtain similarly as above that for all~$\mu\in\params$ it holds
    \begin{align*}
        \operatorname{dist}(\X^N,\{\varphi_\mu^*(T)\}) &\leq \frac{C_{\varphi^*}}{C_\Lambda}\tilde{\varepsilon} + \eta_{\mu_{N+1}}(\tilde{\varphi}_{\mu_{N+1}}^N) \\
        &\leq \left(\frac{C_{\varphi^*}}{C_\Lambda}+1\right)\tilde{\varepsilon} \\
        &= \frac{\tilde{\varepsilon}}{C_\Lambda\gamma} \\
        &= \varepsilon,
    \end{align*}
    where we used that~$\eta_{\mu_{N+1}}(\tilde{\varphi}_{\mu_{N+1}}^k) \leq \tilde{\varepsilon}$, i.e.~that the termination criterion is fulfilled by the last computed final time adjoint. This proves the estimate in~\eqref{equ:approximation-error-estimate}.
\end{proof}

Based on the last two results, the introduced~\Cref{alg:offline-greedy} preserves the optimal decay rates, expressed through the Kolmogorov~$N$-widths, of the optimal final time adjoints manifold~$\mathcal{M}$. More precisely, the sequences~$(d_n(\mathcal{M}))_{n\in\setN}$ and~$(\sigma_n(\mathcal{M}))_{n\in\setN}$, defined by~\eqref{d_n} and~\eqref{sigma_n}, respectively, exhibit the same asymptotic behaviour, up to a multiplicative constant. As the solution manifold~$\mathcal{M}$ is usually beyond our disposal, it is hard to estimate its Kolmogorov~$N$-widths directly. However, these can be related to the Kolmogorov widths of the set of admissible parameters~$\params$, which can often be easily determined, even in the case of infinite dimensional parameters (cf.~\cite[\S 4.3]{cohen2015approximation}). More precisely the following result holds~\cite{cohen2016kolmogorov}, under an additional holomorphic assumption.
\begin{theorem}[Cohen and DeVore~\cite{cohen2016kolmogorov}, Theorem 1.1]\label{thm:greedy-hol}
    For a pair of complex Banach spaces~$X$ and~$V$ assume that~$u\colon O\to V$ is a holomorphic map from an open set~$O\subset X$ into~$V$ with uniform bound. If~$K\subset O$ is a compact subset of~$X$ then for any~$\alpha>1$ and~$C_0>0$ it holds
    \[
        d_n(K)\leq C_0 n^{-\alpha} \quad \Longrightarrow \quad d_n(u(K)) \leq C_1 n^{-\beta} \qquad\text{for all }n\in\setN,
    \]
    for any~$\beta<\alpha-1$ and a constant~$C_1$ depending on~$C_0$, $\alpha$ and the mapping~$u$.
\end{theorem}
The proof of the last theorem also provides an explicit estimate of the constant~$C_1$ in dependence on~$C_0$, $\alpha$ and the mapping~$u$. Combining this result with~\Cref{thm:devore-greedy,thm:weak-greedy-and-approximation-error}, one is able to a priori estimate the number of snapshots, i.e. the dimension of the reduced basis space constructed by the offline part of the greedy algorithm, required to approximate the solution manifold within a given error. However, the price of  such an estimate  is the holomorphic dependence of the solutions on the parameters, which is unlikely to occur in real-world problems. Moreover, such estimates tend to be conservative and are not often used in practice. For this reason we do not explore~\Cref{thm:greedy-hol} and do not require additional smoothness of our parameter mappings, other than the one postulated by~\Cref{as:continuity-parameter-to-system-matrices}. We rather run the offline part of the greedy algorithm and observe the dimension of the reduced basis a posteriori.
\par
Summarizing, we have seen in this section that~\Cref{alg:offline-greedy} is a weak greedy algorithm which enables the construction of reduced spaces with theoretically guaranteed convergence properties and error bounds.

\subsection{Efficient and Reduced Online Computations}\label{subsec:online-computations}
After constructing a reduced basis~$\Phi^N\subset\X$ (where the dimension~$N=|\Phi^N|$ of the reduced space depends on the prescribed error tolerance~$\varepsilon>0$) during the offline phase, the reduced order model can be used during the online phase to replace the costly solution of the original optimal control problem. For a new given parameter~$\mu\in\params$, the approximate final time adjoint~$\tilde{\varphi}_\mu^N\in\X^N=\Span{\Phi^N}$ is computed similarly as described in~\Cref{subsec:greedy-for-reduced-basis}, see in particular~\cref{equ:definition-approximate-adjoint}. For completeness, we state the online procedure also in form of a pseudocode, see~\Cref{alg:online-greedy}.
\begin{algorithm}[H]
    \caption{Online evaluation of the approximate control based on the greedy procedure}\label{alg:online-greedy}
    \begin{algorithmic}[1]
        \Require Parameter~$\mu\in\params$, reduced basis~$\Phi^N=\{\varphi_1,\dots,\varphi_N\}$ constructed by~\Cref{alg:offline-greedy}
        \Ensure Approximate final time adjoint~$\tilde{\varphi}_\mu^N\in\X^N$, approximate optimal control~$\tilde{u}_\mu^N\in G$
        \State compute~$x_i^\mu \gets (I+M\Gramian)\varphi_i$ for~$i=1,\dots,N$ \label{lst:online-compute-final-time-states}
        \State assemble operator~$\bar{\X}_\mu \gets [x_1^\mu\ \cdots\ x_N^\mu]\in\mathcal{L}(\setR^N,\X)$ for projection onto~$\redSpaceTransformed^N$
        \State compute coefficients~$\alpha^\mu=(\alpha_1^\mu,\dots,\alpha_N^\mu)^\top\in\setR^N$ as solution of the~$N \times N$~linear system (see~\cref{equ:projection-onto-parameter-dependent-state-space}) \label{lst:online-solve-system-coefficients}
        \begin{equation}\label{equ:linear-system-for-reduced-coefficients}
            \bar{\X}_\mu^*\bar{\X}_\mu\alpha^\mu=\bar{\X}_\mu^* M\left(e^{A_\mu T}x_\mu^0 - x_\mu^T\right)
            \vspace{.5em}
        \end{equation}
        \State compute final time adjoint~$\tilde{\varphi}_\mu^N \gets \sum_{i=1}^{N} \alpha_i^\mu \varphi_i$\qquad (see~\cref{equ:definition-approximate-adjoint})
        \State solve~$-\dot{\tilde{\varphi}}_\mu(t)=A_\mu^*\tilde{\varphi}_\mu(t)$ for~$t\in[0,T]$,\quad $\tilde{\varphi}_\mu(T)=\tilde{\varphi}_\mu^N$ backwards in time\qquad (see~\cref{fromdualtoprimal}) \label{lst:online-time-dependent-adjoint}
        \State compute associated control~$\tilde{u}_\mu^N(t) \gets -R^{-1}B_\mu^*\tilde{\varphi}_\mu(t)$\qquad (see~\cref{control})\label{lst:online-greedy-control}
        \Return{$\tilde{\varphi}_\mu^N$, $\tilde{u}_\mu^N$}
    \end{algorithmic}
\end{algorithm}

By means of the error estimator~$\eta_\mu$ defined in~\cref{equ:definition-error-estimator}, it is possible to obtain an estimation of the error during the online phase as well. By~\Cref{thm:error-estimator-adjoint}, it holds
\[
    \norm{\varphi_\mu^*(T) - \tilde{\varphi}_\mu^N} \leq \eta_\mu(\tilde{\varphi}_\mu^N).
\]
Additionally, if the reduced basis~$\Phi^N$ is constructed using the greedy algorithm introduced in~\Cref{subsec:greedy-for-reduced-basis} for a tolerance~$\varepsilon>0$, we obtain by combining~\Cref{thm:error-estimator-reduced-space} and~\Cref{thm:weak-greedy-and-approximation-error} the estimate
\[
    \eta_\mu(\tilde{\varphi}_\mu^N) \leq C_\Lambda\varepsilon\qquad\text{for all }\mu\in\params,
\]
and hence it holds
\begin{align}\label{equ:error-bound-online-phase}
    \norm{\varphi_\mu^*(T) - \tilde{\varphi}_\mu^N} \leq C_\Lambda\varepsilon
\end{align}
for all parameters~$\mu\in\params$. We can therefore also bound the error in the online phase given that the reduced space was created with a certain error tolerance.
\par
Since we are typically interested in an approximation of the control~$u_\mu^*$, we also give an error bound for the control. To this end, let us define for~$t\in[0,T]$ the constant
\[
    \zeta_{\mu,t} \coloneqq \norm{R^{-1}B_\mu^*e^{A_\mu^*(T-t)}}_{\mathcal{L}(\X,\U)} < \infty.
\]
Then it holds for all~$t\in[0,T]$ that
\[
    \norm{u_\mu^*(t)-\tilde{u}_\mu^N(t)}_\U \leq \zeta_{\mu,t}\norm{\varphi_\mu^*(T)-\tilde{\varphi}_\mu^N}_\X.
\]
Similar to the bound on the error in the final time adjoint, we can deduce that if the reduced basis~$\Phi^N$ is constructed with a greedy tolerance of~$\varepsilon>0$, the error in the control can be bounded by
\[
    \norm{u_\mu^*(t)-\tilde{u}_\mu^N(t)}_\U \leq \zeta_{\mu,t}C_\Lambda\varepsilon.
\]
The constant~$\zeta_{\mu,t}$ depends in particular on the parameter~$\mu$ and properties of the system matrices, such as stability of the system. Certainly, due to~\Cref{as:continuity-parameter-to-system-matrices} and the compactness of~$\params$, there is also a finite uniform upper bound of the constants~$\zeta_{\mu,t}$ over the parameter set~$\params$.
\par
In our numerical experiments below, we will refer to the reduced order model from~\Cref{alg:online-greedy} as~\RBROM{} (greedy reduced order model).

    \section{Acceleration of the Online Phase using Machine Learning}\label{sec:reduced-order-machine-learning}
The reduced order model introduced in~\Cref{sec:reduced-order-modeling-greedy} already gives a speedup in solving the optimal control problem compared to solving the exact problem from~\cref{equ:optimality-system-main}, see also the numerical experiments in~\Cref{sec:numerical-experiments}. However, the computational costs during the online phase still involve high-dimensional computations. If for instance the state space is finite-dimensional, $\X=\setR^n$, the costs scale with the dimension~$n$. To be more specific, for a new parameter~$\mu\in\params$, one has to compute~$x_i^\mu=(I+M\Gramian)\varphi_i$ for~$i=1,\dots,N$ (with~$N$ denoting the dimension of the reduced space) by solving the system from~\cref{equ:linear-system-for-optimal-final-time-adjoint}, see also~\Cref{rem:applying-gramian-to-vectors}. In order to circumvent these steps, we propose to apply machine learning algorithms to directly approximate the map from the parameters to the reduced coefficients. Training data for the machine learning process can be generated using the~\RBROM{} introduced in~\Cref{sec:reduced-order-modeling-greedy}. The idea is motivated by a similar approach that has been introduced in the context of parametrized PDEs in~\cite{hesthaven2018nonintrusive} and was applied to instationary problems in~\cite{wang2019nonintrusive}. In the following section, we first describe the main ideas of our approach in detail. Afterwards, we discuss how the error estimator introduced in~\Cref{subsec:greedy-for-reduced-basis} can be applied to also evaluate the error of the machine learning approximation. We finally introduce specific machine learning algorithms used in our numerical experiments below in~\Cref{sec:numerical-experiments}.

\subsection{Learning the Map from Parameters to Coefficients}
During the online phase, each evaluation of the reduced model as shown in~\Cref{subsec:online-computations} requires the computation of~$N$ final time states~$x_i^\mu\in \X$ in Line~\ref{lst:online-compute-final-time-states} of~\Cref{alg:online-greedy}. Each of these computations amounts in solving a decoupled system of ordinary differential equations, one forward and one backward, as in~\cref{fromdualtoprimal} and \cref{state}. However, this is only required to setup a (small) linear system of equations to compute the reduced coefficients~$\alpha^\mu\in\setR^N$ in Line~\ref{lst:online-solve-system-coefficients} of~\Cref{alg:online-greedy}. These coefficients determine the projection of the ``target''~$M(e^{A_\mu T}x_\mu^0-x_\mu^T)$ to the space of final states, perturbed by the operator~$-\left((\Gramian)^{-1}+M\right)$, reachable from the reduced space~$\X^N$. Afterwards, in order to compute the approximate control, one only has to solve a backward evolutionary equation (Line~\ref{lst:online-time-dependent-adjoint} in~\Cref{alg:online-greedy}) and apply the operator~$-R^{-1}B^*_{\mu}$. The main computational effort in~\Cref{alg:online-greedy} is spent in Line~\ref{lst:online-compute-final-time-states}. The idea to accelerate the online computations is to replace steps~\ref{lst:online-compute-final-time-states}--\ref{lst:online-solve-system-coefficients} by applying a cheaply to evaluate machine learning surrogate that provides, given a parameter~$\mu\in\params$, an approximation~$\hat{\alpha}^\mu\in\setR^N$ of the reduced coefficients~$\alpha^\mu\in\setR^N$ for the parameter~$\mu$. Similarly to the remaining steps of~\Cref{alg:online-greedy}, this results in an approximation~$\hat{\varphi}_\mu^N=\sum_{i=1}^{N}\hat{\alpha}_i^\mu\varphi_i\in\X^N$ of the final time adjoint and associated control~$\hat{u}_\mu^N\in G$. In other words, we define the mapping~$\pi_N\colon\params\to\setR^N$ for a parameter~$\mu\in\params$ as~$\pi_N(\mu)\coloneqq\alpha^\mu$, where~$\alpha^\mu\in\setR^N$ solves the linear system from~\cref{equ:linear-system-for-reduced-coefficients}. Furthermore, we construct an approximation~$\hat{\pi}_N\colon\params\to\setR^N$ of~$\pi_N$ by training a machine learning algorithm and define the machine learning final time adjoint~$\hat{\varphi}_\mu^N\approx\tilde{\varphi}_\mu^N$ as
\begin{align}\label{equ:definition-approximate-adjoint-ml}
    \hat{\varphi}_\mu^N \coloneqq \sum\limits_{i=1}^{N}\big[\hat{\pi}_N(\mu)\big]_i\varphi_i = \sum\limits_{i=1}^{N}\hat{\alpha}_i^\mu\varphi_i.
\end{align}
In our algorithm we first run the greedy algorithm to construct a reduced basis~$\Phi^N$. Additionally, the greedy algorithm already returns the pairs~$(\mu,\pi_N(\mu))\in\paramstrain\times\setR^N$ for all~$\mu\in\paramstrain$. Any supervised machine learning algorithm that builds an approximation by using training data, consisting of inputs and corresponding outputs of the function to approximate, can be used in our setting and trained on the data set~$D_\mathrm{train}=\{(\mu,\pi_N(\mu)):\mu\in\paramstrain\}$. In particular, the approach is not restricted to the machine learning methods applied in our numerical experiments. Furthermore, the error estimates presented in the next section are also independent of the exact implementation of the machine learning surrogate.
\begin{remark}[Smoothness of the parameter to coefficient map]
    The ability of machine learning algorithms to provide a satisfactory approximation of a function is typically linked to the smoothness of the function. Usually, one can observe that the smoother the function to approximate, the better the machine learning approximation. In our setting, the smoothness of the map~$\pi_N$ depends on the smoothness of the maps~$\mu\mapsto A_\mu$, $\mu\mapsto B_\mu$, $\mu\mapsto x_\mu^0$, and~$\mu\mapsto x_\mu^T$, i.e.~all parameter-dependent parts of the optimal control problem. Due to~\Cref{as:continuity-parameter-to-system-matrices}, we can also expect a smooth parameter to coefficients map~$\pi_N$ that is amenable to approximation by machine learning surrogates.
\end{remark}
For completeness and later reference, we summarize the offline and the online procedure for the machine learning reduced model in the following two pseudocodes:
\begin{algorithm}[H]
    \caption{Offline phase of the machine learning greedy procedure}\label{alg:offline-machine-learning}
    \begin{algorithmic}[1]
        \Require Greedy tolerance~$\varepsilon>0$
        \Ensure Reduced basis~$\Phi^N\subset \X$, approximation~$\hat{\pi}_N$ of the parameter to coefficients map
        \State $\Phi^N,\X^N,D_\mathrm{train} \gets \Call{Greedy}{\varepsilon}$\qquad (see~\Cref{alg:offline-greedy})
        \State train a machine learning algorithm using the data~$D_\mathrm{train}$ to obtain a surrogate~$\hat{\pi}_N$
        \Return{$\Phi^N$, $\hat{\pi}_N$}
    \end{algorithmic}
\end{algorithm}
\begin{algorithm}[H]
    \caption{Online evaluation of the approximate control based on the machine learning greedy procedure}\label{alg:online-machine-learning}
    \begin{algorithmic}[1]
        \Require Parameter~$\mu\in\params$, reduced basis~$\Phi^N$ of size~$N=|\Phi^N|$, machine learning approximation~$\hat{\pi}_N$ of the parameter to coefficients map
        \Ensure Approximate final time adjoint~$\hat{\varphi}_\mu^N\in \X^N$, approximate optimal control~$\hat{u}_\mu^N\in G$
        \State evaluate surrogate~$\hat{\pi}_N$ to obtain approximate coefficients~$\hat{\alpha}^\mu \gets \hat{\pi}_N(\mu)$
        \State compute final time adjoint~$\hat{\varphi}_\mu^N \gets \sum_{i=1}^{N} \hat{\alpha}_i^\mu \varphi_i$\qquad (see~\cref{equ:definition-approximate-adjoint-ml})
        \State solve~$-\dot{\hat{\varphi}}_\mu(t)=A_\mu^*\hat{\varphi}_\mu(t)$ for~$t\in[0,T]$,\quad $\hat{\varphi}_\mu(T)=\hat{\varphi}_\mu^N$ backwards in time\qquad (see~\cref{fromdualtoprimal})\label{lst:online-ml-adjoint-equation}
        \State compute associated control~$\hat{u}_\mu^N(t) \gets -R^{-1}B_\mu^*\hat{\varphi}_\mu(t)$\qquad (see~\cref{control})\label{lst:online-ml-control}
        \Return{$\hat{\varphi}_\mu^N$, $\hat{u}_\mu^N$}
    \end{algorithmic}
\end{algorithm}

\begin{remark}[Choice of training parameters for machine learning]
    It is also possible to add more training samples, i.e.~pairs of the form~$(\mu,\pi_N(\mu))$ for~$\mu\in\params$, to the set of training data for the machine learning algorithm. This additional training data can be generated cheaply by running the online algorithm of the~\RBROM{}, see~\Cref{alg:online-greedy}. In particular, such an enrichment of the training set can be performed adaptively as well, depending on the accuracy of the machine learning prediction. To this end, the a posteriori error estimator introduced in~\Cref{subsec:residual-based-error-estimator} and described in~\Cref{subsec:error-estimation-ml} for the machine learning surrogate might be used to evaluate the machine learning performance, see also~\Cref{rem:adaptive-model-hierarchy}.
\end{remark}
Instead of learning the optimal adjoint datum (or control) directly as functions of parameter (and time), our approach requires learning a mapping between the (typically low-dimensional) spaces~$\params$ and~$\setR^N$ (where usually also $N$ is small). This will allow, as we shall see below, for a significant reduction of the computational costs, while at the same time providing reliable and accurate approximations.
\par
The training data for the machine learning surrogate is generated for free in the offline phase (see~Line~\ref{lst:computation-coefficients-offline-greedy} of~\Cref{alg:offline-greedy}). Solving the exact optimal control problem to generate training data, for instance for learning the controls directly, would be prohibitively expensive both in computational time and in memory. Furthermore, as we will see below, using the greedy algorithm and the reduced basis as an intermediate step results in a priori and a posteriori error estimates for the machine learning results. These might not be available when directly learning the controls or adjoints.

\subsection{A Priori and a Posteriori Error Estimation of Machine Learning Results}\label{subsec:error-estimation-ml}
We can derive a simple a priori bound for the error in the approximation of the optimal final time adjoint that consists of the online greedy error in~\cref{equ:error-bound-online-phase} and the machine learning error in approximating the parameter to coefficients mapping:
\begin{lemma}[Error bound for the machine learning approximation]\label{lem:error-bound-machine-learning}
    Let~$\Phi^N=\{\varphi_1,\dots,\varphi_N\}\subset \X$ be a reduced basis constructed using the greedy algorithm from~\Cref{subsec:greedy-for-reduced-basis} for an error tolerance~$\varepsilon>0$. Further, let~$\bar{\Phi}^N\in\mathcal{L}(\setR^{N},\X)$ be the operator given for the~$i$-th unit vector~$e_i\in\setR^N$ as~$\bar{\Phi}^N e_i=\varphi_i$, $i=1,\dots,N$. Let~$\hat{\pi}_N\colon\params\to\setR^N$ be an approximation of the exact parameter to coefficients map~$\pi_N\colon\params\to\setR^N$. Then it holds
    \[
        \norm{\varphi_\mu^*(T)-\hat{\varphi}_\mu^N} \quad \leq \quad C_\Lambda\varepsilon + \norm{\bar{\Phi}^N}_{\mathcal{L}(\setR^{N},\X)}\norm{\pi_N(\mu)-\hat{\pi}_N(\mu)}
    \]
    for all parameters~$\mu\in\params$.
\end{lemma}
\begin{proof}
    Using the bound in~\cref{equ:error-bound-online-phase} and the definitions of the operator~$\bar{\Phi}^N$ as well as the approximate final time adjoints~$\tilde{\varphi}_\mu^N$ and~$\hat{\varphi}_\mu^N$, it holds for all~$\mu\in\params$ that
    \begin{align*}
        \norm{\varphi_\mu^*(T)-\hat{\varphi}_\mu^N} &\leq \norm{\varphi_\mu^*(T)-\tilde{\varphi}_\mu^N} + \norm{\tilde{\varphi}_\mu^N - \hat{\varphi}_\mu^N} \\
        &\leq C_\Lambda\varepsilon + \norm{\bar{\Phi}^N\alpha^\mu - \bar{\Phi}^N\hat{\alpha}^\mu} \\
        &= C_\Lambda\varepsilon + \norm{\bar{\Phi}^N\big(\pi_N(\mu)-\hat{\pi}_N(\mu)\big)} \\
        &\leq C_\Lambda\varepsilon + \norm{\bar{\Phi}^N}_{\mathcal{L}(\setR^{N},\X)}\norm{\pi_N(\mu)-\hat{\pi}_N(\mu)}.
    \end{align*}
\end{proof}
The first term in the error bound of~\Cref{lem:error-bound-machine-learning} corresponds to the greedy approximation error in the offline phase and can be adjusted by the choice of the greedy tolerance~$\varepsilon$. The second term measures the error of the machine learning in approximating the map~$\pi_N$. If one uses orthonormalization during the greedy procedure, it even holds~$\norm{\bar{\Phi}^N}_{\mathcal{L}(\setR^{N},\X)}=1$, which further simplifies the error estimate.
\par
The error of the approximate final time adjoint~$\hat{\varphi}_\mu^N\in \X^N$ can also be estimated in an a posteriori manner by means of the error estimator~$\eta_\mu$ defined in~\cref{equ:definition-error-estimator}. This results, due to~\Cref{thm:error-estimator-adjoint}, in an efficient and reliable error estimator even for the machine learning results, i.e.~it holds
\[
    \norm{\varphi_\mu^*(T)-\hat{\varphi}_\mu^N} \quad \leq \quad \eta_\mu(\hat{\varphi}_\mu^N) \quad \leq \quad \norm{I+M\Gramian}_{\mathcal{L}(\X,\X)}\norm{\varphi_\mu^*(T)-\hat{\varphi}_\mu^N}
\]
for all~$\mu\in\params$. The error certification and the guarantees obtained by the greedy procedure, see also~\Cref{thm:weak-greedy-and-approximation-error} and~\cref{equ:error-bound-online-phase}, are the main advantages of the presented approach compared to learning the optimal control directly as a function of the parameter.

\begin{remark}[Adaptive and certified surrogate model hierarchy from~\cite{haasdonk2023certified}]\label{rem:adaptive-model-hierarchy}
    Recently, a model hierarchy consisting of a full-order model (similar to the exact solution of the parametrized optimal control problem in our setting), a reduced order model (similar to the~\RBROM{} created by the greedy algorithm in~\Cref{sec:reduced-order-modeling-greedy}), and a machine learning surrogate (similar to the one proposed in this section) was introduced in~\cite{haasdonk2023certified} and further tested with deep kernel models as machine learning surrogates in~\cite{wenzel2023application}. Due to the error estimation described above for the~\RBROM{} as well as the machine learning surrogate, the adaptive model hierarchy is also applicable in the setting of parametrized optimal control problems as considered in this contribution. The reduced model in~\cite{haasdonk2023certified} is also constructed by an adaptive procedure depending on the error estimate but without selecting the training parameters a priorily. Instead, the adaptive model is queried for different parameters and the reduced model (and hence also the machine learning surrogate) is automatically built and updated if necessary, i.e.~if the desired error tolerance is not fulfilled. The adaptive model construction from~\cite{haasdonk2023certified} can therefore be seen as a greedy procedure incorporated into the online phase, and thus shows similarities to the greedy construction of the reduced models considered in this paper.
\end{remark}

\subsection{Machine Learning Approaches}
In this subsection we shortly introduce the machine learning approaches applied in the numerical experiments in~\Cref{sec:numerical-experiments}. As already highlighted above, the whole algorithm is not restricted to these specific choices of methods. We start by defining (deep) neural networks in~\Cref{subsec:deep-neural-networks} and afterwards describe kernel methods in~\Cref{subsec:kernel-methods} and Gaussian process regression in~\Cref{subsec:gaussian-process-regression}. More technical details on the implementation and exact choices of certain parts of the machine learning methods can be found in~\Cref{subsec:implementational-details}. For practical implementations, we assume that the parameter set~$\params$ is finite-dimensional, i.e.~$\params\subset\setR^p$ for some~$p\in\setN$.

\subsubsection{Deep Neural Networks}\label{subsec:deep-neural-networks}
Nowadays, deep neural networks are one of the most popular and widespread machine learning algorithms~\cite{lecun2015deep}. They have also been applied to model order reduction~\cite{hesthaven2018nonintrusive,wang2019nonintrusive,haasdonk2023certified,wenzel2023application} and optimal control~\cite{kmet2011neural}. In this work, we restrict our attention to a standard class of neural networks, so-called \emph{feedforward neural networks}. This class of neural networks applies an alternating sequence of affine transformations and element-wise nonlinear activation functions to the input~\cite{petersen2018optimal}. The function defined by the neural network can be tailored to the training data by adjusting the entries of the matrices and vectors in the affine transformations, the so-called \emph{weights} and \emph{biases} of the neural network.
\par
To be more precise, we describe the notion of a feedforward neural network following a formal definition from~\cite{petersen2018optimal}, see also~\cite[Section~4]{keil2022adaptive}: Let~$L\in\setN$ denote the number of layers of the neural network and~$p=N_0,N_1,\dots,N_{L-1},N_L=N\in\setN$ the numbers of neurons in each layer. Here, we emphasize that the number of neurons in the input layer~$N_0=p$ and the number of neurons in the output layer~$N_L=N$ are chosen such that the corresponding neural network defines a mapping~$\setR^p\to\setR^N$ which is required in our application of approximating the parameter to coefficients map. Furthermore, we consider for~$i=1,\dots,L$ the weight matrices~$W_i\in\setR^{N_i\times N_{i-1}}$ and bias vectors~$b_i\in\setR^{N_i}$ which are collected in the tuple~$\theta=\big((W_1,b_1),\dots,(W_L,b_L)\big)$. To introduce nonlinearity, an activation function~$\rho\colon\setR\to\setR$ is applied component-wise between the affine layers. Let~$\rho_n\colon\setR^n\to\setR^n$ denote the component-wise application of~$\rho$ to~$n$-dimensional vectors. With these ingredients at hand, we define the neural network~$\Phi_\theta\colon\setR^p\to\setR^N$ associated to the weights and biases~$\theta$ for an input~$x\in\setR^p$ as
\[
    \Phi_\theta(x) \coloneqq r_L(x),
\]
where~$r_L\colon\setR^p\to\setR^N$ is defined recursively by
\begin{align*}
    r_L(x) &\coloneqq W_Lr_{L-1}(x)+b_L, \\
    r_i(x) &\coloneqq \rho_{N_i}(W_ir_{i-1}(x)+b_i)\qquad\text{for }i=1,\dots,L-1, \\
    r_0(x) &\coloneqq x.
\end{align*}
\par
For a set of training data~$(x_i,y_i)\in\setR^p\times\setR^N$, $i=1,\dots,n_\mathrm{train}$, the weights and biases~$\theta$ of the neural network are optimized such that the loss function
\[
    \mathcal{L}_\mathrm{dnn}(\theta) \coloneqq \frac{1}{n_\mathrm{train}}\sum\limits_{i=1}^{n_\mathrm{train}}\norm{\Phi_\theta(x_i)-y_i}_2^2
\]
is minimized, i.e.~the mean squared error of the neural network in predicting the outputs~$y_i$ given the inputs~$x_i$ for~$i=1,\dots,n_\mathrm{train}$. Hence, the approximate mapping is chosen as~$\hat{\pi}_N\coloneqq\Phi_{\theta^*}$ with~$\theta^*\in\argmin_{\theta} \mathcal{L}_\mathrm{dnn}(\theta)$. The loss function~$\mathcal{L}_\mathrm{dnn}$ is typically optimized using gradient based methods. In our numerical experiments, we apply the quasi-Newton method L-BFGS, see~\cite{liu1989limited}. The gradient of~$\mathcal{L}_\mathrm{dnn}$ with respect to~$\theta$ can be efficiently computed via backpropagation~\cite{rumelhart1986learning}. To avoid overfitting of the training data, one usually also evaluates the loss function for a validation set, that is chosen distinct from the training set, in each iteration of the optimization algorithm. If the loss on the validation set starts to increase over a couple of consecutive iterations, the optimization is cancelled. This procedure is known as \emph{early stopping}, see~\cite{prechelt1997early} for a discussion of different early stopping approaches and~\cite{molinari2021iterative} for theoretical guarantees in a simplified setting.
\par
The outputs in our application are the coefficients with respect to the reduced basis of final time adjoints. Due to the construction of the reduced basis by the greedy algorithm from~\Cref{subsec:greedy-for-reduced-basis}, the basis functions are sorted by importance (similar to the sorting of singular vectors in a singular value decomposition). Therefore, coefficients associated to different basis functions vary quite heavily in their magnitude. To ensure a sufficiently accurate approximation of all coefficients, we thus apply a simple scaling to the coefficients before training the neural network. The scaling is an affine transformation taking into account the minimum and maximum of the respective coefficient over the training set such that each coefficient is mapped to the interval~$[0,1]$ (on the training set).
\par
A feedforward neural network with~$L\geq 3$ layers is typically called \emph{deep neural network} (DNN). The reduced order model that uses deep neural networks for the coefficient prediction will be called~\DNNROM{} in the remainder of this paper.

\subsubsection{Kernel Methods}\label{subsec:kernel-methods}
Approximations using kernel functions have been applied in the context of surrogate modeling quite successfully in the last years, see for instance~\cite{santin2021kernel,haasdonk2023certified,wenzel2023application}. Very recently, kernel methods were also used for solving optimal control problems, see~\cite{ehring2023hermite}. We refer for instance to~\cite{wendland2005scattered} for a general introduction to kernel methods.
\par
Essentially, these methods build around the notion of (scalar) positive-definite kernels which are mappings~$k\colon\setR^p\times\setR^p\to\setR$ such that the kernel matrix~$[k(x_i,x_j)]_{i,j=1}^n$ is symmetric and positive-definite for all~$n\in\setN$ and distinct~$x_i\in\setR^p$, $i=1,\dots,n$. Associated to every positive-definite kernel~$k$ is a reproducing kernel Hilbert space~$H_{k}$ defined as follows, see the Moore-Aronszajn theorem~\cite{aronszajn1950theory}: Let~$H_{k}^0\coloneqq\Span{\left\{k(\cdot,x):x\in\setR^p\right\}}$, then we can define an inner product~$\langle\cdot,\cdot\rangle_{H_{k}}$ on~$H_{k}^0$ by
\[
    \left\langle\sum_{i=1}^{n}\alpha_i k(\cdot,x_i),\sum_{j=1}^{m}\beta_j k(\cdot,z_j)\right\rangle_{H_{k}} \coloneqq \sum_{i=1}^{n}\sum_{j=1}^{m}\alpha_i\beta_j k(x_i,z_j)
\]
for~$n,m\in\setN$, $\alpha_i,\beta_j\in\setR$ and~$x_i,z_j\in\setR^p$ for~$i=1,\dots,n$ and~$j=1,\dots,m$. The space~$H_{k}$ is now given as the completion of~$H_{k}^0$ with respect to the inner product~$\langle\cdot,\cdot\rangle_{H_{k}}$, i.e.~it holds~$H_{k}=\overline{H_{k}^0}^{\langle\cdot,\cdot\rangle_{H_{k}}}$. Furthermore, the space~$H_{k}$ has the function~$k$ as reproducing kernel, i.e. it holds~$\Phi(x)=\langle\Phi,k(\cdot,x)\rangle_{H_{k}}$ for all~$\Phi\in H_{k}$ and~$x\in\setR^p$. Given a set of data points~$(x_i,y_i)\in\setR^p\times\setR$, $i=1,\dots,n_\mathrm{train}$, kernel methods typically try to minimize a loss function~$\mathcal{L}_\mathrm{kernel}\colon H_{k}\to\setR$ (similar to the loss function~$\mathcal{L}_\mathrm{dnn}$ introduced above for the training of neural networks) of the form
\[
    \mathcal{L}_\mathrm{kernel}(\Phi) \coloneqq \frac{1}{n_\mathrm{train}}\sum\limits_{i=1}^{n_\mathrm{train}}\lvert\Phi(x_i)-y_i\rvert^2 + \lambda\norm{\Phi}_{H_{k}}^2
\]
for~$\Phi\in H_{k}$, where~$\lambda\geq 0$ denotes a regularization parameter and~$\norm{\,\cdot\,}_{H_{k}}\coloneqq\sqrt{\langle\cdot,\cdot\rangle_{H_{k}}}$ is the norm induced by the inner product~$\langle\cdot,\cdot\rangle_{H_{k}}$ on the reproducing kernel Hilbert space. A well-known representer theorem, see for instance~\cite{bohn2019representer}, reveals that there are coefficients~$\alpha_i\in\setR$, $i=1,\dots,n_\mathrm{train}$, such that it holds
\[
    \Phi^* = \sum\limits_{i=1}^{n_\mathrm{train}}\alpha_i k(\cdot,x_i),
\]
where~$\Phi^*=\argmin_{\Phi\in H_{k}}\mathcal{L}_\mathrm{kernel}(\Phi)$ minimizes the loss function. In particular, the minimizer of the loss function can be represented as a linear combination of the kernel~$k$ evaluated only at the data points~$x_i$, $i=1,\dots,n_\mathrm{train}$.
\par
The setting introduced above can be extended to vector-valued outputs~$y_i\in\setR^{N}$ as required in our use case of approximating the coefficients of the reduced representation of the optimal final time adjoint. To this end, one considers matrix-valued kernels~$k_N\colon\setR^p\times\setR^p\to\setR^{N\times N}$ defined as~$k_N=k\cdot I_N$ and vector-valued coefficients~$\alpha_i\in\setR^N$, where~$I_N\in\setR^{N\times N}$ denotes the~$N\times N$~identity matrix. Furthermore, to obtain a surrogate that can be evaluated fast, it is beneficial to obtain a sparse approximation of~$\Phi^*$. To be more precise, one aims to select an appropriately chosen subset~$\Xi\subset\{1,\dots,n_\mathrm{train}\}$ of the training set such that~$\lvert\Xi\rvert\ll n_\mathrm{train}$ and approximate~$\Phi^*$ as
\begin{align}\label{equ:sparse-kernel-approximation}
    \Phi^* \approx \hat{\Phi} \coloneqq \sum\limits_{i\in\Xi}\alpha_i k_N(\cdot,x_i)
\end{align}
with certain coefficients~$\alpha_i\in\setR^N$ for~$i\in\Xi$. The selection of the training inputs used in the approximation, i.e.~of the set~$\Xi$, can for instance be done using a greedy algorithm similar to those described in~\Cref{sec:reduced-order-modeling-greedy}. One very popular example of such an algorithm is the \emph{vectorial kernel orthogonal greedy algorithm} (VKOGA), see~\cite{santin2021kernel,wenzel2021novel}, that is used for our numerical experiments below.
\par
We will refer to the surrogate model using~$\hat{\pi}_N \coloneqq \hat{\Phi}$ built by the VKOGA algorithm as~\VKOGAROM{}.

\subsubsection{Gaussian Process Regression}\label{subsec:gaussian-process-regression}
Applied to regression and probabilistic classification tasks, Gaussian processes are an indispensable tool in the machine learning landscape. In particular in the context of regression problems, thus known as \emph{Gaussian process regression} (GPR), these methods have proven to provide satisfactory results in many applications. Gaussian process regression has recently been applied to reduced order modeling in~\cite{guo2018reduced}, and also to (stochastic) optimal control in~\cite{mayer2019stochastic}. For a general introduction to Gaussian processes in the context of machine learning and in particular to Gaussian process regression, see~\cite{rasmussen2006gaussian}. There are strong connections between kernel methods and Gaussian process regression, see for instance~\cite{kanagawa2018gaussian} for a review concerning this matter. The same topic is also discussed in~\cite[Chapter~2.2]{rasmussen2006gaussian}. Our presentation in this section follows~\cite{rasmussen2006gaussian}.
\par
Gaussian process regression starts from a parametrized model function whose parameters are determined such that given training data are likely to occur as results of the model and that the chosen parameters also have a high probability. To make this precise, let us denote by~$\theta$ the parameters, also known as \emph{weights}, of our model function. A probability distribution is assigned to these parameters that determines how likely certain values of the parameters are to occur in the model. This distribution is called \emph{prior} and denoted as~$P(\theta)$. Furthermore, we collect the inputs of our training set in a matrix~$X\in\setR^{n_\mathrm{train}\times p}$ and the corresponding outputs in a matrix~$Y\in\setR^{n_\mathrm{train}\times N}$. Furthermore, we denote by~$P(Y|X,\theta)$ the so-called \emph{likelihood} which corresponds to the probability of the outputs~$Y$ given the weights~$\theta$. We should emphasize at this point that the output~$Y$ is assumed to be affected by random noise and hence it is given as a random disturbance of the true model with unknown weights~$\theta$. The \emph{posterior} distribution of the weights~$P(\theta|Y,X)$ can now be computed using Bayes' rule as
\[
    P(\theta|Y,X) = \frac{P(Y|X,\theta)P(\theta)}{P(Y|X)}.
\]
Given a new input~$x\in\setR^p$, the prediction of the model is a probability distribution~$P(y|x,X,Y)$ which for a possible output~$y\in\setR^N$ is calculated as
\[
    P(y|x,X,Y) = \int P(y|x,\theta)P(\theta|X,Y)\d{\theta}.
\]
It provides the average of the likelihood for the individual sample~$(x,y)\in\setR^p\times\setR^N$ over the possible parameters weighted by the posterior distribution given the weights~$\theta$. As the final prediction of the model, the mean~$\mathbb{E}_y[P(y|x,X,Y)]$ of the distribution~$P(y|x,X,Y)$ is taken. Additional properties of the distribution~$P(y|x,X,Y)$, such as for instance its variance or higher order moments, can be used to further characterize the result and obtain confidence intervals for the prediction. Typically, the prior is defined by a covariance function in form of a kernel similar to the ones introduced in the previous section on kernel methods. The hyper-parameters of the kernel are fitted during the construction of the surrogate, see~\cite{rasmussen2006gaussian} for more details.
\par
The reduced model based on Gaussian process regression, which considers the approximate mapping from parameter to coefficients defined as~$\hat{\pi}_N(x) \coloneqq \mathbb{E}_y[P(y|x,X,Y)]$, will be called~\GPRROM{} in the following.

    \section{Comparison of Computational Costs}\label{subsec:computational-costs}
In this section we compare the overall computational costs during the offline and the online phase of the three methods considered in this contribution: Solving the linear system from~\cref{equ:linear-system-for-optimal-final-time-adjoint} for the optimal final time adjoint~$\varphi_\mu^*(T)$, computing the approximation~$\tilde{\varphi}_\mu^N$ using~\Cref{alg:online-greedy}, and computing the machine learning surrogate solution~$\hat{\varphi}_\mu^N$ as described in~\Cref{sec:reduced-order-machine-learning}. Additionally, we also consider the computational costs for evaluating the error estimator~$\eta_\mu(p)$ for a given approximate final time adjoint~$p\in X$. The error estimator can be applied to both reduced models provided by~\Cref{alg:online-greedy,alg:online-machine-learning} and might be helpful to certify the results a posteriori (see also~\Cref{rem:adaptive-model-hierarchy}). It should as well be stressed at the beginning of the section that all costs are stated for the finite-dimensional setting, i.e.~we assume here~$n \coloneqq \dim(X) < \infty$.
\par
The overall assumption is that the reduced basis constructed by the greedy algorithm is small compared to the dimension of the state space, i.e.~that it holds~$N\ll n$. In this case, we will see below that the reduced models considered in this contribution can greatly reduce the computational effort compared to the exact solution of the optimal control problem. Due to the error estimates presented in~\Cref{subsec:error-analysis-greedy-algorithm,subsec:online-computations,subsec:error-estimation-ml}, the reduced models can achieve accurate approximations at the same time as well.
\par
Below we first consider general assumptions on the discretization of the optimal control system and the costs for computing certain matrix decompositions that help reducing the overall computational effort. Afterwards, we continue by discussing the online costs and the costs for the a posteriori error estimator, because some of these algorithms appear as steps in the offline procedures as well. Finally, we state the costs for the offline computation of the considered reduced models.

\subsection{Preprocessing Costs}
In the following, we assume that a temporal discretization consisting of~$n_t\in\setN$ time steps is used for each ordinary differential equation that arises in the algorithms. This means that we use a time step size of~$\Delta t=T/n_t$. Furthermore, we use implicit time stepping methods, such as the well-known Crank-Nicolson scheme~\cite{crank1947practical}, which require the solution of linear systems of equations in each time step. Since the system matrices arising in our setting are assumed to be independent of time, all linear systems of equations involve the same fixed matrices in each time step (for a fixed parameter~$\mu\in\params$). We can thus first compute the lower–upper (LU)~decomposition~\cite[Chapter~3.2]{golub1996matrix} of~$A_\mu$ and afterwards solve linear systems involving~$A_\mu$ by forward and backward elimination. Similarly, since the matrix~$R$ is symmetric and positive-definite, it possesses a Cholesky decomposition~\cite[Chapter~4.2.3]{golub1996matrix}, i.e.~we can write~$R=\tilde{R}\tilde{R}^\top$ where~$\tilde{R}\in\setR^{m\times m}$ is a lower triangular matrix. Using the Cholesky decomposition, one can replace the costly solution of a dense linear system by forward and backward elimination, using that the Cholesky factor~$\tilde{R}$ is a lower triangular matrix. Hence, the Cholesky decomposition of~$R$ can be used to speedup the computation of~$u_\mu$, which involves solving multiple linear systems of equations with~$R$ as system matrix (see~\cref{equ:optimality-system-odes}). Certainly, in the case of a parameter-dependent weighting matrix~$R_\mu$, it is generally impossible to precompute the Cholesky factor of the weighting matrix since it changes with the parameter.
\par
Computing the LU~decomposition of~$A_\mu\in\setR^{n\times n}$ requires~$\mathcal{O}(n^3)$ operations in general. For the Cholesky decomposition of~$R\in\setR^{m\times m}$, in total~$\mathcal{O}(m^3)$ operations are needed. The preprocessing costs therefore amount to
\begin{align}\label{equ:preprocessing-cost}
    \mathcal{O}\big(n^3+m^3\big)
\end{align}
operations. To solve an ordinary differential equation of the form as in~\cref{equ:parametrized-control-system} or in~\cref{equ:optimality-system-odes}, we need~$\mathcal{O}\big(n_tn(n+m)\big)$ operations under the assumption that the respective LU~decompositions and Cholesky decompositions have already been computed before. For systems with sparse system matrices, for instance those stemming from finite element discretizations of PDEs, it might nevertheless be beneficial to use iterative solvers instead of LU~decompositions.

\subsection{Costs for Calculating the (approximate) Optimal Control}
In the next three subsections we discuss the costs for computing the exact optimal control, its approximation using~\Cref{alg:online-greedy} and the machine learning based approximation from~\Cref{alg:online-machine-learning}. We already emphasize at this point that the costs stated below will contain the costs for computing the (approximate) optimal adjoint datum as well as the costs for deriving the (approximate) optimal control from the adjoint datum. Furthermore, we note that in typical applications it holds~$n_t\geq n$ and~$m\ll n$.
\subsubsection{Exact Optimal Control}\label{subsubsec:costs-exact-optimal-control}
In order to compute the exact optimal final time adjoint~$\varphi_\mu^*(T)$ (up to a prescribed numerical tolerance), one has to solve the linear system stated in~\Cref{lem:linear-system}. Recall that~$M$ and~$\Gramian$ are symmetric and positive semi-definite matrices. If they commute, also~$I+M\Gramian$ is symmetric and positive semi-definite. Consequently, the linear system of equations in~\eqref{equ:linear-system-for-optimal-final-time-adjoint} can be solved efficiently for~$\varphi_\mu^*(T)$ using the conjugate gradient~(CG) method~\cite[Chapter~10.2]{golub1996matrix}. The~CG~method only requires applications of the (symmetric and positive-definite) system matrix to certain vectors. As described in~\Cref{rem:applying-gramian-to-vectors}, this can be done in our case without explicitly constructing the Gramian matrix~$\Gramian$. However, it is still required to solve two initial value problems (one for the adjoint state and one for the primal state) to compute a single product~$(I+M\Gramian)p$ for a vector~$p\in X$. This step can therefore become costly for large systems and needs to be performed multiple times in each iteration of the~CG~algorithm. We also emphasize that the system matrix~$I+M\Gramian$ depends on the parameter~$\mu\in\params$ and can thus not be efficiently precomputed (even under assumptions such as affine parameter-dependence of~$A_\mu$ and~$B_\mu$, due to the matrix exponential in~\cref{equ:definition-gramian-matrix}). The~CG~algorithm requires to solve the optimality system in~\cref{equ:optimality-system-main} multiple times for different terminal conditions for the adjoint state. The number of iterations of the~CG~algorithm depends on the condition number of the system matrix at hand. To enable a comparison with the reduced models below, let us denote by~$n_\mathrm{CG}\in\setN$, $n_\mathrm{CG}\leq n$, the number of iterations needed in the~CG~algorithm until the norm of the residual drops below a prescribed tolerance (we omit at this point an explicit dependence of the number of iterations on the parameter~$\mu$, but instead assume that the number of iterations is similar for all parameters~$\mu\in\params$). Computation of the optimal control~$u_\mu^*$ for a parameter~$\mu\in\params$ requires
\[
    \mathcal{O}\big(n^3+m^3+n_\mathrm{CG}n_tn(n+m)\big)
\]
operations, where also the preprocessing costs stated in~\cref{equ:preprocessing-cost} are considered.
\begin{remark}
    If~$M\Gramian$ is not symmetric and positive semi-definite (or if no guarantees on it are given a priori), the linear system in~\Cref{lem:linear-system} can be solved by other iterative methods. For instance, gradient descent can be used on the least-squares problem. Again, the computation of the full Gramian is not required, but only its application on vectors. The cost per each iteration is then the cost of the resolution of a forward and a backward dynamical system in dimension~$n$ with~$n_t$ time steps. Moreover, for a given tolerance, the number of iterations needed to achieve that precision is given by the convergence rates of the method. In what follows, for simplicity and due to the clearness of the complexity of the~CG~method, we stick to the case in which the~CG~method can be applied.
\end{remark}
\subsubsection{Reduced Optimal Control based on the Greedy Procedure}\label{subsubsec:costs-greedy-online}
In order to compute the reduced optimal adjoint~$\tilde{\varphi}_\mu^N$ for a given parameter~$\mu\in\params$ according to~\Cref{alg:online-greedy}, we first compute the decompositions of~$R$ and~$A_\mu$ as already described above. Furthermore, Line~\ref{lst:online-compute-final-time-states} in~\Cref{alg:online-greedy} requires the solution of~$2N$ evolution systems (the primal and the adjoint for each vector of the reduced basis), each of the cost~$\mathcal{O}(n_tn(n+m))$. Assembling the matrix~$\bar{X}_\mu^\top\bar{X}_\mu\in\setR^{N\times N}$ requires~$\mathcal{O}(N^2n)$ operations, while solving the corresponding linear system of equations for the reduced coefficients in~\cref{equ:linear-system-for-reduced-coefficients} is of complexity~$\mathcal{O}(N^3)$. Computing the linear combination of the basis functions according to~\cref{equ:definition-approximate-adjoint} requires~$\mathcal{O}(Nn)$ operations and is hence negligible compared to the complexity for setting up the linear system. Altogether, taking into account the preprocessing costs~\eqref{equ:preprocessing-cost}, the costs for computing the reduced optimal control using~\Cref{alg:online-greedy} is of complexity
\[
    \mathcal{O}\big(n^3+m^3+Nn_tn(n+m)+N^2n+N^3\big).
\]
Consequently, the online phase of the greedy reduced order model provides a speedup compared to computing the exact optimal control as long as it holds~$N<n_\mathrm{CG}$.
\subsubsection{Reduced Optimal Control based on the Machine Learning Greedy Procedure}\label{subsubsec:costs-ml-procedure}
Typically, the evaluation of a machine learning surrogate~$\hat{\pi}_N\colon\params\subset\setR^p\to\setR^N$ requires~$\mathcal{O}\big(\chi(p+N)\big)$ operations, where~$\chi\colon\setN\to\setN$ is a polynomial of moderate degree. For instance in the case of deep neural networks, see~\Cref{subsec:deep-neural-networks}, the numbers of neurons in each layer of the network usually scale with~$p+N$ and therefore the computational effort for a forward pass through the neural network is of complexity~$\mathcal{O}\big((p+N)^2\big)$ due to the required matrix-vector multiplications. The constants in the bound depend on the number of layers in the network. Together with the reconstruction step of the approximate final time adjoint from the coefficients, which requires~$\mathcal{O}(Nn)$ operations, see~\cref{equ:definition-approximate-adjoint-ml}, the total costs for the online phase in the case of the machine learning reduced model amount to
\[
    \mathcal{O}\big(n^3+m^3+n_tn(n+m)+\chi(p+N)+Nn\big).
\]
We emphasize at this point that the costs for computing the approximate control using the machine learning surrogate is dominated by the computation of the control for the given approximate final time adjoint, which requires~$\mathcal{O}\big(n^3+m^3+n_tn(n+m)\big)$ operations. The computation of the approximate final time adjoint itself is typically much faster and of complexity~$\mathcal{O}\big(\chi(p+N)+Nn\big)$ when applying the machine learning surrogate.
\par
In comparison to the computational costs for the reduced model based on the greedy procedure, see~\Cref{subsubsec:costs-greedy-online}, the complexity of evaluating the machine learning based reduced model (under the assumptions that~$p\ll n$, $N\ll n$, $m\ll n$, and~$n_t\geq n$) is given by~$\mathcal{O}\big(n_tn(n+m)\big)$ while the reduced model from~\Cref{subsec:online-computations} requires~$O\big(Nn_tn(n+m)\big)$ operations. As we will see in the numerical experiments in~\Cref{sec:numerical-experiments}, this reduction in the computational complexity will result in a serious speedup when using the machine learning procedure.

\subsection{Evaluation of the a Posteriori Error Estimator}\label{subsec:costs-a-posteriori-estimator}
Given a parameter~$\mu\in\params$ and an approximate final time adjoint~$p\in X$, evaluating the error estimate~$\eta_\mu(p)$ defined in~\cref{equ:definition-error-estimator} mainly requires the same steps as evaluating the reduced model in~\Cref{alg:online-greedy}, except that the system in~\cref{equ:optimality-system-main} has to be solved only once and not~$N$~times. Therefore, the dominating costs for the error estimator are of the order
\[
    \mathcal{O}\big(n^3+m^3+n_tn(n+m)\big).
\]
We observe in particular that the evaluation of the error estimator is of the same complexity as computing the reduced optimal control using the machine learning model, see~\Cref{subsubsec:costs-ml-procedure}.
\par
If the approximate final time adjoint~$p\in X$ has been computed using the online procedure of the~\RBROM{} from~\Cref{alg:online-greedy}, i.e.~if it holds~$p=\tilde{\varphi}_\mu^N\in X^N$, we can use that
\begin{align}\label{equ:simplified-computation-perturbed-adjoint}
    (I+M\Gramian)\tilde{\varphi}_\mu^N = \sum\limits_{i=1}^{N}\alpha_i^\mu(I+M\Gramian)\varphi_i = \sum\limits_{i=1}^{N}\alpha_i^\mu x_i^\mu.
\end{align}
Therefore, we can reuse the states~$x_i^\mu$ that have already been computed in~Line~\ref{lst:online-compute-final-time-states} and the right hand side of the system given by~$M\left(e^{A_\mu T}x_\mu^0 - x_\mu^T\right)$ computed in~Line~\ref{lst:online-solve-system-coefficients} of~\Cref{alg:online-greedy} to reduce the computational costs of evaluating the error estimator~$\eta_\mu(\tilde{\varphi}_\mu^N)$ to~$\mathcal{O}\big(Nn\big)$, which is required for computing~$(I+M\Gramian)\tilde{\varphi}_\mu^N$ according to~\eqref{equ:simplified-computation-perturbed-adjoint} and the norm of the residual. All other components of the residual have already been computed in~\Cref{alg:online-greedy}. This simplification can also be used in the implementation of the offline greedy procedure in~\Cref{alg:offline-greedy}.

\subsection{Costs for the Offline Computations}
Subsequently we state the costs for performing the greedy procedure in~\Cref{alg:offline-greedy} and the machine learning greedy procedure from~\Cref{alg:offline-machine-learning}. In the discussion we make use of the results on the costs for the online computations obtained above.
\subsubsection{Greedy Procedure}
If the storage capabilities allow, we can precompute the LU~decomposition of~$A_\mu$ for all parameters~$\mu\in\paramstrain$, which requires~$\mathcal{O}(n_\mathrm{train}n^3)$ operations. Afterwards, the \texttt{while}-loop in~\Cref{alg:offline-greedy} runs for~$N$ iterations (where~$N$ is certainly unknown before executing the algorithm). In each iteration, the exact optimal adjoint datum for one training parameter has to be computed (which needs~$\mathcal{O}(n_\mathrm{CG}n_tn(n+m))$ operations, see~\Cref{subsubsec:costs-exact-optimal-control}). Furthermore, for each of the~$n_\mathrm{train}$ training parameters the reduced optimal adjoint has to be computed, which is of complexity~$\mathcal{O}(Nn_tn(n+m)+N^2n+N^3)$ (since the iteration counter~$k$ is increased up to a value of~$N$), and the a posteriori error estimator is evaluated (which costs only~$\mathcal{O}(n)$ operations, see the discussion in~\Cref{subsec:costs-a-posteriori-estimator}). Altogether, the complexity of running~\Cref{alg:offline-greedy} can be estimated as
\[
    \mathcal{O}\left(n_\mathrm{train}n^3+m^3+N\left(n_\mathrm{CG}n_tn(n+m)+n_\mathrm{train}\big(Nn_tn(n+m)+N^2n+N^3\big)\right)\right).
\]
\subsubsection{Machine Learning Greedy Procedure}
To build a machine learning surrogate using the machine learning greedy procedure from~\Cref{alg:offline-machine-learning}, the greedy procedure in~\Cref{alg:offline-greedy} is performed as the first step. Afterwards, the machine learning algorithm is trained using the collected training data~$D_\mathrm{train}$. If we denote by~$C_\mathrm{train}$ the costs of training the machine learning surrogate (which typically depends on the size~$n_\mathrm{train}$ of the training set, the dimension~$p$ of the parameter space, the reduced dimension~$N$, and the polynomial~$\chi$ introduced in~\Cref{subsubsec:costs-ml-procedure}), the overall costs for~\Cref{alg:offline-machine-learning} are of the complexity
\[
    \mathcal{O}\left(n_\mathrm{train}n^3+m^3+N\left(n_\mathrm{CG}n_tn(n+m)+n_\mathrm{train}\big(Nn_tn(n+m)+N^2n+N^3\big)\right)+C_\mathrm{train}\right).
\]

    \section{Numerical Experiments}\label{sec:numerical-experiments}
In the following section we investigate the proposed algorithms in two numerical examples. We first consider a parametrized version of the heat equation where the parameter influences the conductivity as well as the target state. The control acts on the two boundaries of the one-dimensional computational domain. As a second example, we examine a damped wave equation with boundary control where the parameter changes the wave propagation speed.
\par
Before describing the numerical experiments in detail, we give an overview of some implementational aspects and discuss the software used to perform the test cases.

\subsection{Implementational Details}\label{subsec:implementational-details}
The temporal and spatial discretizations of the considered equations are specified below individually for the respective examples. In this section we only state general details that are the same for both experiments.
\par
As discussed in~\Cref{subsec:optimality-system}, we apply the~CG~algorithm to solve the optimality system for the exact optimal final time adjoint. We choose a tolerance of~$\varepsilon_\mathrm{CG}=10^{-12}$ for the~CG~method in all experiments, i.e.~the algorithm is stopped once the Euclidean norm of the residual drops below~$\varepsilon_\mathrm{CG}$.
\par
All numerical test cases are implemented in the \texttt{Python} programming language. Computations dealing with vectors and matrices, in particular computations of matrix decompositions, use the \texttt{numpy}~\cite{harris2020array} and \texttt{scipy}~\cite{virtanen2020SciPy} packages. The neural networks are implemented using \texttt{PyTorch}~\cite{paszke2019PyTorch} and are adapted from the code used in the model order reduction software \texttt{pyMOR}~\cite{milk2016pyMOR}. Furthermore, we apply the \texttt{VKOGA} library\footnote{The \texttt{VKOGA} library is available at \url{https://github.com/GabrieleSantin/VKOGA}}~\cite{santin2021kernel} for implementing the kernel methods, and use the implementation of Gaussian process regression available in \texttt{scikit-learn}~\cite{scikit-learn}.
\par
We use a fixed neural network architecture throughout the experiments that consists of three hidden layers with~$50$~neurons in each layer, i.e.~$L=4$, $N_0=p$, $N_1=N_2=N_3=50$, and~$N_4=N$. This architecture is chosen sufficiently large for our application and no special adaptation was necessary in our experiments. As activation function we apply the well-known hyperbolic tangent~$\rho=\tanh$. As described in~\Cref{subsec:deep-neural-networks} we use the L-BFGS algorithm implemented in~\texttt{PyTorch} for optimization and perform~$10$~restarts of the training using random initial weights and biases. The neural network obtaining the smallest loss is finally returned by the training procedure and used for the prediction. For validation of the results during the optimization,~$10\%$~of the training parameters are randomly selected as a validation set. The validation data is used to evaluate the early stopping criterion which checks whether the loss decreased within the last~$10$ optimization steps. If that is not the case, the training is stopped.
\par
For the ~\VKOGAROM{} we apply the Gaussian radial basis function kernel~$k\colon\setR^p\times\setR^p\to\setR$ defined for~$x,y\in\setR^p$ as
\[
    k(x,y) \coloneqq \exp\left(-\big(\beta\norm{x-y}_2\big)^2\right),
\]
where the constant~$\beta>0$ is defined for the heat equation experiment in~\Cref{subsec:heat-equation-experiment} as~$\beta=0.1$ and for the wave equation example in~\Cref{subsec:wave-equation-experiment} as~$\beta=1$. Further, in order to select the subset~$\Xi$ of training parameters used for the sparse representation of the kernel interpolant in~\cref{equ:sparse-kernel-approximation}, we use the~$P$-greedy algorithm (see~\cite{santin2021kernel} for more details) with a tolerance of~$\varepsilon_P=10^{-10}$. The regularization parameter~$\lambda$ is set to~$\lambda=0$, i.e.~no additional regularization is incorporated in the loss function.
\par
For the Gaussian process regression surrogate, we apply the prior defined by the covariance function given through the kernel~$k$ defined as
\[
    k(x,y) \coloneqq c\cdot\exp\left(-\frac{\norm{x-y}_2^2}{2l^2}\right)
\]
for inputs~$x,y$ with the hyper-parameters~$c\in[0.1,1000]$ and~$l\in[0.001,1000]$. The optimization process for fitting the hyper-parameters is restarted~$10$~times with different initial guesses to maximize the $\log$-marginal likelihood. To ensure that the matrix appearing during the fitting process is positive-definite, 
we add a constant of~$\alpha=0.001$ to the diagonal of the kernel matrix. Furthermore, the outputs are normalized to zero mean and unit variance which fits to the chosen prior with the same properties.
\par
Regarding the constants~$C_\Lambda$, $C_{\varphi^*}$, and~$\gamma$ introduced in~\Cref{subsec:greedy-for-reduced-basis}, we remark that there typically do not exist sharp and easy to compute estimates for these constants. A naive estimate using submultiplicativity of the matrix norm and the triangle inequality results in pessimistic estimates that are impractical in numerical experiments. We therefore directly prescribe the tolerance~$\tilde{\varepsilon}$ and select an appropriate training set~$\paramstrain\subset\params$ (instead of specifying the error tolerance~$\varepsilon$), which is sufficient to run the greedy procedure from~\Cref{alg:offline-greedy} starting in~Line~\ref{lst:start-greedy-without-tolerance}.
\par
The numerical experiments have been carried out on a dual socket compute server with two Intel(R) Xeon(R) Gold 6254 CPUs running at 3.10GHz and 36~cores in each CPU.
\par
We also provide the source code for our numerical experiments~\cite{sourcecode}, which can be used to reproduce the numerical results stated below\footnote{The corresponding \texttt{GitHub}-repository containing the source code is available at~\url{https://github.com/HenKlei/ML-OPT-CONTROL}}.

\subsection{Test Case 1: Heat Equation}\label{subsec:heat-equation-experiment}
As a first example we consider the heat equation in one spatial dimension with a two-dimensional parameter~$\mu\in\params\coloneqq[1,2]\times[0.5,1.5]\subset\setR^2$. The first component~$\mu_1$ of the parameter~$\mu=[\mu_1,\mu_2]\in\params$ determines the conductivity in the equation whereas the second component~$\mu_2$ determines the target state~$v_\mu^T$. The problem of interest is given as
\begin{alignat*}{2}
    \partial_t v_\mu(t,y) - \mu_1\Delta v_\mu(t,y) &= 0 && \text{for }t\in[0,T],y\in\Omega, \\
    v_\mu(t,0) &= u_{\mu,1}(t) && \text{for }t\in[0,T], \\
    v_\mu(t,1) &= u_{\mu,2}(t) && \text{for }t\in[0,T], \\
    v_\mu(0,y) &= v_\mu^0(y) = \sin(\pi y) \qquad && \text{for }y\in\Omega.
\end{alignat*}
Here, we denote by~$u_\mu(t)=\big[u_{\mu,1}(t),u_{\mu,2}(t)\big]^\top\in\setR^2$ for~$t\in[0,T]$ the (two-dimensional) control that influences the boundary conditions on both ends of the spatial domain~$\Omega=[0,1]$. The target state is given as~$v_\mu^T(y)=\mu_2y$ for~$y\in\Omega$. Furthermore, we consider the dynamics until the final time~$T=0.1$ is reached. The system above is discretized in space using a second-order central finite difference scheme on a uniform spatial grid consisting of~$n_y=100$ inner points with a spatial grid size of~$h=1/(n_y+1)$. In such a way the above system can be written in the form of~\eqref{equ:parametrized-control-system} with the system matrices given as
\begin{align}\label{equ:system-matrices-heat-equation}
    A_\mu = \frac{\mu_1}{h^2}\bar{A}
    \qquad\text{with}\qquad
    \bar{A}=\begin{bmatrix}
        -2 & 1      &        &        &   \\
        1 & -2      & 1      &        &   \\
          &  \ddots & \ddots & \ddots &   \\
          &         & 1      & -2     & 1 \\
          &         &        & 1      & -2
    \end{bmatrix}\in\setR^{n\times n}
    \qquad\text{and}\qquad
    B_\mu = \frac{\mu_1}{h^2}\begin{bmatrix}
        1 & 0 \\
        0 & 0 \\
        \vdots & \vdots \\
        0 & 0 \\
        0 & 1
    \end{bmatrix}\in\setR^{n\times 2},
\end{align}
where the system dimension is~$n=n_y$. This system fulfills the well-known \emph{Kalman rank condition} for all parameters~$\mu\in\params$ and is therefore controllable. Furthermore, the continuity requirements stated in~\Cref{as:continuity-parameter-to-system-matrices} are also obviously fulfilled due to the choice of the system matrices~$A_\mu$ and~$B_\mu$ and the target state~$v_\mu^T$.
\par
For discretization of the time derivative we apply the Crank-Nicolson method, see~\cite{crank1947practical}, for~$n_t=30\cdot n_y$ uniform time steps with a time step size of~$\Delta t=1/n_t$. 
\par
In the context of discretized PDEs, the standard Euclidean inner product might not be well-suited. We therefore choose a weighted version of the Euclidean inner product for the state space that takes into account the spatial discretization size~$h$, see~\cite{haasdonk2011efficient}. To be more precise, we equip the state space~$\X=\setR^n$ with the inner product~$\langle\cdot,\cdot\rangle_h$ defined for~$x,y\in\X$ as
\begin{align}\label{equ:inner-product-state-space-example}
    \langle x,y\rangle_h = h\langle x,y\rangle_2,
\end{align}
where~$\langle\cdot,\cdot\rangle_2$ denotes the standard Euclidean inner product on~$\setR^n$. For the space of controls~$\U=\setR^2$ we use the standard Euclidean inner product. However, on the time-discretized control space, i.e.~the discretized version of the space~$G$, we consider the norm~$\norm{\cdot}_{\Delta t}$ defined for~$u=[u_{i,j}]_{i=1,\dots,n_t;j=1,2}\in\setR^{n_t\times 2}$ as
\begin{align}\label{equ:norm-control-space-example}
    \norm{u}_{\Delta t} \coloneqq \left(\Delta t\sum\limits_{i=1}^{n_t} \norm{u_{i}}_2^2\right)^{\frac{1}{2}},
\end{align}
which is only used to compare optimal and approximate controls. Due to the definition of the cost functional in~\cref{equ:cost-functional}, the inner product~$\langle\cdot,\cdot\rangle_h$ also enters the optimal control problem.
\par
The weighting matrices~$M\in\setR^{n\times n}$ and~$R\in\setR^{2\times 2}$ are chosen as
\[
    M = I\qquad\text{and}\qquad R = \begin{bmatrix}0.125 & 0 \\ 0 & 0.25\end{bmatrix},
\]
which in particular fulfill~\Cref{as:symmetric-product}. The different weights for the different components of the control have the effect that the size of the second component~$u_{\mu,2}$ is penalized stronger than the size of the first one~$u_{\mu,1}$, i.e.~controlling the right boundary value of the state is considered to be more costly than controlling the left boundary value. An example of the optimal final time adjoint as well as the associated controls and states for the parametrized heat equation is shown in~\Cref{fig:heat-equation-example-results} for the parameter~$\mu=(1.5,0.75)\in\params$. We denote the discretized states by~$x_\mu$ to use the same notation as in the previous sections. In the top left plot of the figure we notice that the initial state is driven close to the target state over time with the control acting on the boundary points. This can also be seen in the bottom left plot of the figure where the initial state, target state and optimal state at final time are compared. Since we are solving an optimal control problem where a deviation from the target state is penalized but the system is not forced to hit the target state exactly, we see a slight deviation from the target state in the optimal final time state. Furthermore, we observe in the top right plot of~\Cref{fig:heat-equation-example-results} that the optimal final time adjoint is very smooth, which facilitates the approximation of optimal final time adjoints by lower-dimensional subspaces. The bottom right plot of~\Cref{fig:heat-equation-example-results} shows how the two components of the control change over time. In particular at the final time~$T=0.1$, the first component~$u_{\mu,1}^*$ is close to~$0$ while the second component~$u_{\mu,2}^*$ is close to the value~$0.75$, as expected due to the choice of the target state~$v_\mu^T(y)=\mu_2y$, the parameter value~$\mu_2=0.75$, and because the controls act on the boundary values of the solution.
\par
\begin{figure}[ht]
    \centering
    \begin{tikzpicture}
        \begin{axis}[
            name=state-trajectory,
            anchor=north,
            height=4cm,
            width=.48\textwidth,
            legend cell align={left},
            legend style={font=\footnotesize, fill opacity=1, draw opacity=1, text opacity=1, xshift=5pt, yshift=33pt, draw=white!80!black},
            legend columns=2,
            tick align=outside,
            tick pos=left,
            x grid style={white!69.0196078431373!black},
            xlabel={$x$},
            xlabel style={yshift=12pt},
            xmin=0., xmax=1.,
            xtick style={color=black},
            xtick={0,1},
            xticklabels={0,1},
            scaled y ticks=false,
            y grid style={white!69.0196078431373!black},
            ylabel={$t$},
            ylabel style={yshift=-3pt},
            ytick style={color=black},
            ymin=0., ymax=0.1,
            ytick={0,0.1},
            yticklabels={0,0.1},
            view={0}{90},
            colorbar horizontal,
            colormap/viridis,
            colorbar style={at={(0,-0.5)}, anchor=south west},
        ]
            \addplot3 [surf, mesh/rows=22, shader=interp, z buffer=sort, colormap/viridis] table [x index=1, y index=0] {data/heat_equation/optimal_state_trajectory.txt};
        \end{axis}

        \begin{axis}[
            name=state,
            anchor=north,
            at=(state-trajectory.south),
            yshift=-3cm,
            width=.48\textwidth,
            height=4cm,
            legend cell align={left},
            legend style={font=\footnotesize, fill opacity=1, draw opacity=1, text opacity=1, xshift=5pt, yshift=33pt, draw=white!80!black},
            legend columns=2,
            tick align=outside,
            tick pos=left,
            x grid style={white!69.0196078431373!black},
            xlabel={$x$},
            xlabel style={yshift=12pt},
            xmajorgrids,
            xmin=0, xmax=1,
            xtick style={color=black},
            xtick={0,1},
            xticklabels={0,1},
            scaled y ticks=false,
            y grid style={white!69.0196078431373!black},
            ymajorgrids,
            ylabel={States},
            ylabel style={yshift=-3pt},
            ytick style={color=black}
        ]
            \addplot[thick, color5] table[x index=0, y index=2] {data/heat_equation/final_time_state.txt};
            \addlegendentry{$x_\mu^0$: initial state}
            \addplot[thick, color3] table[x index=0, y index=3] {data/heat_equation/final_time_state.txt};
            \addlegendentry{$x_\mu^T$: target state}
            \addplot[thick, color0, dashed] table[x index=0, y index=1] {data/heat_equation/final_time_state.txt};
            \addlegendentry{$x_\mu^*(T)$: final time state}
        \end{axis}

        \begin{axis}[
            name=control,
            anchor=west,
            at=(state.east),
            xshift=2cm,
            width=.48\textwidth,
            height=4cm,
            legend cell align={left},
            legend style={font=\footnotesize, fill opacity=1, draw opacity=1, text opacity=1, xshift=5pt, yshift=31pt, draw=white!80!black},
            tick align=outside,
            tick pos=left,
            x grid style={white!69.0196078431373!black},
            xlabel={$t$},
            xlabel style={yshift=12pt},
            xmajorgrids,
            xmin=0, xmax=0.1,
            xtick style={color=black},
            xtick={0,0.1},
            xticklabels={0,0.1},
            scaled y ticks = false,
            y grid style={white!69.0196078431373!black},
            ymajorgrids,
            ylabel={Control},
            ylabel style={yshift=-3pt},
            ytick style={color=black}
        ]
            \addplot[thick, color0] table[y index=1] {data/heat_equation/optimal_control.txt};
            \addlegendentry{$u_{\mu,1}^*$: optimal control, first component}
            \addplot[thick, color5] table[y index=2] {data/heat_equation/optimal_control.txt};
            \addlegendentry{$u_{\mu,2}^*$: optimal control, second component}
        \end{axis}

        \begin{axis}[
            name=adjoint,
            anchor=south,
            at=(control.north),
            yshift=2cm,
            width=.48\textwidth,
            height=4cm,
            legend cell align={left},
            legend style={font=\footnotesize, fill opacity=1, draw opacity=1, text opacity=1, xshift=5pt, yshift=20pt, draw=white!80!black},
            tick align=outside,
            tick pos=left,
            x grid style={white!69.0196078431373!black},
            xlabel={$x$},
            xlabel style={yshift=12pt},
            xmajorgrids,
            xmin=0, xmax=1,
            xtick style={color=black},
            xtick={0,1},
            xticklabels={0,1},
            scaled y ticks = false,
            y grid style={white!69.0196078431373!black},
            ymajorgrids,
            ylabel={Adjoint},
            ylabel style={yshift=-3pt},
            ytick style={color=black},
            scaled y ticks=base 10:2,
        ]
            \addplot[thick, color1] table {data/heat_equation/optimal_final_time_adjoint.txt};
            \addlegendentry{$\varphi_\mu^*(T)$: optimal final time adjoint}
        \end{axis}
    \end{tikzpicture}
    \caption{Optimal state~$x_\mu^*$ in a space-time plot (top left), optimal final time adjoint~$\varphi_\mu^*(T)$ (top right), optimal control~$u_\mu^*$ (bottom right) and initial~$x_\mu^0$, final~$x_\mu^*(T)$ and target~$x_\mu^T$ states (bottom left) for the parameter~$\mu=(1.5,0.75)$ in the heat equation example.}
    \label{fig:heat-equation-example-results}
\end{figure}
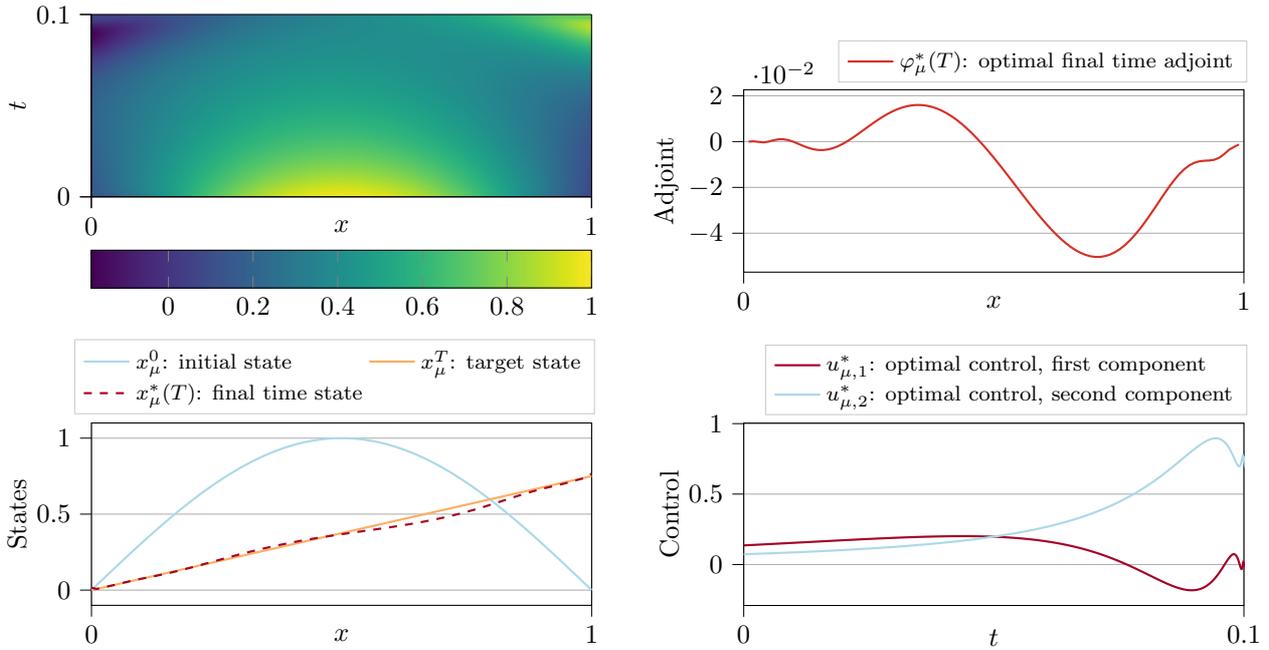
We apply the greedy procedure from~\Cref{subsec:greedy-for-reduced-basis} to construct a reduced basis using~$n_\mathrm{train}=64$ samples on a uniform~$8\times 8$~grid in the parameter space~$\params$. At this point, we emphasize once again that the full optimal control problem is not solved for all parameters in the training set~$\paramstrain$ during~\Cref{alg:offline-greedy}. Instead, the exact solution is only computed for those parameters that are selected by the greedy procedure. It is therefore possible to consider a relatively large training set for the greedy algorithm since the error estimator is rather cheap to evaluate.
\par
The results of the greedy algorithm are depicted in~\Cref{fig:heat-equation-greedy-results}. The greedy tolerance was set to~$\tilde{\varepsilon}=10^{-6}$ which is reached for a reduced basis size of~$N=8$, see~\Cref{fig:heat-equation-greedy-results}. Since no prior knowledge of the constants~$C_\Lambda$ and~$\gamma$ is available for this example, we directly set the tolerance~$\tilde{\varepsilon}$ appearing in the termination criterion in~Line~\ref{lst:offline-termination-criterion} of~\Cref{alg:offline-greedy} instead of the greedy tolerance~$\varepsilon$. We also observe that the overestimation of the true error by the error estimator becomes smaller with a larger reduced basis size and that the estimated maximum error is close to the true maximum error. In~\Cref{fig:heat-equation-singular-values}, the singular values of the optimal final time adjoints associated to the~$n_\mathrm{train}=64$ training parameters are shown in decreasing order. To be more precise, we solved for the optimal final time adjoint for all training parameters, collected these vectors as columns in a matrix, and computed the singular value decomposition of that matrix. Due to equivalences of norms, the respective singular values show (approximately) the order of the decay of the Kolmogorov~$N$-width of the solution manifold~$\mathcal{M}$ of optimal final time adjoints (at least on the training set~$\paramstrain$). One can observe an exponential decay in the singular values until machine precision is reached, see also the blue line marked by~\raisebox{2pt}{\addlegendimageintext{thick, color6}} in the figure. This suggests that the manifold~$\mathcal{M}$ is amenable for approximation by low-dimensional linear subspaces. Hence, the greedy algorithm can be applied in this setting and finds a suitable reduced space according to~\Cref{thm:devore-greedy,thm:weak-greedy-and-approximation-error}.
\par
\begin{figure}[ht]
    \centering
    \begin{minipage}[b]{.49\textwidth}
            \centering
            \begin{tikzpicture}
                \begin{axis}[
                    name=errors,
                    width=.95\textwidth,
                    height=4cm,
                    legend cell align={left},
                    legend style={font=\footnotesize, fill opacity=1, draw opacity=1, text opacity=1, xshift=5pt, yshift=49pt, draw=white!80!black},
                    tick align=outside,
                    tick pos=left,
                    x grid style={white!69.0196078431373!black},
                    xlabel={Greedy step~$k$},
                    xmajorgrids,
                    xmin=-0.1, xmax=8.1,
                    xtick style={color=black},
                    xtick={0,...,8},
                    xticklabels={0,...,8},
                    scaled y ticks = false,
                    y grid style={white!69.0196078431373!black},
                    ymajorgrids,
                    ymin=1e-7, ymax=1e-1,
                    ytick={1e-9,1e-8,1e-7,1e-6,1e-5,1e-4,1e-3,1e-2,1e-1},
                    ymode=log,
                    log basis y={10},
                    ylabel={Greedy errors},
                    ylabel style={yshift=-3pt},
                    ytick style={color=black}
                ]
                    \addplot[very thick, mark=none, color2, dashed, samples=2] coordinates {(-0.1,1e-6) (8.1,1e-6)};
                    \addlegendentry{$\tilde{\varepsilon}=10^{-6}$: greedy tolerance}
                    \addplot[thick, color0, mark=triangle*, mark size=3, mark options={solid, fill opacity=0.5}] table[y index=1] {data/heat_equation/greedy_results.txt};
                    \addlegendentry{$\max_{\mu\in\paramstrain} \eta_\mu(\tilde{\varphi}_\mu^k)$: estimated max.~error}
                    \addplot[thick, color5, mark=*, mark size=3, mark options={solid, fill opacity=0.5}] table[y index=2] {data/heat_equation/greedy_results.txt};
                    \addlegendentry{$\max_{\mu\in\paramstrain} \norm{\varphi_\mu^*(T)-\tilde{\varphi}_\mu^k}_h$: true max.~error}
                \end{axis}
            \end{tikzpicture}
            \caption{Results of the greedy algorithm applied to the heat equation example.}
            \label{fig:heat-equation-greedy-results}
    \end{minipage}%
    \hfill
    \begin{minipage}[b]{.49\textwidth}
            \centering
            \begin{tikzpicture}
                \begin{axis}[
                    name=singular-values,
                    anchor=west,
                    at=(errors.east),
                    xshift=2cm,
                    width=.95\textwidth,
                    height=4cm,
                    legend cell align={left},
                    legend style={font=\footnotesize, fill opacity=1, draw opacity=1, text opacity=1, xshift=5pt, yshift=33pt, draw=white!80!black},
                    tick align=outside,
                    tick pos=left,
                    x grid style={white!69.0196078431373!black},
                    xlabel={Mode number~$N$},
                    xmajorgrids,
                    xmin=-0.1, xmax=65.1,
                    xtick style={color=black},
                    xtick={1,16,32,48,64},
                    xticklabels={1,16,32,48,64},
                    scaled y ticks = false,
                    y grid style={white!69.0196078431373!black},
                    ymajorgrids,
                    ymin=1e-16, ymax=4,
                    ymode=log,
                    ytick={1e-16,1e-12,1e-8,1e-4,1e-0},
                    log basis y={10},
                    ylabel={Singular values},
                    ylabel style={yshift=-3pt},
                    ytick style={color=black},
                    clip mode=individual
                ]
                    \addplot[thick, color1, mark=diamond*, mark size=3, mark options={solid, fill opacity=0.5}] table {data/heat_equation/singular_values_optimal_adjoints.txt};
                    \addlegendentry{Singular values of final time adjoints}
                    \addplot[samples=2, domain=1:24, thick, color6] {10*exp(-1.5*x)};
                    \addlegendentry{$10\cdot\exp(-1.5\cdot N)$}
                \end{axis}
            \end{tikzpicture}
            \caption{Singular value decay of optimal final time adjoints for the heat equation.}
            \label{fig:heat-equation-singular-values}
    \end{minipage}
\end{figure}
Afterwards, the same set of training parameters as for the greedy algorithm has been used for training the different machine learning surrogates. Once again, we can use the relatively large training set consisting of~$n_\mathrm{train}=64$ parameters, because the computation of the training snapshots is already performed during the greedy algorithm and the exact optimal control problem does not have to be solved for all parameters in~$\paramstrain$. All reduced models, the~\RBROM{}, the~\DNNROM{}, the~\VKOGAROM{}, and the~\GPRROM{} have been tested on a set of~$100$ randomly chosen parameters that have not been part of the training set. The results are summarized in~\Cref{tab:heat-equation-results} and presented in more detail in~\Cref{fig:heat-equation-box-plot,fig:heat-equation-errors-parameters}. Here we address two types of errors: the errors in the optimal final time adjoints, i.e.~the deviation of~$\varphi_\mu^*(T)$ from its approximate value, and similarly the errors in the optimal controls. Both, the averaged and maximal errors in the final time adjoint, are well below~$10^{-4}$ for all considered methods. We emphasize at this point that the tolerance~$\tilde{\varepsilon}=10^{-6}$ prescribed for the greedy construction of the reduced basis is only reached by the~\RBROM{}. The approximation of the reduced coefficients by machine learning introduces an additional error that is reflected in a larger error in the final time adjoint. Nevertheless, the machine learning algorithms provide good approximations in particular in terms of the average error. Further, the average errors in the control for all reduced models are of order~$10^{-6}$ or even smaller and therefore the approximate controls are close to the optimal ones. The last two columns of~\Cref{tab:heat-equation-results} show average runtimes of the considered algorithms for computing the (approximate) control and speedups of the reduced models compared to the computation of the exact solution. Here we see a speedup of about~$2$ for the~\RBROM{} while all machine learning surrogates reach speedups of around~$40$. Therefore, the machine learning reduced models outperform the~\RBROM{} in terms of computational efficiency during the online phase. In this example, in particular the~\GPRROM{} constitutes a serious alternative to the~\RBROM{} also with respect to approximation accuracy. We omit a detailed listing of the training times for the machine learning models since they are performed completely offline and amount to a couple of seconds at most. They are negligible compared to the time required for building the reduced basis using the greedy procedure. In~\Cref{fig:heat-equation-errors-parameters} we present the true errors and the estimated errors for the various reduced models over the enumerated test parameter set. The estimated errors are about an order of magnitude larger for the machine learning models compared to the true errors. Accordingly, the efficiency constant of the error estimator~$\eta_\mu$ seems to be relatively moderate, which results in a good error estimate for the machine learning models. It is also visible that the estimator indeed provides a reliable estimate, i.e.~the estimated error is an upper bound for the actual error. For the~\RBROM{} we in particular observe that both, the true and the estimated error in the final time adjoint, are below the prescribed tolerance of~$10^{-6}$ over the whole parameter test set.
\begin{table}[ht]
	\centering
	\resizebox{\columnwidth}{!}{%
		\begin{tabular}{l c c c c d{1.4} d{2.2}}
			\toprule
			{Method} & {Max.~error adjoint} & {Avg.~error adjoint} & {Max.~error control} & {Avg.~error control} & \mc{Avg.~runtime (s)} & \mc{Avg.~speedup} \\ \midrule \midrule
			
			Exact solution & {$-$} & {$-$} & {$-$} & {$-$} & 6.2760 & \mc{$-$} \\
			
			\RBROM{} & {$2.3\cdot 10^{-7}$} & {$5.3\cdot 10^{-8}$} & {$2.3\cdot 10^{-8}$} & {$5.4\cdot 10^{-9}$} & 2.6526 & 2.37 \\
			
            \DNNROM{} & {$2.2\cdot 10^{-5}$} & {$5.8\cdot 10^{-6}$} & {$9.1\cdot 10^{-6}$} & {$2.0\cdot 10^{-6}$} & 0.1623 & 40.33 \\

            \VKOGAROM{} & {$7.0\cdot 10^{-5}$} & {$1.8\cdot 10^{-5}$} & {$2.5\cdot 10^{-5}$} & {$6.9\cdot 10^{-6}$} & 0.1580 & 41.03 \\

            \GPRROM{} & {$1.4\cdot 10^{-5}$} & {$2.2\cdot 10^{-6}$} & {$4.2\cdot 10^{-6}$} & {$7.6\cdot 10^{-7}$} & 0.1572 & 41.40 \\
			\bottomrule
	\end{tabular}}
	\caption{Results of the numerical experiments for the heat equation using the~\RBROM{}, \DNNROM{}, \VKOGAROM{} and~\GPRROM{}.}
	\label{tab:heat-equation-results}
\end{table}

\begin{figure}[ht]
    \begin{minipage}{.48\textwidth}
        \centering
        \begin{tikzpicture}
            \begin{axis}[
                height=6cm,
                boxplot/draw direction=y,
                xtick={1,2,3,4},
                xticklabels={\RBROM{},\DNNROM{},\VKOGAROM{},\GPRROM{}},
                xticklabel style={rotate=-45},
                tick align=outside,
                tick pos=left,
                x grid style={white!69.0196078431373!black},
                xtick style={color=black},
                scaled y ticks=false,
                y grid style={white!69.0196078431373!black},
                ymajorgrids,
                ylabel={Errors in final time adjoint},
                ylabel style={yshift=-3pt},
                ytick style={color=black},
                ymin=1e-9, ymax=1e-4,
                ytick={1e-9,1e-8,1e-7,1e-6,1e-5,1e-4},
                ymode=log,
                log basis y={10},
            ]
                \addplot[color0, boxplot, thick] table [y index=1] {data/heat_equation/analysis_results_errors.txt};
                \addplot[color5, boxplot, thick] table [y index=3] {data/heat_equation/analysis_results_errors.txt};
                \addplot[color6, boxplot, thick] table [y index=5] {data/heat_equation/analysis_results_errors.txt};
                \addplot[color3, boxplot, thick] table [y index=7] {data/heat_equation/analysis_results_errors.txt};
            \end{axis}
        \end{tikzpicture}
        \caption{Box plot showing the statistics of the numerical results for the heat equation using the~\RBROM{}, \DNNROM{}, \VKOGAROM{} and~\GPRROM{}.}
        \label{fig:heat-equation-box-plot}
    \end{minipage}%
    \hfill
    \begin{minipage}{.48\textwidth}
        \centering
        \begin{tikzpicture}
            \begin{axis}[
                height=6cm,
                boxplot/draw direction=y,
                xtick={1,25,50,75,100},
                xticklabels={1,25,50,75,100},
                tick align=outside,
                tick pos=left,
                x grid style={white!69.0196078431373!black},
                xtick style={color=black},
                scaled y ticks=false,
                y grid style={white!69.0196078431373!black},
                ymajorgrids,
                ylabel={Errors in final time adjoint},
                ylabel style={yshift=-3pt},
                ytick style={color=black},
                ymin=1e-9, ymax=1e-3,
                ymode=log,
                ytick={1e-9,1e-8,1e-7,1e-6,1e-5,1e-4,1e-3},
                xlabel={Number of test parameter},
                log basis y={10},
                legend columns=2,
                legend cell align={left},
                legend style={font=\footnotesize, fill opacity=1, draw opacity=1, text opacity=1, yshift=67pt, draw=white!80!black}
            ]
                \addlegendimage{legend image with text={True error}}
                \addlegendentry{}
                \addlegendimage{legend image with text={Estimated error}}
                \addlegendentry{}
                \addplot[color0, thick] table [y index=1] {data/heat_equation/analysis_results_errors.txt};
                \addlegendentry{}
                \addplot[color0, densely dotted, thick] table [y index=2] {data/heat_equation/analysis_results_errors.txt};
                \addlegendentry{\RBROM{}}
                \addplot[color5, thick] table [y index=3] {data/heat_equation/analysis_results_errors.txt};
                \addlegendentry{}
                \addplot[color5, densely dotted, thick] table [y index=4] {data/heat_equation/analysis_results_errors.txt};
                \addlegendentry{\DNNROM{}}
                \addplot[color6, thick] table [y index=5] {data/heat_equation/analysis_results_errors.txt};
                \addlegendentry{}
                \addplot[color6, densely dotted, thick] table [y index=6] {data/heat_equation/analysis_results_errors.txt};
                \addlegendentry{\VKOGAROM{}}
                \addplot[color3, thick] table [y index=7] {data/heat_equation/analysis_results_errors.txt};
                \addlegendentry{}
                \addplot[color3, densely dotted, thick] table [y index=8] {data/heat_equation/analysis_results_errors.txt};
                \addlegendentry{\GPRROM{}}
            \end{axis}
        \end{tikzpicture}
        \caption{True errors and error estimation for the heat equation over the enumerated test parameters in the online phase.}
        \label{fig:heat-equation-errors-parameters}
    \end{minipage}
\end{figure}

\subsection{Test Case 2: Damped Wave Equation}\label{subsec:wave-equation-experiment}
In our second numerical example we consider a parametrized damped wave equation with a Dirichlet control acting on the right boundary of the one-dimensional spatial domain. The problem of interest reads as
\begin{alignat*}{2}
    \partial_{tt} v_\mu(t,y) + \nu\partial_t v_\mu(t,y) - \mu\Delta v_\mu(t,y) &= 0 && \text{for }t\in[0,T],y\in\Omega, \\
    v_\mu(t,0) &= 0 && \text{for }t\in[0,T], \\
    v_\mu(t,1) &= u_{\mu}(t) && \text{for }t\in[0,T], \\
    v_\mu(0,y) &= v_\mu^0(y) = \sin(\pi y) \qquad && \text{for }y\in\Omega, \\
    \partial_t v_\mu(0,y) &= 0 && \text{for }y\in\Omega.
\end{alignat*}
The damping constant is chosen as~$\nu=10$, see the discussion below as well. Further, we choose as final time~$T=1$, the spatial domain~$\Omega=[0,1]$, while the parameter~$\mu$, determining the velocity, ranges within the interval~$\params=[3,10]\subset\setR$. Similarly as in the example of the heat equation we discretize the spatial derivative by means of second-order central finite differences on a grid with~$n_y=100$ inner grid points. For time discretization we apply the Crank-Nicolson scheme with~$n_t=10\cdot n_y$ time steps. We rewrite the discretized damped wave equation as a first order system of dimension~$n=2\cdot n_y=200$. The resulting system matrices are given as
\[
    A_\mu = \begin{bmatrix}[ccc|ccc]
        \hphantom{0}&\hphantom{0}&\hphantom{0}&\hphantom{0}&\hphantom{0}&\hphantom{0} \\
        \multicolumn{3}{c|}{0} & \multicolumn{3}{c}{I} \\
        &&&&& \\
        \cmidrule(lr){1-6}
        &&&&& \\
        \multicolumn{3}{c|}{\frac{\mu}{h^2}\bar{A}} & \multicolumn{3}{c}{-\nu I} \\
        &&&&&
    \end{bmatrix}\in\setR^{n\times n}
    \qquad\text{and}\qquad
    B_\mu = \frac{\mu}{h^2}\begin{bmatrix}
        0 \\
        \vdots \\
        0 \\
        1
    \end{bmatrix}\in\setR^{n\times 1},
\]
where the matrix~$\bar{A}\in\setR^{n_y\times n_y}$ was defined above for the heat equation example in~\cref{equ:system-matrices-heat-equation}, and~$I\in\setR^{n_y\times n_y}$ denotes the identity matrix. Similar to the heat equation example, the discretized system is controllable, i.e.~the Kalman rank condition is fulfilled. Furthermore, the parameter to system matrices mappings are Lipschitz continuous and therefore~\Cref{as:continuity-parameter-to-system-matrices} holds as well.
\par
We prescribe the target~$v_\mu^T(y)=y$ and~$\partial_tv_\mu^T(y)=0$ for all~$y\in\Omega$ and~$\mu\in\params$. Further, we choose the same inner products and norms as in the previous example of the heat equation, see~\cref{equ:inner-product-state-space-example,equ:norm-control-space-example}. In this example we set~$M=10\cdot I\in\setR^{n\times n}$ and~$R=0.1\cdot I\in\setR^{1\times 1}$ for the weighting matrices associated to the different terms in the objective functional. Hence, also~\Cref{as:symmetric-product} is fulfilled in this case.
\par
\Cref{fig:wave-equation-example-results} depicts the optimal final time adjoint, optimal control and corresponding optimal states (as space-time plot and at final time~$T$) for the damped wave equation with parameter~$\mu=5\in\params$. The initial state is driven towards the target state as can be seen in the plots on the left hand side of the figure. As in the heat equation example, the target state is not hit exactly but a deviation is allowed. The velocity component of the states is not shown in the plot since the target and initial velocity are zero and the velocity of the optimal state at final time is of order~$10^{-3}$ or smaller. The top right plot shows the optimal final time adjoint divided into its two components corresponding to the state and the velocity. One can see that the final time adjoint state is quite smooth for both components and that in particular the second component, associated to the velocity, is relatively close to zero. This facilitates a low-dimensional approximation of the final time adjoint. The bottom right plot of~\Cref{fig:wave-equation-example-results} shows the optimal control which exhibits the typical oscillations patterned at the end of the time interval. These oscillations are much more pronounced than the oscillations in the optimal final time adjoint. Hence, it seems advantageous to approximate the adjoint datum instead of the control.
\par
\begin{figure}[ht]
    \centering
    \begin{tikzpicture}
        \begin{axis}[
            name=state-trajectory,
            anchor=north,
            height=4cm,
            width=.48\textwidth,
            legend cell align={left},
            legend style={font=\footnotesize, fill opacity=1, draw opacity=1, text opacity=1, xshift=5pt, yshift=33pt, draw=white!80!black},
            legend columns=2,
            tick align=outside,
            tick pos=left,
            x grid style={white!69.0196078431373!black},
            xlabel={$x$},
            xlabel style={yshift=12pt},
            xmin=0., xmax=0.5,
            xtick style={color=black},
            xtick={0,1},
            xticklabels={0,1},
            scaled y ticks=false,
            y grid style={white!69.0196078431373!black},
            ylabel={$t$},
            ylabel style={yshift=-3pt},
            ytick style={color=black},
            ymin=0., ymax=1.,
            ytick={0,1},
            yticklabels={0,1},
            view={0}{90},
            colorbar horizontal,
            colormap/viridis,
            colorbar style={at={(0,-0.5)}, anchor=south west},
        ]
            \addplot3 [surf, mesh/rows=22, shader=interp, z buffer=sort, colormap/viridis] table [x index=1, y index=0] {data/wave_equation/optimal_state_trajectory.txt};
        \end{axis}

        \begin{axis}[
            name=state,
            anchor=north,
            at=(state-trajectory.south),
            yshift=-3cm,
            width=.48\textwidth,
            height=4cm,
            legend cell align={left},
            legend style={font=\footnotesize, fill opacity=1, draw opacity=1, text opacity=1, xshift=5pt, yshift=33pt, draw=white!80!black},
            legend columns=2,
            tick align=outside,
            tick pos=left,
            x grid style={white!69.0196078431373!black},
            xlabel={$x$},
            xlabel style={yshift=12pt},
            xmajorgrids,
            xmin=0., xmax=0.5,
            xtick style={color=black},
            xtick={0,0.5},
            xticklabels={0,1},
            scaled y ticks=false,
            y grid style={white!69.0196078431373!black},
            ymajorgrids,
            ylabel={States},
            ylabel style={yshift=-3pt},
            ytick style={color=black}
        ]
            \addplot[thick, color5] table[y index=2, restrict x to domain=0:0.5] {data/wave_equation/final_time_state.txt};
            \addlegendentry{$x_\mu^0$: initial state}
            \addplot[thick, color3] table[y index=3, restrict x to domain=0:0.5] {data/wave_equation/final_time_state.txt};
            \addlegendentry{$x_\mu^T$: target state}
            \addplot[thick, color0, dashed] table[y index=1, restrict x to domain=0:0.5] {data/wave_equation/final_time_state.txt};
            \addlegendentry{$x_\mu^*(T)$: final time state}
        \end{axis}

        \begin{axis}[
            name=control,
            anchor=west,
            at=(state.east),
            xshift=2cm,
            width=.48\textwidth,
            height=4cm,
            legend cell align={left},
            legend style={font=\footnotesize, fill opacity=1, draw opacity=1, text opacity=1, xshift=5pt, yshift=20pt, draw=white!80!black},
            tick align=outside,
            tick pos=left,
            x grid style={white!69.0196078431373!black},
            xlabel={$t$},
            xlabel style={yshift=12pt},
            xmajorgrids,
            xmin=0., xmax=1.,
            xtick style={color=black},
            xtick={0,1},
            xticklabels={0,1},
            scaled y ticks = false,
            y grid style={white!69.0196078431373!black},
            ymajorgrids,
            ylabel style={yshift=-3pt},
            ylabel={Control},
            ytick style={color=black},
        ]
            \addplot[thick, color0] table[y index=1] {data/wave_equation/optimal_control.txt};
            \addlegendentry{$u_{\mu}^*$: optimal control}
        \end{axis}

        \begin{axis}[
            name=adjoint,
            anchor=south,
            at=(control.north),
            yshift=1.7cm,
            width=.48\textwidth,
            height=4cm,
            legend cell align={left},
            legend style={font=\footnotesize, fill opacity=1, draw opacity=1, text opacity=1, xshift=5pt, yshift=35pt, draw=white!80!black},
            tick align=outside,
            tick pos=left,
            x grid style={white!69.0196078431373!black},
            xlabel={$x$},
            xlabel style={yshift=12pt},
            xmajorgrids,
            xmin=0., xmax=0.5,
            xtick style={color=black},
            xtick={0,0.5},
            xticklabels={0,1},
            scaled y ticks = false,
            y grid style={white!69.0196078431373!black},
            ymajorgrids,
            ylabel={Adjoint},
            ylabel style={yshift=-3pt},
            ytick style={color=black},
        ]
            \addplot[thick, color1] table[restrict x to domain=0:0.5] {data/wave_equation/optimal_final_time_adjoint.txt};
            \addlegendentry{$\varphi_\mu^*(T)$: optimal final time adjoint, first component}
            \addplot[thick, color5] table[restrict x to domain=0:0.5, x expr=\thisrow{x}-0.5] {data/wave_equation/optimal_final_time_adjoint.txt};
            \addlegendentry{$\varphi_\mu^*(T)$: optimal final time adjoint, second component}
        \end{axis}
    \end{tikzpicture}
    \caption{Optimal state~$x_\mu^*$ in a space-time plot (top left), optimal final time adjoint~$\varphi_\mu^*(T)$ with the first component associated to the state and second component associated to the velocity (top right), optimal control~$u_\mu^*$ (bottom right) and initial~$x_\mu^0$, final~$x_\mu^*(T)$ and target~$x_\mu^T$ states (bottom left) for the parameter~$\mu=5$ in the damped wave equation example. The velocities are not shown in the plots of the states on the left-hand side, as they are all close to being zero. For the optimal final time adjoint, both components are plotted since these have to be approximated by the reduced space.}
    \label{fig:wave-equation-example-results}
\end{figure}
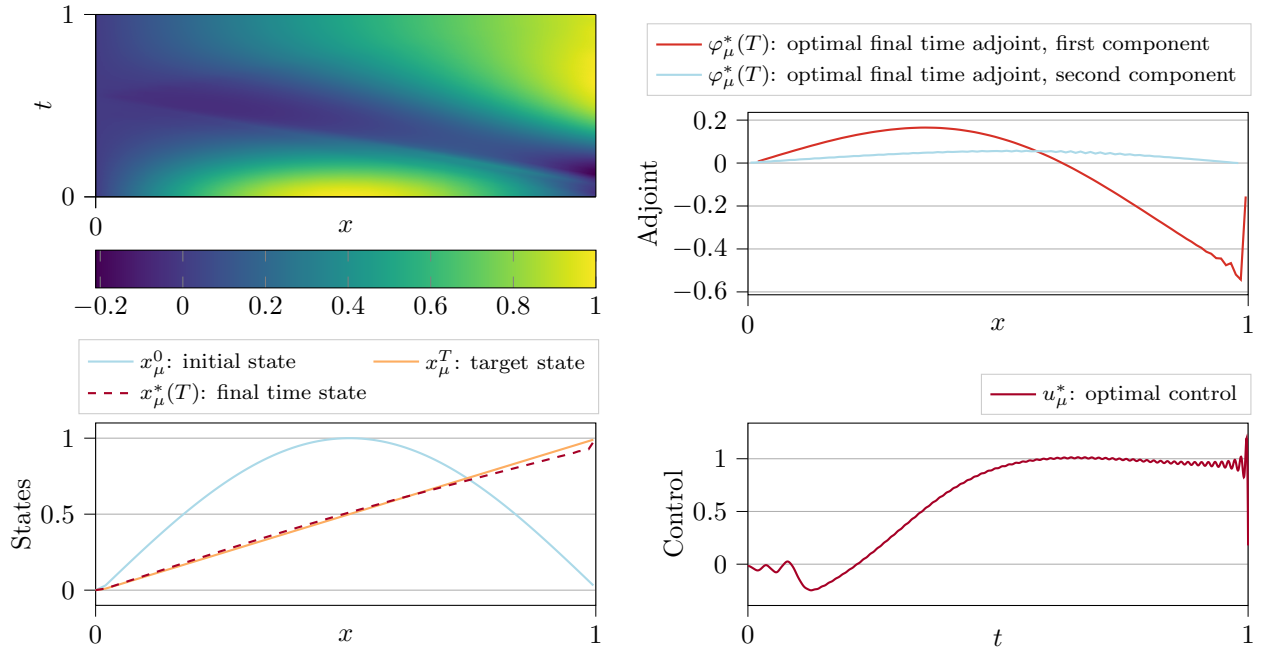
The plots in~\Cref{fig:wave-equation-example-results} depict the solution for a single parameter value~$\mu=5$. In order to deal with the full range of parameter dependent problems, we apply the greedy algorithm~\cref{alg:offline-greedy} with a training set consisting of~$n_\mathrm{train}=50$ training samples from a uniform grid on the parameter domain~$\params$. The greedy tolerance is set to~$\tilde{\varepsilon}=10^{-2}$. Similar to the previous example we have no prior knowledge of the constants~$C_\Lambda$ and~$\gamma$ and therefore directly select~$\tilde{\varepsilon}$ instead of~$\varepsilon$.
\par
\Cref{fig:wave-equation-greedy-results} shows the estimated maximum and true maximum errors over the greedy steps until the greedy tolerance~$\tilde{\varepsilon}$ is reached for a reduced basis size of~$N=18$. The parameters selected by the greedy algorithm are also shown in~\Cref{fig:wave-equation-errors-parameters} (marked as crosses~\raisebox{.24ex}{\addlegendimageintext{mark=x, only marks, color1}}) and are distributed almost uniformly across the parameter range with a slight clustering at the smaller parameter values. One can observe in~\Cref{fig:wave-equation-greedy-results} that the estimator is more pessimistic in this example than in the heat equation problem and overestimates the true error by about two orders of magnitude. Nevertheless, the error estimator still provides a reliable estimate of the error, i.e.~the estimated error is an upper bound for the true error. However, the greedy procedure is not stopped with a reduced size of~$5$ although the true error would have been already below the prescribed tolerance for such a basis size. This behavior reveals the fact that the constant~$C_\Lambda$, which determines the efficiency of the error estimator, is larger for this example. This is to be expected, as the dissipation rate is weaker in this case than it was for the heat equation example.
\par
In order to explore how the dissipation rate influences the problem, we analyze its dependence on the damping constant~$\nu$. \Cref{fig:wave-equation-singular-values} shows the singular values for the optimal final time adjoints over the training set for different values of the damping constant~$\nu$. The singular value decay is rather slow in the case of the undamped wave equation, i.e.~for~$\nu=0$, see also~\cite{greif2019decay}. The same behavior emerges for a small value of the damping constant like~$\nu=5$, which means that also the decay of the Kolmogorov~$N$-width of the manifold of optimal final time adjoints is slow in these cases. Thus, a relatively large reduced space would be necessary to obtain a sufficiently accurate reduced model. In that regard, the restriction to linear subspaces (constructed by the greedy algorithm in this case) can be seen as a limitation of the general methodology since it might not be applicable to all control systems. The introduction of a larger damping constant such as~$\nu=10$ or even~$\nu=100$, however, leads to a faster decay of the singular values (see the curves marked as~\addlegendimageintext{thick, color1, mark=diamond*, mark size=3, mark options={solid, fill opacity=0.5}} and~\addlegendimageintext{thick, color4, mark=triangle*, mark size=3, mark options={solid, fill opacity=0.5}} in~\Cref{fig:wave-equation-singular-values}) and thus makes the solution manifold amenable for approximation by linear subspaces of small dimension. For a large damping constant of~$\nu=100$ we observe an exponential decay of the singular values (similar to the heat equation example) up to a reduced basis size of around~$30$. For a moderate damping constant of~$\nu=10$, the decay is also exponential at least for the first couple of modes. We should nevertheless mention that also in the case of the damped wave equation with~$\nu=10$, the decay of the singular values slows down quite rapidly after about~$7$ modes. Hence, to reach very small approximation errors, still a large reduced basis would be required. Since we are not interested in an example where the dissipativity, resulting from a large damping constant, dominates the solution behavior, we choose a value of~$\nu=10$ and the larger greedy tolerance of~$\tilde{\varepsilon}=10^{-2}$.
\par
\begin{figure}
    \centering
    \begin{minipage}[b]{.49\textwidth}
        \centering
        \begin{tikzpicture}
            \begin{axis}[
                name=errors,
                width=.95\textwidth,
                height=6cm,
                legend cell align={left},
                legend style={font=\footnotesize, fill opacity=1, draw opacity=1, text opacity=1, xshift=5pt, yshift=49pt, draw=white!80!black},
                tick align=outside,
                tick pos=left,
                x grid style={white!69.0196078431373!black},
                xlabel={Greedy step~$k$},
                xmajorgrids,
                xmin=-0.1, xmax=18.1,
                xtick style={color=black},
                xtick={0,...,18},
                xticklabels={0,...,18},
                scaled y ticks = false,
                y grid style={white!69.0196078431373!black},
                ymajorgrids,
                ymin=1e-4, ymax=1e-0,
                ytick={1e-4,1e-3,1e-2,1e-1,1e-0},
                ymode=log,
                log basis y={10},
                ylabel={Greedy errors},
                ylabel style={yshift=-3pt},
                ytick style={color=black}
            ]
                \addplot[very thick, mark=none, color2, dashed, samples=2] coordinates {(-0.1,1e-2) (18.1,1e-2)};
                \addlegendentry{$\tilde{\varepsilon}=10^{-2}$ greedy tolerance}
                \addplot[thick, color0, mark=triangle*, mark size=3, mark options={solid, fill opacity=0.5}] table[y index=1] {data/wave_equation/greedy_results.txt};
                \addlegendentry{$\max_{\mu\in\paramstrain} \eta_\mu(\tilde{\varphi}_\mu^k)$: estimated max.~error}
                \addplot[thick, color5, mark=*, mark size=3, mark options={solid, fill opacity=0.5}] table[y index=2] {data/wave_equation/greedy_results.txt};
                \addlegendentry{$\max_{\mu\in\paramstrain} \norm{\varphi_\mu^*(T)-\tilde{\varphi}_\mu^k}$: true max.~error}
            \end{axis}
        \end{tikzpicture}
        \caption{Results of the greedy algorithm applied to the damped wave equation example.}
        \label{fig:wave-equation-greedy-results}
    \end{minipage}%
    \hfill
    \begin{minipage}[b]{.49\textwidth}
        \begin{tikzpicture}
            \begin{axis}[
                name=singular-values,
                anchor=west,
                at=(errors.east),
                xshift=2cm,
                width=.95\textwidth,
                height=6cm,
                legend cell align={left},
                legend style={font=\footnotesize, fill opacity=1, draw opacity=1, text opacity=1, xshift=5pt, yshift=57pt, draw=white!80!black},
                tick align=outside,
                tick pos=left,
                x grid style={white!69.0196078431373!black},
                xlabel={Mode number~$N$},
                xmajorgrids,
                xmin=-0.1, xmax=51.1,
                xtick style={color=black},
                xtick={1,10,20,...,50},
                xticklabels={1,10,20,...,50},
                scaled y ticks = false,
                y grid style={white!69.0196078431373!black},
                ymajorgrids,
                ymin=1e-10, ymax=1e2,
                ytick={1e2, 1e-1, 1e-4, 1e-7, 1e-10},
                ymode=log,
                log basis y={10},
                ylabel={Singular values},
                ylabel style={yshift=-3pt},
                ytick style={color=black},
                clip mode=individual
            ]
                \addplot[thick, color6, mark=*, mark size=3, mark options={solid, fill opacity=0.5}] table[y index=1] {data/wave_equation/singular_values_optimal_adjoints_comparison.txt};
                \addlegendentry{undamped wave equation with~$\nu=0$}
                \addplot[thick, color5, mark=square*, mark size=3, mark options={solid, fill opacity=0.5}] table[y index=3] {data/wave_equation/singular_values_optimal_adjoints_comparison.txt};
                \addlegendentry{damped wave equation with~$\nu=5$}
                \addplot[thick, color1, mark=diamond*, mark size=3, mark options={solid, fill opacity=0.5}] table[y index=4] {data/wave_equation/singular_values_optimal_adjoints_comparison.txt};
                \addlegendentry{damped wave equation with~$\nu=10$}
                \addplot[thick, color4, mark=triangle*, mark size=3, mark options={solid, fill opacity=0.5}] table[y index=6] {data/wave_equation/singular_values_optimal_adjoints_comparison.txt};
                \addlegendentry{damped wave equation with~$\nu=100$}
            \end{axis}
        \end{tikzpicture}
        \caption{Singular value decay of optimal final time adjoints for the damped wave equation.}
        \label{fig:wave-equation-singular-values}
    \end{minipage}
\end{figure}
After constructing a reduced basis using the greedy procedure from~\Cref{subsec:greedy-for-reduced-basis}, the machine learning reduced models described in~\Cref{sec:reduced-order-machine-learning} are trained on the same~$n_\mathrm{train}=50$ training parameters that were already used in the greedy algorithm. For testing purposes, we draw~$100$ random values from the parameter range~$\params$ that were not contained in the training set and evaluate the performance of the reduced models on the test parameter set. The results of the online phase are presented in~\Cref{tab:wave-equation-results}. Furthermore, the error statistics for the reduced models are summarized in~\Cref{fig:wave-equation-box-plot}, whereas the errors of the reduced models with respect to the test parameters are shown in detail in~\Cref{fig:wave-equation-errors-parameters}. First of all, we observe an enormous speedup reached by all machine learning models compared to the exact solution of the optimal control and also compared to the~\RBROM{}. The average speedup of the latter model is about~$21$, while the average speedup of the three considered machine learning surrogates is around~$730$. In total, the average runtime could be reduced from about~$229$ seconds for the exact solution to roughly~$0.32$ seconds using the combined machine learning/reduced basis surrogates. In terms of approximation accuracy, we see that the average error in the final time adjoint is about one order of magnitude larger for the machine learning reduced models compared to the~\RBROM{}. For the control we observe that the~\DNNROM{}, the~\VKOGAROM{}, and the~\GPRROM{} all reach average errors of the order~$10^{-4}$. The box plot in~\Cref{fig:wave-equation-box-plot} and also the true errors in~\Cref{fig:wave-equation-errors-parameters} reveal that the~\VKOGAROM{} and the~\GPRROM{} actually suffer from several outliers, in particular at the boundaries of the parameter domain, which limit their average behavior. However, on most parts of the parameter domain the error of those two methods is close to the error of the~\RBROM{}. The~\DNNROM{} on the other hand produces errors that are about one order of magnitude larger on the whole parameter domain but with less outliers. It might be possible to reduce these outliers and further improve the performance of the machine learning models by tuning their hyper-parameters. However, such an optimization of the machine learning models is beyond the scope of this article and not the focus of this work. To adaptively adjust for the outliers, we remark that the error estimator also captures the outliers well and might therefore help to detect areas in the parameter domain where additional training data should be generated to improve the machine learning models. At this point, we also would like to advert to the relatively large overestimation of the error that can be observed in~\Cref{fig:wave-equation-errors-parameters} and was already present in~\Cref{fig:wave-equation-greedy-results}. Also for the most accurate reduced model, i.e.~the~\RBROM{}, the estimated error is about two orders of magnitude larger than the true error. Nevertheless, the curves of estimated errors follow the true errors quite closely in terms of error variation with respect to the parameter, although being two orders of magnitude larger.
\begin{table}[ht]
	\centering
	\resizebox{\columnwidth}{!}{%
		\begin{tabular}{l c c c c d{3.4} d{3.2}}
			\toprule
			{Method} & {Max.~error adjoint} & {Avg.~error adjoint} & {Max.~error control} & {Avg.~error control}& \mc{Avg.~runtime (s)} & \mc{Avg.~speedup} \\ \midrule \midrule
			
			Exact solution & {$-$} & {$-$} & {$-$} & {$-$} & 228.8106 & \mc{$-$} \\
			
			\RBROM{} & {$3.0\cdot 10^{-4}$} & {$4.7\cdot 10^{-5}$} & {$1.3\cdot 10^{-4}$} & {$2.3\cdot 10^{-5}$} & 10.8503 & 21.10 \\
			
            \DNNROM{} & {$3.8\cdot 10^{-3}$} & {$3.8\cdot 10^{-4}$} & {$4.6\cdot 10^{-3}$} & {$7.0\cdot 10^{-4}$} & 0.3230 & 731.48 \\

            \VKOGAROM{} & {$2.0\cdot 10^{-2}$} & {$5.7\cdot 10^{-4}$} & {$5.7\cdot 10^{-3}$} & {$2.0\cdot 10^{-4}$} & 0.3226 & 733.80 \\

            \GPRROM{} & {$8.9\cdot 10^{-3}$} & {$3.9\cdot 10^{-4}$} & {$1.1\cdot 10^{-2}$} & {$5.3\cdot 10^{-4}$} & 0.3359 & 708.55 \\
			\bottomrule
	\end{tabular}}
	\caption{Results of the numerical experiments for the damped wave equation using the~\RBROM{}, \DNNROM{}, \VKOGAROM{} and~\GPRROM{}.}
	\label{tab:wave-equation-results}
\end{table}

\begin{figure}[ht]
    \begin{minipage}{.47\textwidth}
        \centering
        \begin{tikzpicture}
            \begin{axis}[
                height=6cm,
                boxplot/draw direction=y,
                xtick={1,2,3,4},
                xticklabels={\RBROM{},\DNNROM{},\VKOGAROM{},\GPRROM{}},
                xticklabel style={rotate=-45},
                tick align=outside,
                tick pos=left,
                x grid style={white!69.0196078431373!black},
                xtick style={color=black},
                scaled y ticks=false,
                y grid style={white!69.0196078431373!black},
                ymajorgrids,
                ylabel={Errors in final time adjoint},
                ylabel style={yshift=-3pt},
                ytick style={color=black},
                ymin=1e-6, ymax=2e-2,
                ymode=log,
                log basis y={10},
            ]
                \addplot[color0, boxplot, thick] table [y index=2] {data/wave_equation/analysis_results_errors.txt};
                \addplot[color5, boxplot, thick] table [y index=4] {data/wave_equation/analysis_results_errors.txt};
                \addplot[color6, boxplot, thick] table [y index=6] {data/wave_equation/analysis_results_errors.txt};
                \addplot[color3, boxplot, thick] table [y index=8] {data/wave_equation/analysis_results_errors.txt};
            \end{axis}
        \end{tikzpicture}
        \caption{Box plot showing the statistics of the numerical results for the damped wave equation using the~\RBROM{}, \DNNROM{}, \VKOGAROM{} and~\GPRROM{}.}
        \label{fig:wave-equation-box-plot}
    \end{minipage}%
    \hfill
    \begin{minipage}{.47\textwidth}
        \centering
        \begin{tikzpicture}
            \begin{axis}[
                height=6cm,
                boxplot/draw direction=y,
                xtick={1,25,50,75,100},
                xticklabels={1,25,50,75,100},
                tick align=outside,
                tick pos=left,
                x grid style={white!69.0196078431373!black},
                xtick style={color=black},
                scaled y ticks=false,
                y grid style={white!69.0196078431373!black},
                ymajorgrids,
                ylabel={Errors in final time adjoint},
                ylabel style={yshift=-3pt},
                ytick style={color=black},
                ymin=1e-6, ymax=25,
                ymode=log,
                ytick={1e-6,1e-5,1e-4,1e-3,1e-2,1e-1,1e-0,1e+1},
                xlabel={$\mu$},
                xtick={3,...,10},
                xticklabels={3,...,10},
                log basis y={10},
                legend columns=2,
                legend cell align={left},
                legend style={font=\footnotesize, fill opacity=1, draw opacity=1, text opacity=1, yshift=75pt, draw=white!80!black}
            ]
                \addlegendimage{legend image with text={True error}}
                \addlegendentry{}
                \addlegendimage{legend image with text={Estimated error}}
                \addlegendentry{}
                \addplot[color0, thick] table [x index=1, y index=2] {data/wave_equation/analysis_results_errors.txt};
                \addlegendentry{}
                \addplot[color0, densely dotted, thick] table [x index=1, y index=3] {data/wave_equation/analysis_results_errors.txt};
                \addlegendentry{\RBROM{}}
                \addplot[color5, thick] table [x index=1, y index=4] {data/wave_equation/analysis_results_errors.txt};
                \addlegendentry{}
                \addplot[color5, densely dotted, thick] table [x index=1, y index=5] {data/wave_equation/analysis_results_errors.txt};
                \addlegendentry{\DNNROM{}}
                \addplot[color6, thick] table [x index=1, y index=6] {data/wave_equation/analysis_results_errors.txt};
                \addlegendentry{}
                \addplot[color6, densely dotted, thick] table [x index=1, y index=7] {data/wave_equation/analysis_results_errors.txt};
                \addlegendentry{\VKOGAROM{}}
                \addplot[color3, thick] table [x index=1, y index=8] {data/wave_equation/analysis_results_errors.txt};
                \addlegendentry{}
                \addplot[color3, densely dotted, thick] table [x index=1, y index=9] {data/wave_equation/analysis_results_errors.txt};
                \addlegendentry{\GPRROM{}}
                \addlegendimage{legend image with text={}}
                \addlegendentry{}
                \addplot[mark=x, only marks, color1] table [x index=0, y expr=2e-6] {data/wave_equation/greedy_selected_parameters.txt};
                \addlegendentry{Selected parameters}
            \end{axis}
        \end{tikzpicture}
        \caption{True errors and error estimation for the damped wave equation over the test parameters in the online phase. The parameter values selected during the greedy algorithm are shown as crosses.}
        \label{fig:wave-equation-errors-parameters}
    \end{minipage}
\end{figure}

    \section{Conclusion and Outlook}\label{sec:conclusion-outlook}
In this contribution we considered parameterized optimal control problems where the objective functional penalizes a (weighted) deviation from a target state and a large control energy. To solve such optimal control problems fast for several values of the parameter, for instance in a real-time application or many-query scenario, we first extended the greedy control procedure, previously developed in the controllability context, either the exact~\cite{lazar2016greedy}, or the approximate one~\cite{lazar2023greedy}, to this setting. In contrast to the penalisation approach applied in this work,  the desired target state is considered as a constraint in the cited papers. However, both approaches allow for an application of the Hilbert uniqueness method. In particular, the optimal control, i.e.~the one that minimizes the objective functional, can be completely described by the optimal final state of the adjoint problem. For this reason, the  greedy algorithm builds a reduced basis for the manifold of optimal final time adjoint states over the parameter set, and not for the manifold of optimal controls which is usually more complex and difficult to describe.
\par
Secondly, we applied machine learning approaches, such as deep neural networks, kernel methods, and Gaussian process regression, to learn the reduced coefficients as a function of the parameter. We derived error bounds for the greedy approximation and proved that the proposed algorithm is indeed a weak greedy algorithm. Furthermore, we also showed how to bound the error of the machine learning with respect to the greedy error and the error in approximating the coefficients. A comparison of the computational costs reveals the enormous potential in reducing the computational effort by applying machine learning in our scenario. The numerical experiments exhibit that the machine learning surrogates can accurately approximate the reduced coefficients while providing a massive speedup compared to the exact solution of the optimal control problem. In addition, due to a posteriori error estimates we derived, it is possible to obtain a reliable bound on the error of the machine learning results in a cheap manner without computing the exact solution.
\par
The numerical experiments in this paper are run for control systems whose underlying dynamics (governed by the heat and the damped wave operator) are dissipative, which is crucial for the efficiency of the greedy algorithm. Namely, a conservative system supports large, non-linear variations of optimal final time adjoints with respect to the parameter~\cite{greif2019decay}, which requires a relatively big reduced basis to obtain a sufficiently accurate reduced model (cf.~the discussion on the damping constant in~\Cref{subsec:wave-equation-experiment}). In this way, the developed methods can be similarly (and successfully) applied to other dissipative systems, such as advection-diffusion-reaction ones. Although our examples are based on PDEs in one space dimension, the same procedure can be applied in a higher dimensions cases as well. Here we do not expect difficulties, apart from those conditioned by the efficiency of standard numerical methods that have to be employed in the offline phase.
\par
As discussed extensively in~\Cref{sec:reduced-order-machine-learning}, many different machine learning algorithms can be integrate in the approach. The only requirement on the machine learning surrogate is that it approximates vector-valued functions and can be trained in a supervised manner, i.e.~by providing samples of the function to approximate. A theoretical investigation of the machine learning models in terms of approximation quality of the parameter to solution map could be a topic of future research. For instance for the kernel methods applied above, theoretical results on the approximation properties and bounds for the errors exist (e.g.~\cite{santin2017convergence}) that allow for a rigorous a priori error bound that involves only the $P$-greedy tolerance~$\varepsilon_P$ and the norm of the parameter to solution map in the reproducing kernel Hilbert space (under the assumption that this mapping is indeed contained in the reproducing kernel Hilbert space).
\par
To further speedup the computations of the reduced models during the online phase (in particular Lines~\ref{lst:online-compute-final-time-states}, \ref{lst:online-time-dependent-adjoint}, and~\ref{lst:online-greedy-control} in~\Cref{alg:online-greedy} and Lines~\ref{lst:online-ml-adjoint-equation} and~\ref{lst:online-ml-control} in~\Cref{alg:online-machine-learning}), an additional hyper-reduction, that accelerates solving the equation for the time-dependent adjoint variable, may be applied. Similarly, the error estimator could be approximated by replacing the system in~\cref{equ:optimality-system-odes} by a reduced system. To guarantee a reliable and efficient error estimator, theoretical investigations of such an approach would be indispensable. However, we should mention that replacing the optimality system in~\cref{equ:optimality-system-main} by a surrogate would also speedup the exact computation of the optimal final time adjoint. Several approaches considering parametrized model order reduction of control systems have been suggested in the last decades, see~\cite{benner2015survey} for a survey of methods.
\par
An application of the adaptive model hierarchy from~\cite{haasdonk2023certified} as described in~\Cref{rem:adaptive-model-hierarchy} to the parametrized optimal control setting could be an interesting extension and combination of the approaches. The error estimators and reduced models developed in this paper would serve as the main ingredients to an adaptive model hierarchy with certified results for parametrized optimal control problems.
\par
Finally, the approaches discussed in this paper can be extended to other classes of optimal control problems. As an example, one might  explore a generalization of the method to linear time-variant systems with time-dependent parameters, objective functionals that are not necessarily quadratic with respect to the state and the control, or to problems with additional constraints such as bounds on the control variables. Instead of a target state, it may also be of interest to include an output quantity and an objective functional measuring a deviation from a target output. Similarly, one could penalise a deviation of the whole trajectory from some desired one. This would require the study of the coupled optimality system, and the optimal feedback controls obtained by solving the corresponding Riccati equations should be explored. In such a setting, the greedy approach should probably be accompanied by POD, which is common when constructing a reduced basis for a manifold of time-dependent functions~\cite{hesthaven2016certified}.

    \printbibliography

@article{cohen2016kolmogorov,
    title   = {Kolmogorov widths under holomorphic mappings},
    journal = {IMA Journal of Numerical Analysis},
    volume  = {36},
    number  = {1},
    pages   = {1-12},
    year    = {2016},
    doi     = {10.1093/imanum/dru066},
    author  = {Albert Cohen and Ronald DeVore}
}

@article{cohen2015approximation,
    author  = {Albert Cohen and Ronald DeVore},
    title   = {Approximation of high-dimensional parametric PDEs},
    journal = {Acta Numerica},
    volume  = {24},
    number  = {1},
    pages   = {1-159},
    year    = {2015},
    doi     = {10.1017/S0962492915000033}
}

@article{dalsanto2020data,
    title = {Data driven approximation of parametrized PDEs by reduced basis and neural networks},
    author = {Dal Santo, Niccolo and Deparis, Simone and Pegolotti, Luca},
    publisher = {ACADEMIC PRESS INC ELSEVIER SCIENCE},
    journal = {Journal Of Computational Physics},
    address = {San Diego},
    volume = {416},
    pages = {109550},
    year = {2020},
    doi = {10.1016/j.jcp.2020.109550}
}

@inproceedings{molinari2021iterative,
  title={Iterative regularization for convex regularizers},
  author={Molinari, Cesare and Massias, Mathurin and Rosasco, Lorenzo and Villa, Silvia},
  booktitle={International conference on artificial intelligence and statistics},
  pages={1684--1692},
  year={2021},
  editor = 	 {Banerjee, Arindam and Fukumizu, Kenji},
  volume = 	 {130},
  series = 	 {Proceedings of Machine Learning Research},
  publisher =    {PMLR},
}

@book{peypouquet2015convex,
  title={Convex optimization in normed spaces: theory, methods and examples},
  doi = {10.1007/978-3-319-13710-0},
  year = {2015},
  publisher = {Springer International Publishing},
  author = {Juan Peypouquet}
}

@incollection{ballarin2022spacetime,
    title = {Chapter 9 - Space-time POD-Galerkin approach for parametric flow control},
    editor = {Emmanuel Trélat and Enrique Zuazua},
    series = {Handbook of Numerical Analysis},
    publisher = {Elsevier},
    volume = {23},
    pages = {307-338},
    year = {2022},
    booktitle = {Numerical Control: Part A},
    issn = {1570-8659},
    doi = {10.1016/bs.hna.2021.12.009},
    author = {Francesco Ballarin and Gianluigi Rozza and Maria Strazzullo}
}

@article{daniel2020model,
	Author = {Daniel, Thomas and Casenave, Fabien and Akkari, Nissrine and Ryckelynck, David},
	Journal = {Advanced Modeling and Simulation in Engineering Sciences},
	Number = {1},
	Pages = {16},
	Title = {Model order reduction assisted by deep neural networks ({ROM}-net)},
	Volume = {7},
	Year = {2020},
    doi = {10.1186/s40323-020-00153-6},
}

@article{lazar2023greedy,
    title = {Greedy search of optimal approximate solutions},
    journal = {Pure and Applied Functional Analysis},
    volume = {8},
    number = {2},
    pages = {547-564},
    year = {2023},
    author = {Martin Lazar and Enrique Zuazua}
}

@incollection{lazar2022control,
    title = {Chapter 8 - Control of parameter dependent systems},
    editor = {Emmanuel Trélat and Enrique Zuazua},
    series = {Handbook of Numerical Analysis},
    publisher = {Elsevier},
    volume = {23},
    pages = {265-306},
    year = {2022},
    booktitle = {Numerical Control: Part A},
    issn = {1570-8659},
    doi = {10.1016/bs.hna.2021.12.008},
    author = {Martin Lazar and Jérôme Lohéac}
}

@article{lazar2016greedy,
    title = {Greedy controllability of finite dimensional linear systems},
    journal = {Automatica},
    volume = {74},
    pages = {327-340},
    year = {2016},
    doi = {10.1016/j.automatica.2016.08.010},
    author = {Martin Lazar and Enrique Zuazua}
}

@article{hesthaven2018nonintrusive,
    AUTHOR = {Hesthaven, Jan S. and Ubbiali, Stefano},
    TITLE = {Non-intrusive reduced order modeling of nonlinear problems using neural networks},
    JOURNAL = {Journal of Computational Physics},
    VOLUME = {363},
    YEAR = {2018},
    PAGES = {55--78},
    MRCLASS = {65N30 (65N22)},
    MRNUMBER = {3784416},
    DOI = {10.1016/j.jcp.2018.02.037},
}

@article{wang2019nonintrusive,
    author   = {Wang, Qian and Hesthaven, Jan S. and Ray, Deep},
    title    = {Non-intrusive reduced order modeling of unsteady flows using artificial neural networks with application to a combustion problem},
    year     = {2019},
    volume   = {384},
    pages    = {289--307},
    doi      = {10.1016/j.jcp.2019.01.031},
    journal = {Journal of Computational Physics},
    mrclass  = {80A25 (65M99)},
    mrnumber = {3920924},
}

@article{devore2013greedy,
    author={DeVore, Ronald and Petrova, Guergana and Wojtaszczyk, Przemyslaw},
    title={Greedy Algorithms for Reduced Bases in Banach Spaces},
    journal={Constructive Approximation},
    year={2013},
    volume={37},
    number={3},
    pages={455-466},
    doi={10.1007/s00365-013-9186-2},
}

@article{haasdonk2023certified,
    author = {Haasdonk, Bernard and Kleikamp, Hendrik and Ohlberger, Mario and Schindler, Felix and Wenzel, Tizian},
    title = {A New Certified Hierarchical and Adaptive RB-ML-ROM Surrogate Model for Parametrized PDEs},
    journal = {SIAM Journal on Scientific Computing},
    volume = {45},
    number = {3},
    pages = {A1039-A1065},
    year = {2023},
    doi = {10.1137/22M1493318}
}

@article{wenzel2023application,
      title={Application of Deep Kernel Models for Certified and Adaptive RB-ML-ROM Surrogate Modeling}, 
      author={Tizian Wenzel and Bernard Haasdonk and Hendrik Kleikamp and Mario Ohlberger and Felix Schindler},
      year={2023},
      eprint={2302.14526},
      archivePrefix={arXiv},
      primaryClass={math.NA}
}

@inproceedings{kmet2011neural,
    author={Kmet, Tibor},
    editor={Honkela, Timo
    and Duch, W{\l}odzis{\l}aw
    and Girolami, Mark
    and Kaski, Samuel},
    title={Neural Network Solution of Optimal Control Problem with Control and State Constraints},
    booktitle={Artificial Neural Networks and Machine Learning -- ICANN 2011},
    year={2011},
    publisher={Springer Berlin Heidelberg},
    address={Berlin, Heidelberg},
    pages={261--268},
    doi={10.1007/978-3-642-21738-8_34},
    isbn={978-3-642-21738-8}
}

@article{dede2012reduced,
    author={Ded{\`e}, Luca},
    title={Reduced Basis Method and Error Estimation for Parametrized Optimal Control Problems with Control Constraints},
    journal={SIAM Journal of Scientific Computing},
    year={2012},
    volume={50},
    number={2},
    pages={287-305},
    doi={10.1007/s10915-011-9483-5},
}

@article{ehring2023hermite,
      title={Hermite kernel surrogates for the value function of high-dimensio\-nal nonlinear optimal control problems}, 
      author={Tobias Ehring and Bernard Haasdonk},
      year={2023},
      eprint={2305.06122},
      archivePrefix={arXiv},
      primaryClass={math.OC}
}

@book{wendland2005scattered,
    title = {Scattered {D}ata {A}pproximation},
    publisher = {Cambridge University Press},
    year = {2005},
    author = {Wendland, Holger},
    volume = {17},
    series = {Cambridge Monographs on Applied and Computational Mathematics},
    address = {Cambridge},
    isbn = {978-0521-84335-5},
    doi = {10.1017/CBO9780511617539},
}

@book{hesthaven2016certified,
    title = {Certified Reduced Basis Methods for Pa\-ra\-me\-trized Partial Differential Equations},
    publisher = {Springer Cham},
    year = {2016},
    author = {Jan S. Hesthaven and Gianluigi Rozza and Benjamin Stamm},
    address={New York},
    series = {SpringerBriefs in Mathematics},
    isbn = {978-3-319-22470-1},
    doi = {10.1007/978-3-319-22470-1},
}

@incollection{santin2021kernel,
    author       = {Santin, Gabriele and Haasdonk, Bernard},
    title        = {Kernel Methods for Surrogate Modeling},
    booktitle    = {Model Order Reduction},
    year         = {2021},
    editor       = {Benner, Peter and Grivet-Talocia, Stefano and Quarteroni, Alfio and Rozza, Gianluigi and Schilders, Wil and Silveira, Luís Miguel},
    booksubtitle = {System- and Data-Driven Methods and Algorithms},
    volume       = {2},
    publisher    = {De Gruyter},
    doi          = {10.1515/9783110498967-009},
}

@article{wenzel2021novel,
    title = {A novel class of stabilized greedy kernel approximation algorithms: Convergence, stability and uniform point distribution},
    journal = {Journal of Approximation Theory},
    volume = {262},
    pages = {105508},
    year = {2021},
    doi = {10.1016/j.jat.2020.105508},
    author = {Tizian Wenzel and Gabriele Santin and Bernard Haasdonk}
}

@article{santin2017convergence,
    title = {Convergence rate of the data-independent P-greedy algorithm in kernel-based approximation},
    volume = {10},
    DOI = {10.14658/pupj-drna-2017-Special_Issue-9},
    number = {6},
    journal = {Dolomites Research Notes on Approximation},
    publisher = {Padova University Press},
    author = {Santin, Gabriele and Haasdonk, Bernard},
    year = {2017},
    pages = {68–78}
}

@article{bohn2019representer,
    author = {Bohn, Bastian and Rieger, Christian and Griebel, Michael},
    title = {A Representer Theorem for Deep Kernel Learning},
    year = {2019},
    publisher = {JMLR.org},
    volume = {20},
    number = {1},
    journal = {Journal of Machine Learning Research},
    pages = {2302–2333},
    numpages = {32},
    doi = {https://dl.acm.org/doi/10.5555/3322706.3362005},
}

@article{harris2020array,
    title = {Array programming with {NumPy}},
    author = {Charles R. Harris and K. Jarrod Millman and St{\'{e}}fan J.
                 van der Walt and Ralf Gommers and Pauli Virtanen and David
                 Cournapeau and Eric Wieser and Julian Taylor and Sebastian
                 Berg and Nathaniel J. Smith and Robert Kern and Matti Picus
                 and Stephan Hoyer and Marten H. van Kerkwijk and Matthew
                 Brett and Allan Haldane and Jaime Fern{\'{a}}ndez del
                 R{\'{i}}o and Mark Wiebe and Pearu Peterson and Pierre
                 G{\'{e}}rard-Marchant and Kevin Sheppard and Tyler Reddy and
                 Warren Weckesser and Hameer Abbasi and Christoph Gohlke and
                 Travis E. Oliphant},
    year = {2020},
    journal = {Nature},
    volume = {585},
    number = {7825},
    pages = {357--362},
    doi = {10.1038/s41586-020-2649-2},
    publisher = {Springer Science and Business Media {LLC}},
}

@article{virtanen2020SciPy,
    author = {Virtanen, Pauli and Gommers, Ralf and Oliphant, Travis E. and
            Haberland, Matt and Reddy, Tyler and Cournapeau, David and
            Burovski, Evgeni and Peterson, Pearu and Weckesser, Warren and
            Bright, Jonathan and {van der Walt}, St{\'e}fan J. and
            Brett, Matthew and Wilson, Joshua and Millman, K. Jarrod and
            Mayorov, Nikolay and Nelson, Andrew R. J. and Jones, Eric and
            Kern, Robert and Larson, Eric and Carey, C J and
            Polat, {\.I}lhan and Feng, Yu and Moore, Eric W. and
            {VanderPlas}, Jake and Laxalde, Denis and Perktold, Josef and
            Cimrman, Robert and Henriksen, Ian and Quintero, E. A. and
            Harris, Charles R. and Archibald, Anne M. and
            Ribeiro, Ant{\^o}nio H. and Pedregosa, Fabian and
            {van Mulbregt}, Paul and {SciPy 1.0 Contributors}},
    title = {{{SciPy} 1.0: Fundamental Algorithms for Scientific Computing in Python}},
    journal = {Nature Methods},
    year = {2020},
    volume = {17},
    pages = {261--272},
    doi = {10.1038/s41592-019-0686-2},
}

@article{milk2016pyMOR,
    author    = {Milk, Ren\'{e} and Rave, Stephan and Schindler, Felix},
    title     = {{pyMOR} -- Generic Algorithms and Interfaces for Model Order Reduction},
    journal   = {SIAM Journal of Scientific Computing},
    year      = {2016},
    volume    = {38},
    number    = {5},
    pages     = {S194--S216},
    doi       = {10.1137/15m1026614},
    publisher = {Society for Industrial {\&} Applied Mathematics ({SIAM})},
}

@incollection{paszke2019PyTorch,
    author    = {Paszke, Adam and Gross, Sam and Massa, Francisco and Lerer, Adam and Bradbury, James and Chanan, Gregory and Killeen, Trevor and Lin, Zeming and Gimelshein, Natalia and Antiga, Luca and Desmaison, Alban and Kopf, Andreas and Yang, Edward and DeVito, Zachary and Raison, Martin and Tejani, Alykhan and Chilamkurthy, Sasank and Steiner, Benoit and Fang, Lu and Bai, Junjie and Chintala, Soumith},
    title     = {PyTorch: An Imperative Style, High-Performance Deep Learning Library},
    booktitle = {Advances in Neural Information Processing Systems 32},
    publisher = {Curran Associates, Inc.},
    year      = {2019},
    editor    = {H. Wallach and H. Larochelle and A. Beygelzimer and F. d' Alch\'{e}-Buc and E. Fox and R. Garnett},
    pages     = {8024--8035},
    doi = {https://dl.acm.org/doi/10.5555/3454287.3455008},
}

@article{petersen2018optimal,
    author   = {Philipp Petersen and Felix Voigtlaender},
    title    = {Optimal approximation of piecewise smooth functions using deep {ReLU} neural networks},
    journal  = {Neural Networks},
    year     = {2018},
    volume   = {108},
    pages    = {296 - 330},
    doi      = {10.1016/j.neunet.2018.08.019},
}

@article{rumelhart1986learning,
    author    = {Rumelhart, David E. and Hinton, Geoffrey E. and Williams, Ronald J.},
    title     = {{Learning representations by back-propagating errors}},
    journal   = {Nature},
    year      = {1986},
    volume    = {323},
    number    = {6088},
    pages     = {533--536},
    doi       = {10.1038/323533a0},
}

@article{liu1989limited,
    author  = {Dong C. Liu and Jorge Nocedal},
    title   = {On the Limited Memory {BFGS} Method for Large Scale Optimization},
    journal = {Mathematical Programming},
    year    = {1989},
    volume  = {45},
    pages   = {503--528},
    doi     = {10.1007/bf01589116},
}

@inproceedings{prechelt1997early,
    author    = {Lutz Prechelt},
    title     = {Early Stopping - but when?},
    booktitle = {Neural Networks: Tricks of the Trade, volume 1524 of LNCS, chapter 2},
    year      = {1997},
    pages     = {55--69},
    publisher = {Springer-Verlag},
    doi       = {10.1007/978-3-642-35289-8_5},
}

@article{scikit-learn,
    title={Scikit-learn: Machine Learning in {P}ython},
    author={Pedregosa, Fabian and Varoquaux, Ga\"{e}l and Gramfort, Alexandre and Michel, Vincent and Thirion, Bertrand and Grisel, Olivier and Blondel, Mathieu and Prettenhofer, Peter and Weiss, Ron and Dubourg, Vincent and Vanderplas, Jake and Passos, Alexandre and Cournapeau, David and Brucher, Matthieu and Perrot, Matthieu and Duchesnay, \'{E}douard},
    journal={Journal of Machine Learning Research},
    volume={12},
    pages={2825--2830},
    year={2011},
    doi={https://dl.acm.org/doi/10.5555/1953048.2078195},
}

@article {guo2018reduced,
    AUTHOR = {Guo, Mengwu and Hesthaven, Jan S.},
    TITLE = {Reduced order modeling for nonlinear structural analysis using {G}aussian process regression},
    JOURNAL = {Computer Methods in Applied Mechanics and Engineering},
    VOLUME = {341},
    YEAR = {2018},
    PAGES = {807--826},
    MRCLASS = {65N30 (65N99)},
    MRNUMBER = {3845646},
    DOI = {10.1016/j.cma.2018.07.017},
}

@article{lecun2015deep,
    author={LeCun, Yann and Bengio, Yoshua and Hinton, Geoffrey},
    title={Deep learning},
    journal={Nature},
    year={2015},
    day={01},
    volume={521},
    number={7553},
    pages={436-444},
    doi={10.1038/nature14539},
}

@article{crank1947practical,
    title={A practical method for numerical evaluation of solutions of partial differential equations of the heat-conduction type},
    volume={43},
    DOI={10.1017/S0305004100023197},
    number={1},
    journal={Mathematical Proceedings of the Cambridge Philosophical Society},
    publisher={Cambridge University Press},
    author={Crank, John and Nicolson, Phyllis},
    year={1947},
    pages={50–67}
}

@book{golub1996matrix,
    title = {Matrix Computations},
    author = {Golub, Gene H. and Van Loan, Charles F.},
    publisher = {The Johns Hopkins University Press},
    year = {1996},
    edition = {Third Edition},
    doi = {https://dl.acm.org/doi/10.5555/248979},
    isbn = {0801854148},
}

@inbook{graessle2021model,
    title = {Model order reduction by proper orthogonal decomposition},
    booktitle = {Volume 2 Snapshot-Based Methods and Algorithms},
    booktitle = {Volume 2: Snapshot-Based Methods and Algorithms},
    author = {Carmen Gräßle and Michael Hinze and Stefan Volkwein},
    editor = {Peter Benner and Stefano Grivet-Talocia and Alfio Quarteroni and Gianluigi Rozza and Wil Schilders and Luís Miguel Silveira},
    publisher = {De Gruyter},
    address = {Berlin, Boston},
    pages = {47--96},
    doi = {10.1515/9783110671490-002},
    isbn = {9783110671490},
    year = {2021},
}

@book{rasmussen2006gaussian,
    author = {Rasmussen, Carl Edward and Williams, Christopher K. I.},
    isbn = {026218253X},
    pages = {I-XVIII, 1-248},
    publisher = {MIT Press},
    series = {Adaptive computation and machine learning},
    title = {Gaussian processes for machine learning},
    year = {2006},
    doi = {10.7551/mitpress/3206.001.0001},
}

@inproceedings{mayer2019stochastic,
    author={Mayer, Jana and Dolgov, Maxim and Stickling, Tobias and Özgen, Selim and Rosenthal, Florian and Hanebeck, Uwe D.},
    booktitle={2019 American Control Conference (ACC)},
    title={Stochastic Optimal Control Using Gaussian Process Regression over Probability Distributions},
    year={2019},
    volume={},
    number={},
    pages={4847-4853},
    doi={10.23919/ACC.2019.8814658}
}

@article{benner2015survey,
    doi = {10.1137/130932715},
    year = {2015},
    publisher = {Society for Industrial {\&} Applied Mathematics ({SIAM})},
    volume = {57},
    number = {4},
    pages = {483--531},
    author = {Peter Benner and Serkan Gugercin and Karen Willcox},
    title = {A Survey of Projection-Based Model Reduction Methods for Parametric Dynamical Systems},
    journal = {{SIAM} Review}
}

@article{kanagawa2018gaussian,
    title={Gaussian Processes and Kernel Methods: A Review on Connections and Equivalences}, 
    author={Motonobu Kanagawa and Philipp Hennig and Dino Sejdinovic and Bharath K Sriperumbudur},
    year={2018},
    eprint={1807.02582},
    archivePrefix={arXiv},
    primaryClass={stat.ML}
}

@article{haasdonk2011efficient,
    doi = {10.1080/13873954.2010.514703},
    year = {2011},
    publisher = {Informa {UK} Limited},
    volume = {17},
    number = {2},
    pages = {145--161},
    author = {Bernard Haasdonk and Mario Ohlberger},
    title = {Efficient reduced models and a posteriori error estimation for parametrized dynamical systems by offline/online decomposition},
    journal = {Mathematical and Computer Modelling of Dynamical Systems}
}

@article{aronszajn1950theory,
    doi = {10.1090/s0002-9947-1950-0051437-7},
    year = {1950},
    publisher = {American Mathematical Society ({AMS})},
    volume = {68},
    number = {3},
    pages = {337--404},
    author = {Nachman Aronszajn},
    title = {Theory of reproducing kernels},
    journal = {Transactions of the American Mathematical Society}
}

@book{hale1980ordinary,
    year = {1980},
    author = {Jack K. Hale},
    title = {Ordinary Differential Equations},
    edition = {2},
    publisher = {Robert E. Krieger Publishing Company},
    isbn = {0-89874-011-8}
}

@article{greif2019decay,
	title = {Decay of the {Kolmogorov} {N}-width for wave problems},
	journal = {Applied Mathematics Letters},
	volume = {96},
	pages = {216-222},
	year = {2019},
	issn = {0893-9659},
	doi = {10.1016/j.aml.2019.05.013},
	author = {Constantin Greif and Karsten Urban}
}

@misc{sourcecode,
    doi = {10.5281/ZENODO.8188417},
    url = {https://zenodo.org/record/8188417},
    author = {Kleikamp, Hendrik and Lazar, Martin and Molinari, Cesare},
    title = {Software for ``Be greedy and learn: efficient and certified algorithms for parametrized optimal control problems''},
    publisher = {Zenodo},
    year = {2023},
    copyright = {Open Access}
}

@Article{keil2022adaptive,
    author={Keil, Tim and Kleikamp, Hendrik and Lorentzen, Rolf J. and Oguntola, Micheal B. and Ohlberger, Mario},
    title={Adaptive machine learning-based surrogate modeling to accelerate PDE-constrained optimization in enhanced oil recovery},
    journal={Advances in Computational Mathematics},
    year={2022},
    volume={48},
    number={6},
    pages={73},
    doi={10.1007/s10444-022-09981-z},
}
    
	\newpage

    \section*{Statements and Declarations}

    \paragraph{Acknowledgments}
    \begin{itemize}
        \item H.~Kleikamp acknowledges funding by the Deutsche Forschungsgemeinschaft (DFG, German Research Foundation) under Germany's Excellence Strategy EXC 2044 –390685587, Mathematics Münster: Dynamics–Geometry–Structure.
        \item M.~Lazar was supported by the Alexander von Humboldt-Professorship program and through the project ``Uncertain Data in Control Systems with PDEs'' funded by German Academic exchange Service DAAD and Croatian Ministry of Science and Education.
        \item C.~Molinari is part of the Indam group ``Gruppo Nazionale per l'Analisi Matematica, la Probabilit\`a e le loro applicazioni'' and acknowledges the support of the AFOSR project FA8655-22-1-703.
    \end{itemize}

    \paragraph{Competing Interests}
    The authors have no relevant financial or non-financial interests to disclose.

    \paragraph{Code Availability} The source code used to carry out the numerical experiments presented in this contribution can be found in~\cite{sourcecode}.

    \appendix

\section{Proof of~\texorpdfstring{\Cref{thm:optimality-system}}{Theorem~\ref{thm:optimality-system}}}\label{app:proof-optimality-system}
For simplicity, we omit the dependence of the involved quantities on the parameter~$\mu\in\params$.
\begin{proof}
    Let~$u^*\in G$ denote an optimal control, $x^*\in H$ the corresponding state trajectory. We are going to prove that the first variation of~$\mathcal{J}$ vanishes if~$x^*$, $\varphi^*$ and~$u^*$ solve the boundary value problem stated above. To this end, let~$v\in G$, and consider the perturbation~$u\in G$ of~$u^*$ defined as
    \[
        u(t) \coloneqq u^*(t) + \varepsilon v(t)
    \]
    for~$\varepsilon\in\setR$. Hence, the state equation for the control~$u$ reads
    \[
        \dot{x}(t) = Ax(t)+Bu^*(t)+\varepsilon Bv(t)\qquad\text{for }t\in[0,T].
    \]
    The solution is explicitly given by
    \begin{align*}
        x(t) &= e^{At}x^0 + \int\limits_0^t e^{A(t-s)}\big(Bu^*(s)+\varepsilon Bv(s)\big)\d{s} \\
        &= x^*(t) + \varepsilon\int\limits_0^t e^{A(t-s)}Bv(s)\d{s} \\
        &= x^*(t) + \varepsilon z(t)
    \end{align*}
    for~$z\in H$ defined as~$z(t)\coloneqq\int_0^t e^{A(t-s)}Bv(s)\d{s}$. It holds that~$z$ satisfies the ordinary differential equation
    \[
        \dot{z}(t) = Az(t)+Bv(t),\qquad z(0)=0.
    \]
    Introducing the adjoint state~$\varphi\in H$ and the Hamiltonian function~$\mathcal{H}\colon\X\times\U\times\X\to\setR$ given as
    \[
        \mathcal{H}(x(t),u(t),\varphi(t)) = \frac{1}{2} \langle u(t), Ru(t) \rangle + \langle \varphi(t), \big(Ax(t)+Bu(t)\big)\rangle ,
    \]
    we can rewrite the functional~$\mathcal{J}$ as
    \[
        \mathcal{J}(u) = \frac{1}{2}\langle x(T)-x^T, M\left(x(T)-x^T\right)\rangle + \int\limits_0^T \mathcal{H}(x(t),u(t),\varphi(t)) - \langle \varphi(t), \dot{x}(t)\rangle \d{t},
    \]
    since~$x\in H$ solves the state equation. This holds for any adjoint state~$\varphi\in H$. Similarly, for the optimal control~$u^*$ and corresponding state trajectory~$x^*$ it holds
    \[
        \mathcal{J}(u^*) = \frac{1}{2}\langle x^*(T)-x^T ,M\left(x^*(T)-x^T\right)\rangle + \int\limits_0^T \mathcal{H}(x^*(t),u^*(t),\varphi(t)) - \langle \varphi(t), \dot{x}^*(t)\rangle \d{t}.
    \]
    We now consider the difference~$\mathcal{J}(u)-\mathcal{J}(u^*)$, which is given as
    \begin{equation}\label{equ:difference-cost-functional}
        \begin{aligned}
            \mathcal{J}(u)-\mathcal{J}(u^*) &= \frac{1}{2}\left[\langle x(T)-x^T, M\left(x(T)-x^T\right)\rangle - \langle x^*(T)-x^T, M\left(x^*(T)-x^T\right)\rangle \right] \\
            & \hphantom{==}+ \int\limits_0^T \mathcal{H}(x(t),u(t),\varphi(t))-\mathcal{H}(x^*(t),u^*(t),\varphi(t))\d{t} \\
            & \hphantom{==}+ \int\limits_0^T \langle \varphi(t), \dot{x}^*(t)-\dot{x}(t)\rangle \d{t}.
        \end{aligned}
    \end{equation}
    We obtain for the first term in~\cref{equ:difference-cost-functional} the identity
    \begin{align*}
        & \frac{1}{2}\left[\langle x(T)-x^T, M\big(x(T)-x^T\big) \rangle-\langle x^*(T)-x^T, M\big(x^*(T)-x^T\big)\rangle \right] \\
        &= \frac{1}{2}\left[\langle x^*(T)+\varepsilon z(T)-x^T, M\big(x^*(T)+\varepsilon z(T)-x^T\big) \rangle-\langle x^*(T)-x^T, M\big(x^*(T)-x^T\big)\rangle \right] \\
        &= \varepsilon \langle z(T), M\big(x^*(T)-x^T\big)\rangle +\mathcal{O}(\varepsilon^2),
    \end{align*}
    where we used that~$M$ is self-adjoint. For the difference of the Hamiltonians in the second term in~\cref{equ:difference-cost-functional} it holds
    \begin{align*}
        & \mathcal{H}(x(t),u(t),\varphi(t))-\mathcal{H}(x^*(t),u^*(t),\varphi(t)) \\
        &= \frac{1}{2} \langle u(t), Ru(t)\rangle +\langle \varphi(t),Ax(t)+Bu(t)\rangle - \frac{1}{2} \langle u^*(t), Ru^*(t)\rangle - \langle \varphi(t), Ax^*(t)+Bu^*(t)\rangle \\
        &= \frac{1}{2}\langle u^*(t) + \varepsilon v(t), R\big(u^*(t) + \varepsilon v(t)\big)\rangle +\langle \varphi(t), Ax^*(t)+\varepsilon Az(t)+Bu^*(t)+\varepsilon Bv(t)\rangle\\
        & \hphantom{==}- \frac{1}{2}\langle u^*(t), Ru^*(t)\rangle - \langle\varphi(t), Ax^*(t)+Bu^*(t)\rangle \\
        &= \varepsilon \langle u^*(t), Rv(t)\rangle+\langle \varphi(t), \varepsilon Az(t)+\varepsilon Bv(t)\rangle + \mathcal{O}(\varepsilon^2) \\
        &= \varepsilon\Big[\langle u^*(t), Rv(t)\rangle+\langle \varphi(t), Az(t)\rangle +\langle \varphi(t), Bv(t)\rangle \Big] + \mathcal{O}(\varepsilon^2) \\
        &= \varepsilon\Big[\langle Ru^*(t)+B^*\varphi(t), v(t)\rangle  + \langle \varphi(t), Az(t)\rangle\Big] + \mathcal{O}(\varepsilon^2)
    \end{align*}
    for all~$t\in[0,T]$, where we used that~$R$ is a self-adjoint operator. Further, recall that it holds~$x(t)=x^*(t)+\varepsilon z(t)$ and therefore~$\dot{x}^*(t)-\dot{x}(t)=-\varepsilon\dot{z}(t)$ for all~$t\in[0,T]$. For the last term in~\cref{equ:difference-cost-functional} we thus obtain via integration by parts
    \begin{align*}
        \int\limits_0^T \langle \varphi(t), \dot{x}^*(t)-\dot{x}(t)\rangle \d{t} &= -\varepsilon\int\limits_0^T \langle \varphi(t), \dot{z}(t)\rangle\,\d{t} \\
        &= \big[-\varepsilon\langle\varphi(t), z(t)\rangle\big]_0^T + \varepsilon\int\limits_0^T \langle \dot{\varphi}(t), z(t) \langle \d{t} \\
        &= -\varepsilon\langle \varphi(T), z(T)\rangle + \varepsilon\int\limits_0^T \langle \dot{\varphi}(t), z(t)\rangle \d{t},
    \end{align*}
    where we used that it holds~$z(0)=0$. Bringing everything together gives
    \begin{align*}
        \mathcal{J}(u)-\mathcal{J}(u^*) &= \varepsilon\left[\langle z(T), M\big(x^*(T)-x^T\big)\rangle \vphantom{\int\limits_0^T} + \int\limits_0^T \langle Ru^*(t)+B^*\varphi(t), v(t) \rangle + \langle\varphi(t), Az(t)\rangle \d{t} \right. \\
        & \hphantom{=\varepsilon=}\left. + \int\limits_0^T \langle \dot{\varphi}(t), z(t) \rangle \d{t} - \langle \varphi(T), z(T)\rangle \right] + \mathcal{O}(\varepsilon^2) \\
        &= \varepsilon\left[ \langle z(T), M(x^*(T)-x^T-\varphi(T)\rangle \vphantom{\int\limits_0^T} + \int\limits_0^T \langle Ru^*(t)+B^*\varphi(t), v(t) \rangle \d{t} \right. \\
        & \hphantom{=\varepsilon=}\left. + \int\limits_0^T \langle \dot{\varphi}(t)+A^*\varphi(t), z(t)\rangle \d{t}\right] + \mathcal{O}(\varepsilon^2).
    \end{align*}
    Since~$u^*$ is assumed to be an optimal control, it has to hold for all~$v\in \U$ that
    \begin{align*}
        0 &= \lim\limits_{\varepsilon\to 0}\frac{\mathcal{J}(u^*+\varepsilon v)-\mathcal{J}(u^*)}{\varepsilon} \\
        &= \lim\limits_{\varepsilon\to 0}\frac{\mathcal{J}(u)-\mathcal{J}(u^*)}{\varepsilon} \\
        &= \langle z(T), M(x^*(T)-x^T)-\varphi(T)\rangle \vphantom{\int\limits_0^T} \\
        & \hphantom{==}+ \int\limits_0^T \langle Ru^*(t)+B^*\varphi(t), v(t) \rangle \d{t} + \int\limits_0^T \langle \dot{\varphi}(t)+A^*\varphi(t), z(t)\rangle \d{t}.
    \end{align*}
    This yields the claimed necessary conditions for~$u^*$, $x^*$ and~$\varphi^*$.
\end{proof}
\end{document}